\documentclass[a4paper,10pt]{article}

\usepackage{moyennabilite}

\title{Twisted Patterson-Sullivan measures  and applications to amenability and coverings 
}
\author{R\'emi Coulon, Rhiannon Dougall, Barbara Schapira, Samuel Tapie}
\date{  \today}

\begin{document}
\selectlanguage{english}
\markboth{Twisted Patterson-Sullivan}{Coulon - Dougall - Schapira - Tapie}

\maketitle

\begin{abstract}
	Let $\Gamma'<\Gamma$ be two discrete groups acting properly by isometries 
on a Gromov-hyperbolic space $X$.
	We prove that their critical exponents coincide if and only if 
$\Gamma'$ is co-amenable in $\Gamma$, under the assumption
	that the action of $\Gamma$ on $X$ is  strongly positively recurrent, 
i.e. has a growth gap at infinity.
	This generalizes all previously known results on this question, 
which required either $X$ to be the real hyperbolic space and $\Gamma$ geometrically finite, 
or $X$ Gromov hyperbolic and $\Gamma$ cocompact.
	This result is optimal: we provide several  counterexamples 
when the action is not strongly positively recurrent.
\end{abstract}

\noindent
{\footnotesize 
\textbf{Keywords.} 
Twisted Patterson-Sullivan measure, critical exponents, unitary representations, amenability, strongly positively recurrent actions, growth gap at infinity, Gromov hyperbolic spaces \\
\textbf{Primary MSC Classification.}   
20F67, 
20F65, 
57S30, 
37A35, 
37A25, 
37A15, 
37D40. 
\\
\textbf{Secondary MSC Classification.}  	
06B99, 
20C99, 
37D35. 
}
\tableofcontents

\medskip

%
\section{Introduction}
%

%
\subsection{Main result}
%

Let $(X,d)$ be a proper metric space, $o\in X$ be some fixed origin, 
and $\Gamma$ be a discrete group acting properly by isometries on $X$. 
The \emph{critical exponent} of $\Gamma$   (for its action on $X$)  is
\begin{equation*}
h_\Gamma
=
h_\Gamma(X) 
= 
\limsup_{r\to +\infty} \frac 1r \ln \card{\set{\gamma\in \Gamma}{\dist o{\gamma o} \leq r}}.
\end{equation*}
Of course, any subgroup  $\Gamma'<\Gamma$ satisfies $h_{\Gamma'}\leq h_\Gamma$.
This paper is devoted to the study of the equality case.
\begin{center}
\emph{When do we have $h_{\Gamma'} = h_\Gamma$?}
\end{center}

We are particularly interested in the case where $X$ is Gromov hyperbolic.
The answer to this question is intimately related to the \emph{co-amenability} of $\Gamma'$ in $\Gamma$, as was first independently shown by Grigorchuk \cite{Grigorchuk:1980wx}, Cohen \cite{Cohen:1982gt} and Brooks \cite{Brooks:1981jp}. Saying that $\Gamma'$ is co-amenable in $\Gamma$ is a natural way to generalise the fact that the quotient $\Gamma / \Gamma'$ is amenable when $\Gamma'$ is not a normal subgroup of $\Gamma$, see \autoref{def: amenability}.
Our main theorem widely extends all  previously known results on this question. 
It holds under the assumption that the action of $\Gamma$  has a \emph{growth gap at infinity}, i.e. some critical exponent at infinity, representing the growth of $\Gamma$ far from the orbit of any compact set, is strictly smaller than $h_\Gamma$. 
We also call such actions \emph{strongly positively recurrent}.  
See below for the rigorous definition. 
This assumption is much weaker than more usual assumptions such as convex-cocompactness or geometrical finiteness, as shown in \cite{Schapira:2018ti}. 

\begin{theo}
	\label{res: main}
	Let $X$ be a proper hyperbolic geodesic space.
	Let $\Gamma$ be a group acting  properly by isometries on $X$, and $\Gamma'$ a subgroup of $\Gamma$.
	Assume that the action of $\Gamma$ is strongly positively recurrent.
	The following are equivalent.
	\begin{enumerate}
		\item  $h_{\Gamma'} = h_\Gamma$
		\item The subgroup $\Gamma'$ is co-amenable in $\Gamma$.
	\end{enumerate}
\end{theo}

As will be shown in \autoref{sec:counterexample}, this result is optimal. 
None of the assumptions can be weaken without hitting numerous counterexamples.
Our main theorem closes the above question for group actions on Gromov hyperbolic spaces. 
Beyond this, we believe that our main tool, the \emph{twisted Patterson-Sullivan measure}, is at least as important as the result, and  should  have various other applications in the future.

\autoref{res: main} has also a quantified version (\autoref{res: growth vs almost inv vec - quantified}) which leads to the following wide generalisation of Corlette's growth rigidity result \cite{Corlette:1990br}, see also \cite{Dougall:2017va,Coulon:2017vz}.

\begin{theo}
\label{res: main T}
	Let $X$ be a proper hyperbolic geodesic space.
	Let $\Gamma$ be a group with Kazhdan's property (T) acting properly by isometries on $X$.
	Assume that the action of $\Gamma$ is strongly positively recurrent.
	There exists $\varepsilon \in \R_+^*$ such that for every subgroup $\Gamma'$ of $\Gamma$, either $h_{\Gamma'} \leq h_\Gamma - \varepsilon$ or $\Gamma'$ is a finite index subgroup of $\Gamma$.
\end{theo}

We  now  give a brief  historical background on this question, introduce the notion of strongly positively recurrent action, and sketch the proof  of \autoref{res: main T}.

%
\subsection{Historical background}
%

The first relations between critical exponents and amenability appeared independently in the eighties,  in the work of Brooks \cite{Brooks:1981jp,Brooks:1985ky}, in the context of hyperbolic manifolds, and Grigorchuk \cite{Grigorchuk:1980wx}, Cohen \cite{Cohen:1982gt} in a combinatorial setting.

Let $\Gamma$ be a finitely generated free group acting on its Cayley graph $X$, with respect to a free basis.
Given any normal subgroup $\Gamma'$ of $\Gamma$, Grigorchuk and Cohen  relate by a delicate explicit computation  the critical exponent of $\Gamma'$ (also called \emph{co-growth} of $\Gamma/\Gamma'$) to the spectral radius of the random walk on $\Gamma/\Gamma'$.
 Combined with Kesten's amenability criterion, they obtain the following statement.

\begin{theo}[Grigorchuk \cite{Grigorchuk:1980wx}, Cohen \cite{Cohen:1982gt}]
\label{res: cohen-grigorchuk}
	Let $\Gamma$ be a finitely generated free group and $X$ its Cayley graph with respect to a free basis.
	For every normal subgroup $\Gamma' $ of $ \Gamma$, the quotient $\Gamma/\Gamma'$ is amenable if and only if $h_{\Gamma'} = h_\Gamma$.
\end{theo}

At the same period, Brooks showed the following statement  using the spectral properties of the Laplace-Beltrami operator.

\begin{theo}[Brooks, \cite{Brooks:1985ky}]
\label{theo: brooks}
Let $n \in \N$ and $M = \mathbb H^{n+1}/\Gamma$ be a \emph{convex-cocompact hyperbolic} manifold with $h_\Gamma> n/2$. 
Then for every normal subgroup $\Gamma' $ of $ \Gamma$, the quotient $\Gamma/\Gamma'$ is amenable if and only if $h_{\Gamma'} = h_\Gamma$.
\end{theo}

Let us discuss briefly the strategy behind this last  result.
Recall that a negatively curved manifold is convex-cocompact if all closed geodesics are included in a given compact set (or equivalently, if the geodesic flow has a compact nonwandering set). 
Brooks' approach actually starts in a much larger context.
Given a Riemannian manifold $M$  whose Laplacian satisfies a \emph{spectral gap} condition, he showed that for every normal covering $M'$ of $M$ 
the quotient $\pi_1(M) /\pi_1(M')$ is amenable if and only if 
the bottom spectra of their respective Laplace-Beltrami operators satisfy $\lambda_0(M) = \lambda_0(M')$.
If $M =  \mathbb H^{n+1}/\Gamma$ is a hyperbolic manifold with $h_\Gamma > n/2$, 
then Sullivan's formula relates $\lambda_0(M)$ to $h_\Gamma$ \cite{Sullivan:1987bt}.
Moreover,  Brooks' spectral condition is satisfied for convex-cocompact hyperbolic manifolds with
 $h_\Gamma > n/2$, which gives \autoref{theo: brooks}.

We will not define  this spectral gap condition for the Riemannian Laplacian here -- see  \cite[Section~1]{Brooks:1985ky} -- but it is exactly the spectral analog to the \emph{growth gap at infinity} (or \emph{strongly positive recurrence}) which we will introduce below for group actions, under which our main theorems are valid.  

The assumption $h_\Gamma > n/2$ is specific to this approach and cannot be removed 
as long as  one uses Laplace spectrum.

\medskip
Sullivan's formula relating the bottom of the spectrum of the Laplacian with critical exponents has been extended by Corlette-Iozzi  \cite{Corlette:1999bs} to all other locally symmetric hyperbolic manifolds. Therefore Brooks method extends verbatim to these exotic hyperbolic manifolds.
Note also that  Brooks's result can be extended when $\Gamma'$ is not normal in $\Gamma$.  This  can be seen following the alternative proof of Brooks' Theorem given in \cite{Roblin:2013bh}.
 
 Using Patterson-Sullivan theory, Roblin in \cite{Roblin:2005fn} 
is the first to prove the so-called ``easy direction'' in a much wider context. 
Namely, if  $\Gamma$ is a discrete group of isometries acting on a CAT$(-1)$ space $X$ 
and $\Gamma'$ is a normal subgroup of $\Gamma$ such that $\Gamma/\Gamma'$ is amenable, 
then $h_\Gamma=h_{\Gamma'}$. 
His proof  extends easily to actions on Gromov hyperbolic spaces, 
but requires in a crucial way that $\Gamma'$ be normal in $\Gamma$.

\medskip
The reciprocal statement was generalised by Stadlbauer in \cite{Stadlbauer:2013dg}, 
using a dynamical argument inspired by Kesten's work on random walks, 
see also Jaerisch \cite{Jaerisch:2014if}.
If $\Gamma$ is an \emph{essentially free} discrete group of isometries 
of $\mathbb H^{n+1}$, then for all normal subgroups $\Gamma'$ of $ \Gamma$, 
the quotient  $\Gamma/\Gamma'$ is amenable if and only if $h_{\Gamma'} = h_\Gamma$.
His method allows to remove the artificial assumption $h_\Gamma>n/2$, 
and to our knowledge it is the only published work to deal 
with certain specific non convex-cocompact manifolds (geometrically finite).

	Stadlbauer's arguments have been used later on by Dougall-Sharp in \cite{Dougall:2016cc} 
with a symbolic coding in order to extend the result to convex-cocompact manifolds 
with pinched negative curvature, when $\Gamma'$ is a normal subgroup of $ \Gamma$.  
	A generalization by Coulon, Dal'bo and Sambusetti in \cite{Coulon:2017vz} 
allows to deal with proper cocompact actions of $\Gamma$ on some  Gromov-hyperbolic spaces $X$, more precisely  CAT($-1$) spaces or the Cayley graph of $\Gamma$.
	Moreover the subgroup $\Gamma'$ need not be   normal in $\Gamma$.

%
\subsection{Strongly positively recurrent actions}
%

The notion of strongly positively recurrent action is crucial in our work. 
Let us  present  the definition and its origin. 
A detailed presentation can be found in  \autoref{sec: spr}. 
Let $(X,d)$ be a proper geodesic space and $\Gamma$
a group acting properly by isometries on $X$.
Given a compact subset $K$ of $X$,
we define $\Gamma_K$ as the set of elements $\gamma \in \Gamma$
for which there exists two points $x,y \in K$ and a geodesic $c \colon [a,b] \to X$ joining $x$ to $\gamma y$
such that $c \cap \Gamma \cdot K$ is contained in $K \cup \gamma K$.
The critical exponent $h_{\Gamma_K}$  of  $\Gamma_K $ is called the \emph{entropy outside  $K$}. 
The \emph{entropy at infinity} of $\Gamma$ is the quantity
\begin{equation*}
	h_\Gamma^\infty = \inf_K h_{\Gamma_K}
\end{equation*}
 The action of $\Gamma$ on $X$ has a \emph{growth gap at infinity} if $h_\Gamma^\infty < h_\Gamma$. We will say then that the action is \emph{strongly positively recurrent}. 
This notion which has both a dynamical and a geometric origin has been introduced independently in different contexts.

\paragraph{A dynamical origin.} 
Heuristically, a dynamical system is strongly positively recurrent (with respect to a constant potential) if its entropy at infinity is strictly smaller than its topological entropy, see for instance \cite{Sarig:2001bj}.
The terminology \emph{stably positively recurrent} has been first introduced in the context of Markov shifts over a countable alphabet by Gurevi\v c-Savchenko \cite{Gurevich:1998it}, and became   \emph{strongly positively recurrent} later in Sarig \cite{Sarig:2001bj}. 
This terminology, with the  notion of entropy at infinity, has been  used later on by several authors considering dynamical systems on a non-compact space, such as Ruette \cite{Ruette:2003fj},  Boyle, Buzzi and Gomez \cite{Boyle:2014jf}, or more recently Riquelme and Velozo \cite{Riquelme:2016wc,Velozo:2017vy}. 
We do not define here the entropy at infinity of a dynamical system, however for the geodesic flow of a non-compact negatively curved manifold it coincides with the quantity $h_\Gamma^\infty$ defined above \cite{Riquelme:2016wc,Velozo:2017vy,Schapira:2018ti}.

\paragraph{A geometric point of view.}
Dal'bo, Otal and Peign\'e in \cite{Dalbo:2000eh} introduced the terminology of \emph{parabolic gap} concerning geometrically finite groups $\Gamma$ of isometries of a negatively curved space $X$ whose parabolic subgroups $P$ all satisfy $h_P < h_\Gamma$. 
Extending the work of Dal'bo \emph{et al} \cite{DPPS11}, this was later generalized by Arzhantseva, Cashen and Tao \cite[Definition~1.6]{Arzhantseva:2015cl} to the 
so-called \emph{growth gap} property, which is exactly the  \emph{growth gap at infinity}  defined above. 
They showed that if the action of $\Gamma$ on $X$ has a growth gap at infinity and admits a contracting element, then $\Gamma$ is \emph{growth tight} (see \cite{Arzhantseva:2015cl} for a definition). 
This notion has also  been studied by Yang \cite[Definition~1.4]{Yang:2016wa} under the name \emph{statistically convex-cocompact action}.
 His terminology comes from the following intuition. 
Given $r \in \R_+$, the $\Gamma$-orbit of a point $o \in X$ is in general not $r$-quasi-convex.
If $K$ stands for the closed ball $B(o,r)$, then $\Gamma_K$ is exactly the set of elements $\gamma \in \Gamma$ violating the definition of quasi-convexity.
The assumption $h_\Gamma^\infty < h_\Gamma$ states that most elements of $\Gamma$ behave as in a convex-cocompact setting.

\paragraph{Combining dynamical and geometric approaches.}
The paper \cite{Schapira:2018ti} by Schapira and Tapie introduced strongly positively recurrent actions in order to study the geodesic flow of  negatively curved manifolds (independently but identically to Arzhantseva et al. and Yang), and provided several new examples. It was both inspired by the aforementioned dynamical works of Sarig and Buzzi and the geometric approach of Dal'bo, Otal and Peign\'e.

In the present work, we combine intuitions from dynamical systems -- especially many tools used for the ergodic study of the geodesic flow on non-compact negatively curved manifolds -- and geometric group theory to get our main result.

%
\subsection{Outline of the proofs}
%

Let us give a brief account on the proofs, and the main novelties of this paper.
\autoref{res: main} is the combination of two results:
\begin{enumerate}
	\item the so-called \emph{``easy direction''}, i.e. showing that if $\Gamma'$ is a co-amenable subgroup of $\Gamma$, then $h_{\Gamma'} = h_\Gamma$;
	\item conversely, showing that for any subgroup  $\Gamma'$ of $\Gamma$, if $h_{\Gamma'} = h_\Gamma$ then $\Gamma'$ is co-amenable in $\Gamma$.
\end{enumerate}

The ``easy direction'', detailed in \autoref{res: easy dir - final statement}, is based on an explicit estimation of the spectral radius of some random walks on $\Gamma/\Gamma'$, as in \cite{Coulon:2017vz}.

The core of this paper is the other direction.
In the context of general Gromov hyperbolic spaces instead of negatively curved manifolds or CAT$(-1)$-spaces, and maybe even more problematic  when the action of $\Gamma$ is not cocompact, all the approaches described above fail.
Indeed, the approach via the spectrum of a Laplace-Beltrami operator seems specific to \emph{locally symmetric Riemannian manifolds with negative curvature} and might not be adapted in this more general setting. 
Moreover,  we are not aware of any coding of the geodesic flow which would allow to transpose Stadlbauer's work.
We develop therefore  a new strategy combining Patterson-Sullivan theory and representation theory.

Assume for simplicity here that $X$ is a proper CAT($-1$) space and 
let $\Gamma$ be a discrete group acting properly by isometries on $X$. 
Let $\Gamma'$ be a subgroup of $\Gamma$ and $\mathcal H = \ell^2(\Gamma/ \Gamma')$.
Then $\Gamma'$ is co-amenable in $\Gamma$ if and only if the corresponding unitary representation $\rho \colon \Gamma\to \mathcal U(\mathcal H)$ almost admits invariant vectors. 
Given $s>0$, we associate to this representation the following formal twisted Poincar\'e series
\begin{equation*}
	A(s) = \sum_{\gamma\in \Gamma} e^{- s d(o,\gamma o)}\rho(\gamma),
\end{equation*}
and show that there exists a \emph{critical exponent} $h_\rho$ such that for every $s>h_\rho$, $A(s)$ is a bounded operator of $\mathcal H$.
Moreover this exponent satisfies
\begin{equation}
\label{eq:CriticalIneq}
	h_{\Gamma'} \leq h_\rho \leq h_\Gamma,
\end{equation}
see \autoref{res: lower bound delta rho}.
By analogy with the standard Patterson-Sullivan measure, we associate to any $x \in X$, an operator-valued measure 
\begin{equation*}
	a^\rho_{x,s} = \frac 1{\norm{A(s)}} \sum_{\gamma\in \Gamma} e^{- s d(x,\gamma o)} {\rm Dirac}(\gamma o)\rho(\gamma).
\end{equation*}
When $s$ approaches $h_\rho$ from above, we are able, using an ultra-filter $\omega$ (see \autoref{sec: ultra-limits}) to let these measures ``converge'' to a measure $a^\rho_x$ supported on the boundary $\partial X$ of $X$ and taking its values in  the space of bounded operators $\mathcal B(\mathcal H_\omega)$ on a larger Hilbert space $\mathcal H_\omega$.  
We call it the \emph{twisted Patterson-Sullivan measure}.  

In \autoref{sec: twisted ps} we properly define and study this measure.
In particular, we show that it satisfies all the properties of the classical Patterson-Sullivan measures: $h_\rho$-conformality (\autoref{res: b - conformality}),  $\Gamma$-invariance twisted by the limit representation $\rho_\omega \colon \Gamma \to \mathcal U(\mathcal H_\omega)$  induced by $\rho$ (\autoref{res: b - equivariance}), Shadow Lemma (\autoref{res: b - shadow lemma}), etc.

The existence of a growth gap at infinity is used at a single but crucial place to prove that the measure $a^\rho_x$ gives full mass to the \emph{radial limit set} (\autoref{res: twisted PS charges radial limit set}). 
This apparently technical result allows to approximate the measure of any Borel set by measures of shadows.
Then, the proof of \autoref{res: main} becomes particularly simple.
Assume indeed that $h_{\Gamma'} = h_\Gamma$.
By (\ref{eq:CriticalIneq}) the classical and twisted Patterson-Sullivan measures have the same conformal dimension, namely $h_\Gamma= h_\rho$.
Using the Shadow Lemma we deduce that $a^\rho_x$ is absolutely continuous with respect to the standard Patterson-Sullivan measure (\autoref{res: twisted absolute continuity}).
Thanks to the ergodicity of Bowen-Margulis current  we prove that the corresponding ``Radon-Nikodym derivative'' -- which takes its values in $\mathcal B(\mathcal H_\omega)$ --
is essentially constant, equal to say $D \in \mathcal B(\mathcal H_\omega) \setminus \{0\}$.
The twisted equivariance of $a^\rho_x$ directly implies 
that the image of $D$ (which is non trivial) is contained in the subspace of $\rho_\omega$-invariant vectors.
It follows then from the construction of $\rho_\omega$ that the original representation $\rho$ almost has invariant vectors, i.e. $\Gamma'$ is co-amenable in $\Gamma$ (\autoref{sec: invariant vectors}).

\medskip
When  $X$ is a proper Gromov hyperbolic space, the above ideas work  exactly in the same way.
One just has to be careful that all measures are only quasi-conformal. 
However, this proof requires   an important ergodicity argument.  We use the fact that the Bowen-Margulis current  is ergodic for the diagonal action of $\Gamma$ on the double boundary $\partial^2X$.
This is  well-known when $X$ is a negatively curved Hadamard manifold, or even a CAT$(-1)$ space  and the action of $\Gamma$ is strongly positively recurrent \cite{Roblin:2003vz, Schapira:2018ti}.  
Bader and Furman proved that the statement also holds when $\Gamma$ acts cocompactly on a Gromov hyperbolic space \cite{Bader:2017te}.
Although the result is quite expected, it had not been written yet for a non-cocompact action on a Gromov hyperbolic space, such as strongly positively recurrent actions. 
As it  should  be useful to other people, we decided to expose this argument in the fullest possible generality.

More precisely, if $\Gamma$ is a discrete group acting properly by isometries on a Gromov-hyperbolic space $X$, using the  abstract  geodesic flow already studied in \cite{Bader:2017te}, we prove
a Hopf-Tsuji-Sullivan dichotomy (\autoref{res: hopf tsuji}): 
 the Bowen-Margulis current  on  the double boundary  $\partial^2X$ is ergodic with respect to the action of $\Gamma$ if and only if the geodesic flow is ergodic and conservative   (for the Bowen-Margulis measure), if and only if the usual Patterson-Sullivan measure gives full measure to the radial limit set. 
The desired ergodicity for a strongly positively recurrent action then directly follows from \autoref{res: PS charges radial limit set}.

For the sake of completeness, we also included a finiteness criterion for the Bowen-Margulis measure (\autoref{res: finite bm - thm}) inspired from \cite{Pit:2018eg}, which allows to deduce that the Bowen-Margulis measure is finite in the presence of a growth  gap at infinity (\autoref{res: finite bm - finiteness final}).

\paragraph{Outline of the paper.} 
We  recall basics on Gromov hyperbolic spaces and the definition of the classical Patterson-Sullivan measure in \autoref{sec: Hyperbolic geometry}. 
In  \autoref{sec: spr} we define and study strongly positively recurrent actions.
In \autoref{sec: BMCurrent}, we develop  the ergodic study of Patterson-Sullivan and Bowen-Margulis measures in the context of Gromov-hyperbolic spaces.
\autoref{sec: twisted ps} is devoted to the twisted Patterson-Sullivan measures.  
In \autoref{sec:groups} we introduce the notion of co-amenable subgroup and prove \autoref{res: main}  and   other applications of our method.
We conclude in \autoref{sec:comments} with some questions. 

%
\subsection*{Acknowledgements}
%

The authors thank the Centre Henri Lebesgue (Labex ANR-11-LABX-0020-01), the Universit\'e Bretagne Loire and Rennes Metropole for the financial support which made this work possible.
The first author acknowledges support from the Agence Nationale de la Recherche under Grant \emph{Dagger} ANR-16-CE40-0006-01.
The authors also thank Gilles Carron, Fran\c coise Dal'bo, S\'ebastien Gou\"ezel, Tatiana Nagnibeda, Andrea Sambusetti and Wenyuan Yang for many interesting discussions during the elaboration of this paper.
 
%
\section{Patterson-Sullivan measures in hyperbolic spaces}
%
\label{sec: Hyperbolic geometry}

%
\subsection{Gromov hyperbolic spaces}
%

We review a few important facts about hyperbolic spaces and their compactifications.
For more details we refer the reader to Gromov's original paper \cite{Gro87} or \cite{CooDelPap90,Ghys:1990ki}.

Let $(X,d)$ be a proper geodesic metric space.
We denote by $B(x,r)$ the \emph{closed} ball of radius $r$ centred at $x$.

\paragraph{The four point inequality.}
Given three points $x,y,z \in X$, the \emph{Gromov product} is defined by
\begin{equation*}
	\gro xyz = \frac 12 \left[ \dist xz + \dist yz - \dist xy \right].
\end{equation*}
Let $\delta \in \R_+$.
The space $X$ is \emph{$\delta$-hyperbolic} if for all $x,y,z,t \in X$,
we have
\begin{equation}	\label{eqn: four point inequality}
	\gro xzt \geq \min \left\{ \gro xyt , \gro yzt \right\} - \delta.
\end{equation}
It is said to be \emph{Gromov hyperbolic} if it is $\delta$-hyperbolic for some  $\delta\in \R_+$.
 Nevertheless, for simplicity we will always assume that $\delta > 0$. 

\paragraph{The boundary at infinity.}
Let $o$ be a base point of $X$.
A sequence $(x_n)$ of points of $X$ \emph{converges to infinity} if $\gro {x_n}{x_m}o$ tends to infinity as $n$ and $m$ approach to infinity.
The set $\mathcal S$ of such sequences is endowed with a binary relation defined as follows.
Two sequences $(x_n)$ and $(y_n)$ are related if
\begin{displaymath}
	\lim_{n \rightarrow + \infty} \gro {x_n}{y_n}o
	= + \infty.
\end{displaymath}
By (\ref{eqn: four point inequality}), this relation is  an equivalence relation.
The \emph{boundary at infinity} of $X$, denoted by $\partial X$, is the quotient of $\mathcal S$ by this relation.
A sequence $(x_n)$  in the class of
$\xi \in \partial X$ is said \emph{converging} to $\xi$. We write
\begin{displaymath}
	\lim_{n \rightarrow + \infty} x_n = \xi.
\end{displaymath}
The definition of $\partial X$ does not depend on the base point $o$.
As $X$ is proper and geodesic, the Gromov boundary coincides with the visual boundary of $X$ \cite[Chapitre~2]{CooDelPap90}.

\medskip
The Gromov product of three points can be extended to the boundary.
Let $x\in X$ and $y,z \in X \cup \partial X$.
Define $\gro yz x$ as the infimum
\begin{displaymath}
	\liminf_{n\rightarrow + \infty} \gro {y_n}{z_n}x
\end{displaymath}
where $(y_n)$ and $(z_n)$ run over all sequences
which converge to $y$ and $z$ respectively.
This definition coincides with the original one when $y,z \in X$.
By (\ref{eqn: four point inequality}),
for any two sequences $(y_n)$ and $(z_n)$ converging respectively to $\eta,\xi \in \partial X$ one has
\begin{equation*}
	\gro \eta\xi x
	\leq
	\liminf_{n \to \infty} \gro {y_n}{z_n}x
	\leq
	\limsup_{n \to \infty} \gro{y_n}{z_n}x
	\leq
	\gro \eta\xi x + 2\delta.
\end{equation*}
Two points $\xi$ and $\eta$ of $\partial X$ are equal
if and only if $\gro \xi \eta x = + \infty$.
Moreover, for every $t \in X$, for every $x,y,z \in X \cup \partial X$,
the four point inequality (\ref{eqn: four point inequality}) leads to
\begin{equation}
	\label{eqn: hyperbolicity condition with boundary}
	\gro xzt \geq \min\left\{\gro xyt, \gro yzt \right\} - \delta.
\end{equation}

The Gromov boundary is a metrizable compact space.
More precisely there exists a metric on $\partial X$ that we denote $\distV[\partial X]$ and two numbers $a_0 \in (0,1)$ and $\varepsilon_0\in \R_+$ such that for every $\eta, \xi \in \partial X$,
\begin{equation}
	\label{eqn: metric gromov bdry}
	\abs{\ln \dist[\partial X]\eta \xi + a_0 \gro \eta\xi o } \leq \varepsilon_0.
\end{equation}
See for instance \cite[Chapitre~11, Lemme~1.7]{CooDelPap90}.

\paragraph{Limit sets.}
Assume that $\Gamma$ is a group acting by isometries on $X$.
This action extends to an action by homeomorphisms on $\partial X$.
Given any subset $S$ of $\Gamma$, the \emph{limit set of $S$}, denoted by $\Lambda(S)$, 
is the intersection   $\overline{S x}\setminus S x$ of the closure 
of the orbit  $S x$ with  $\partial X$, for  some (hence any) point $x\in X$.

Let $K$ be a compact subset of $X$.
The \emph{$K$-radial limit set} of $\Gamma$, denoted by $\Lambda_{\rm{rad}}^K(\Gamma)$, 
is the set of  points $\xi \in \partial X$ for which there exists a geodesic ray 
$c \colon \R_+ \to X$ ending at $\xi$ whose image intersects infinitely many translates 
of $K$ by elements of $\Gamma$.
It is a $\Gamma$-invariant subset of $\Lambda(\Gamma)$. 
The \emph{radial limit set} is the increasing union
\begin{equation*}
	\Lambda_{\rm{rad}}(\Gamma)
	=
	\bigcup_{K \subset X} \Lambda_{\rm{rad}}^K(\Gamma).
\end{equation*}
If there is no ambiguity we will drop $\Gamma$ from all the notations.

\paragraph{Horocompactification.}
We denote by $\mathbf 1$ the constant function equal to $1$. Let $C(X)$ be the set of continuous functions from $X$ to $\R$
endowed with the topology of uniform convergence on every compact subset.
We denote by $C_*(X)$ its quotient  by the one-dimensional $\R\mathbf 1$
endowed with the quotient topology.
As $X$ is proper, $C_*(X)$ is compact.
Alternatively $C_*(X)$ can be seen as the space of continuous cocycles on $X$,
i.e. maps $b \colon X \times X \to \R$ such that
$b(x,z) = b(x,y) + b(y,z)$, for every $x,y,z \in X$.
These two realisations of $C_*(X)$ are canonically identified via the isomorphism sending a map $f \colon X \to \R$ to the cocycle $b \colon X \times X \to \R$ defined by $b(x,y) = f(x) - f(y)$.

Given $x \in X$, we write $d_x \colon X \to \R$ for the map
sending $y$ to $\dist xy$.
The map $x \to d_x$ induces a homeomorphism from $X$ onto its image.
The \emph{horocompactification} of $X$, denoted by $\bar{X}_h$
is the closure of $X$ in $C_*(X)$.
The \emph{horoboundary} $\partial_hX$ is defined as
$\partial_hX = \bar{X}_h \setminus X$.

\medskip
We extend  the Gromov product to $\bar{X}_h$ as follows.
Given $x\in X$ and $b,b' \in \partial_h X$, we set
\begin{equation}
	\label{eqn: gromov product horoboundary}
	\gro b{b'}x  = \frac 12 \sup_{z \in X} \left[ b(x,z) + b'(x,z) \right].
\end{equation}

\medskip
Let $\Gamma$ be a group acting by isometries on $X$.
This action induces an action of $\Gamma$ on $C_*(X)$   as follows.
For every cocycle $b \in C_*(X)$,
for every $\gamma \in \Gamma$, and all $(x,y)\in X^2$,
\begin{equation*}
	[\gamma \cdot  b] (x,y) = b(\gamma^{-1}x, \gamma^{-1}y).
\end{equation*}
The horoboundary $\partial_hX$ is invariant under this action.
Moreover, the action preserves the Gromov product
defined in (\ref{eqn: gromov product horoboundary}).

\paragraph{Comparison with the Gromov boundary.}
	Given a geodesic ray $\alpha \colon \R_+ \to X$ the \emph{Busemann cocycle along $\alpha$} is the map $b \colon X \times X \to \R$ defined by
	\begin{displaymath}
		b(x,y) = \lim_{t \to \infty} \left[\dist x{\alpha(t)} - \dist y{\alpha(t)}\right].
	\end{displaymath}
	It is an example of point in the horoboundary $\partial_h X$. Note that there are in general several geodesic rays ending at a given point of the Gromov boundary $\partial X$, which may induce distinct Busemann cocycles.

\begin{prop}[Coornaert-Papadopoulos {\cite[Proposition~3.3 and Corollary~3.8]{Coornaert:2001ff}}]
	\label{res: projection horofunction to boundary}
There exists a map $\pi \colon \partial_h X \to \partial X$, 
which is continuous, $\Gamma$-invariant and onto.
	
Moreover, for every geodesic ray $\alpha \colon \R_+ \to X$ starting at $x$, 
the Busemann cocycle along $\alpha$ is a preimage of $\alpha(\infty)$ in $\partial_h X$.
	In addition, two cocycles $b_1,b_2 \in \partial_h X$ have the same image 
in $\partial X$ if and only if $\norm[\infty]{b_1-b_2} \leq 64 \delta$.
\end{prop}

The following lemma ensures that the extension to the horoboundary 
of the Gromov product is close to its value in the Gromov boundary.

\begin{lemm}
	\label{res: finiteness Bourdon product}
	Let $b, b' \in \partial_h X$ be two cocycles, and $x \in X$.
	Let $\xi$ and $\xi'$ be their respective images in $\partial X$.
	Then
	\begin{equation}
		\gro \xi{\xi'}x \leq \gro b{b'}x \leq \gro \xi{\xi'}x + 2 \delta.
	\end{equation}
\end{lemm}

\begin{proof}
	First, if  $\xi = \xi'$, then both $\gro \xi{\xi'}x$ and $\gro b{b'}x$ are infinite.
	Indeed the infiniteness of $\gro \xi{\xi}x$ follows from the definition 
	of the Gromov product on $\bar X$.
	On the other hand, $b$ and $b'$ differ by at most $64\delta$
	(\autoref{res: projection horofunction to boundary}).
	Hence
	\begin{equation*}
		\gro b{b'}x \geq \gro bbx -32\delta \geq \infty.
	\end{equation*}
	Therefore we can assume that $\xi \neq \xi'$.
	By definition of the horoboundary, there exist two sequences $(y_n)$ and $(y'_n)$ of points of $X$ which respectively converge to $b$ and $b'$ in $\bar{X}_h$.
	Up to passing to a subsequence we may assume that $(y_n)$ and $(y'_n)$ respectively converge to $\xi$ and $\xi'$ in $\bar X$.
	Let $z \in X$.
	Triangle inequality gives for all $n \in \N$
	\begin{equation*}
		\frac 12 \left\{\left[\dist {y_n}x - \dist{y_n}z\right] +
		\left[\dist {y'_n}x - \dist{y'_n}z\right]\right\}
		\leq \gro {y_n}{y'_n}x
	\end{equation*}
	Passing to the limit we get
	\begin{equation*}
		\frac 12 \left\{ b(x,z) + b'(x,z)\right\} \leq \liminf_{n \to \infty}\gro {y_n}{y'_n}x \leq  \gro \xi {\xi'}x + 2 \delta.
	\end{equation*}
	This inequality holds for every $x \in X$, hence $\gro b{b'}x \leq \gro \xi{\xi'}x$, which completes the proof of the right inequality.

	\medskip
	For every $n \in \N$, we denote by $p_n$ a projection of $x$ on a geodesic $\geo{y_n}{y'_n}$.
	It follows that
	\begin{equation*}
		\dist x{p_n} \leq \gro{y_n}{y'_n}x + 4\delta,
	\end{equation*}
	see for instance \cite[Chapitre~3, Lemme~2.7]{CooDelPap90}.
	As $(y_n)$ and $(y'_n)$ converges to distinct points in $\partial X$, the sequence $\gro {y_n}{y'_n}x$ is uniformly bounded.
	Recall that $X$ is proper.
	Thus, up to passing to a subsequence we can assume that $(p_n)$ converges to a point $p \in X$.
	Since $p_n$ lies on $\geo{y_n}{y'_n}$, for every $n \in \N$, we have
	\begin{equation*}
		\gro {y_n}{y'_n}x
		= \frac 12 \left\{\left[\dist {y_n}x - \dist{y_n}{p_n}\right] +  \left[\dist {y'_n}x - \dist{y'_n}{p_n}\right]\right\}
	\end{equation*}
	Passing to the limit we get
	\begin{equation*}
		\gro \xi{\xi'}x
		\leq \liminf_{n \to \infty} \gro {y_n}{y'_n}x
		\leq \frac 12 \left\{b(x,p)+b'(x,p)\right\}
		\leq \gro b{b'}x,
	\end{equation*}
	which corresponds to the left inequality.
\end{proof}

%
\subsection{Patterson-Sullivan measures}
%

The Patterson-Sullivan measure  is a well-known very useful object in the study of negatively curved manifolds.
It was extended by Coornaert in the context of  hyperbolic spaces $X$ \cite{Coo93}.
His work used the Gromov compactification $\bar X = X \cup \partial X$.
Nevertheless the measure that he obtains is not exactly conformal but only quasi-conformal.
Following \cite{Burger:1996kc},  we run the construction  in the horocompactification $\bar  X_h = X \cup \partial_hX$ rather than   $\bar X$.
We obtain thus easily an \emph{exactly} conformal family of measures, 
and a $\Gamma$-invariant measure on the double Gromov boundary $ \partial^2 X $, 
contrarily to the $\Gamma$-quasi-invariant construction of  \cite[Corollaire~9.4]{Coo93}.

\paragraph{Poincar\'e series and critical exponent.}

Let $\Gamma$ be a group acting properly by isometries on $X$.
We fix a base point $o \in X$.
To any subset $S$ of $\Gamma$ we associate a \emph{Poincar\'e series} defined by
\begin{equation*}
	\mathcal P_S(s) = \sum_{\gamma \in S} e^{-s\dist{\gamma o}o},
\end{equation*}
Its critical exponent $h_S$ is also the \emph{exponential growth rate} of $S$, i.e.
\begin{equation*}
	h_S
	=
	\limsup_{ r\to \infty} \frac 1r \ln \card{\set{ \gamma \in S}{\dist{\gamma o}o \leq r}}.
\end{equation*}
This quantity does not depend on the choice of $o$.
The group $\Gamma$ is called \emph{convergent} (\resp \emph{divergent})
if the Poincar\'e series $\mathcal P_\Gamma(s)$ converges (\resp diverges) at
$s = h_\Gamma$.
According to Patterson study of Dirichlet series \cite{Patterson:1976hp}, 
 there exists a map $\theta_0 \colon \R_+ \to \R_+$ with the following properties.
\begin{enumerate}
	\item For every $\varepsilon > 0$, there exists $t_0 \geq 0$, 
such that for every $t \geq t_0$ and $u \geq 0$, we have $ \theta_0(t + u ) \leq e^{\varepsilon u} \theta_0(t)$.
	\item The weighted series
	      \begin{equation}\label{eq: Patterson trick}
		      \mathcal P'_\Gamma(s)
		      =
		      \sum_{\gamma \in S} \theta_0(\dist o{\gamma o})e^{-s\dist{\gamma o}o}
	      \end{equation}
 is divergent whenever $s \leq h_\Gamma$, and convergent otherwise.
\end{enumerate}

\paragraph{Measure on the horoboundary.}
Let us now define the Patterson-Sullivan measure.
 It is well known that there is a one-to-one correspondence between Radon measures and positive linear forms on the space of continuous functions.  
We adopt the latter point of view here. It  may look overcomplicated, however it emphasizes the analogy with the \emph{twisted Patterson-Sullivan measure} that we are going to define in \autoref{sec: conf family of hilbert functionals}.

Denote by $C(\bar{X}_h)$ the set of continuous functions from $\bar{X}_h$ to $\R$.
Let $x \in X$.
For every $s > h_\Gamma$, we define a  positive  continuous linear form
$L \colon C(\bar{X}_h) \to \R$ by
\begin{equation*}
	L_{x,s}(f)
	=
	\frac 1{\mathcal P'_\Gamma(s)} \sum_{\gamma \in \Gamma} \theta_0(\dist x{\gamma o}) e^{-s \dist x{\gamma o}} f(\gamma o).
\end{equation*}
Since $\bar{X}_h$ is compact, the dual of $C(\bar{X}_h)$ endowed
with the weak-$\ast$ topology is compact as well.
Thus, there exists a sequence $(s_n)$ converging to $h_\Gamma$
such that $(L_{o,s_n})$ converges to a  positive  continuous linear form $L_o \colon C(\bar{X}_h) \to \R$.
By Riesz representation Theorem, there exists
a unique  Radon  measure $\tilde \nu_o$ on $\bar{X}_h$ 
such that for every $f \in C(\bar{X}_h)$
\begin{equation*}
	L_o(f) = \int f d\tilde \nu_o.
\end{equation*}
By construction of $\theta_0$,  the series  $\mathcal P'_\Gamma(s_n)$ diverges
when $s_n$ approaches to $h_\Gamma$.
As a consequence the support of the measure $\tilde \nu_o$ is contained in $\partial_h X$.
A standard argument shows that for every $x \in X$,
$(L_{x,s_n})$ also converges to a continuous linear form on $C(\bar{X}_h)$ that can be represented by a measure $\tilde \nu_x$ on $\bar X_h$ which belongs to the same class as $\tilde \nu_o$.
The resulting family $(\tilde \nu_x)_{x\in X}$ is \emph{$h_\Gamma$-conformal}, i.e.  for $\tilde \nu_o$-almost every $b\in\partial_hX$,
\begin{equation*}
	\frac {d\tilde \nu_x}{d\tilde \nu_y} (b)  = e^{-h_\Gamma b(x,y)}.
\end{equation*}
This family is also \emph{$\Gamma$-equivariant} in the sense that for all $\gamma\in\Gamma$ and $x\in X$, we have $\gamma_\ast \tilde \nu_x = \tilde \nu_{\gamma x}$.

\paragraph{Measure on the Gromov boundary.}
Recall that   $\pi \colon \partial_h X \to \partial X$ denotes the continuous $\Gamma$-invariant map from the horoboundary to the Gromov boundary (\autoref{res: projection horofunction to boundary}).
For  $x \in X$,  denote by $\nu_x = \pi_\ast \tilde \nu_x$
the push-forward measure.
As $\pi$ is $\Gamma$-equivariant, so is the family $(\nu_x)$.
Recall that any two cocycles $b,b' \in \partial_h X$ lying in the same fibre of $\pi$ differ by at most $64\delta$.
It follows that $(\nu_x)$ is \emph{$h_\Gamma$-quasi-conformal}, i.e. there exists $C \in \R_+^*$, such that for every $x,y \in X$, for $\nu_0$-almost every $\xi \in \partial X$, for every $b \in \pi^{-1}(\xi)$,
\begin{equation}\label{eqn: quasi conformality Coornaert}
	\frac 1C e^{-h_\Gamma b(x,y)} \leq \frac {d \nu_x}{d \nu_y} (\xi)  \leq C e^{-h_\Gamma b(x,y)}.
\end{equation}
A key tool is the well known Sullivan Shadow Lemma,  due to Coornaert in our context \cite{Coo93}.
Recall that $o$ is a fixed base point in $X$.
Let $x \in X$ and $r \in \R_+$.
The \emph{shadow of $B(x,r)$ seen from $o$}, denoted  by $\mathcal O_o(x,r)$,
is the set of points $y \in \bar X$ for which there exists a geodesic from $o$ to $y$ intersecting the ball $B(x,r)$.

\begin{lemm}[Shadow Lemma {\cite[Proposition~6.1]{Coo93}}]
	\label{res: shadow lemma}
	Let $(\alpha_x)_{x\in X}$ be a $\Gamma$-invariant $a$-quasi-conformal family of measures on the Gromov boundary $\partial X$.
	There exist $r_0 , C\in \R_+^*$ such that for all $r\geq r_0$, for all $\gamma\in\Gamma$,
	\begin{equation*}
		\frac{1}{C}e^{-a d(o,\gamma o)}
		\leq \alpha_o\left(\mathcal O_o(\gamma o,r)\right)
		\leq Ce^{2ar}e^{-a d(o,\gamma o)}.
	\end{equation*}
\end{lemm}

In terms of shadows, the radial limit set (defined in the previous section) is also the set of points $\xi\in \partial X$ which belong to infinitely many distinct shadows $\mathcal{O}_x(\gamma_n o,r)$ for some $x\in X$ and $r \in \R_+^*$.

\begin{coro}
	\label{res: noatom}
	Assume that $\nu_o$ gives full measure to the radial limit set.
	Then it is unique, non-atomic, and is  ergodic with respect to the action of $\Gamma$ on $\partial X$.
	Moreover the Poincar\'e series of $\Gamma$ diverges at $h_\Gamma$%
\end{coro}

\begin{proof} The proof is well known and elementary.
	We recall the arguments, as they will appear later in a more sophisticated manner (see \autoref{res: twisted absolute continuity}).
	First,  it is non-atomic. Indeed,  \autoref{res: shadow lemma}  implies that any radial limit point has a sequence of decreasing neighbourhoods whose measure decreases to zero.

	\medskip
	Let us show that $\nu_o$ is  ergodic. Let $A\subset\partial X$ be a $\Gamma$-invariant set with $\nu_o(A)>0$.
	Without loss of generality, we can assume that $A\subset \Lambda_{\rm rad}^k(\Gamma)$, for some compact subset $k \subset X$.
	Consider the new family of measures
	\begin{equation*}
		\nu'_x = \frac 1{\nu_o(A)} \mathbf 1_A \nu_x,
	\end{equation*}
	which is also $\Gamma$-invariant and $h_\Gamma$-quasi-conformal.
	Therefore, it also satisfies the Shadow Lemma.
	In particular, for all $r\in \R_+$, for all $\gamma\in \Gamma$,
\begin{equation*}
\nu_o(\mathcal O_o(\gamma o,r))\leq 
C(r) \nu'_o(\mathcal O_o(\gamma o,r)),
\end{equation*}
	where $C(r) \in \R_+^*$ is a parameter which only depends on $r$.
	By a Vitali type argument, one easily proves that  for  any compact subset $K$ containing $k$, in restriction to $\Lambda_{\rm rad}^K$, the measure $\nu_o$ is absolutely continuous  with respect to  $\nu'_o$.
	We deduce that $\nu_o(\Lambda_{\rm rad}^K\setminus A)=0$ for all $K\supset k$, so that $\nu_o(\partial X\setminus A)=0$.
	Uniqueness directly follows from the ergodicity.

	\medskip
	As $\nu_o$ gives full measure to $\Lambda_{\rm rad}$, there exists 
some compact subset $k \subset X$ large enough so that $\Lambda_{\rm rad}^k$
 has positive measure.
In addition, there exists $r > 0$, 
such that for every finite subset $S$ of $\Gamma$ the collection
	\begin{equation*}
		\left( \mathcal O_o(\gamma o, r)\right)_{\gamma \in \Gamma \setminus S}
	\end{equation*}
	covers $\Lambda_{\rm{rad}}^k$.
	By \autoref{res: shadow lemma}, there exists $\varepsilon \in \R_+^*$, independent of $S$, such that
	\begin{equation*}
		\sum_{\gamma \in \Gamma \setminus S} e^{-h_\Gamma\dist o{\gamma o}}
		\geq \varepsilon \sum_{\gamma \in \Gamma \setminus S} \nu_o\left(\mathcal O_o(\gamma o, r)\right)
		\geq \varepsilon \nu_o\left(\Lambda_{\rm{rad}}^k\right)
		> 0.
	\end{equation*}
	Hence the Poincar\'e series of $\Gamma$ diverges at $h_\Gamma$.
\end{proof}

%
\subsection{The Bowen-Margulis current.}
%

Given any two cocycles $b, b' \in \partial_hX$, we define
\begin{equation*}
	D(b,b') = e^{-\gro b{b'}o}.
\end{equation*}
It can be thought of as the analogue of the Bourdon distance (cf \cite{Bourdon:1995em}),
except that it does not satisfy the triangle inequality. By definition of
$\gro b{b'}o$ we get:
\begin{equation}\label{res: action gromov product on horoboundary}
	D(\gamma^{-1}b, \gamma^{-1}b')
	= e^{-\frac 12\left[b(\gamma o, o) + b'(\gamma o,o)\right]}D(b,b').
\end{equation}
We follow the standard notations for the double boundary of $X$ and let
\begin{align}
	\label{eqn: def double Gromov boundary}
	\partial^2X & = \set{ (\eta, \xi) \in \partial X \times \partial X}{\eta \neq \xi}, \\
	\label{eqn: def double horoboundary}
	\partial_h^2X & = \set{(b,b') \in \partial_hX \times \partial_h X }{\pi(b) \neq \pi(b')}.
\end{align}
We still denote by $\pi$ the continuous $\Gamma$-invariant map ${\partial_hX \times \partial_h X \to \partial X \times \partial X}$ induced by $\pi \colon \partial_h X \to \partial X$.

\begin{defi}\label{res: BM current}
	The \emph{Bowen-Margulis current on $\partial^2_hX $} is the measure $\tilde \mu$  defined  by
	\begin{equation*}
		\tilde \mu = \frac 1{D^{2h_\Gamma}} \tilde \nu_o \otimes \tilde \nu_o.
	\end{equation*}	
	 
	The \emph{Bowen-Margulis current on $\partial^2X$} is the push-forward measure $\mu = \pi_\ast\tilde \mu$.
\end{defi}

By \autoref{res: finiteness Bourdon product},
there exists $C_0 \in \R_+^*$ such that for $\mu$-almost every $(\eta, \xi) \in \partial^2X$, 
\begin{equation}
\label{eqn: bm product}
	 \frac 1{C_0}e^{2h_\Gamma \gro \eta\xi o} \leq \frac{d\mu}{d\left(\nu_o \otimes \nu_o\right)} (\eta, \xi)\leq C_0e^{2h_\Gamma \gro \eta\xi o}.
\end{equation}

The above definitions combined with (\ref{res: action gromov product on horoboundary}) give the following lemma.

\begin{lemm}
\label{res: inv BM current}
	The Bowen-Margulis currents $\tilde \mu$ on $\partial_h^2X$ 
and $\mu$ on $\partial^2X$ are both $\Gamma$-invariant.
	If the Patterson-Sullivan measure $\nu_0$ on $\partial X$ gives full measure to the radial limit set,
	then $\mu$ gives full measure to  $(\Lambda_{\rm rad} \times \Lambda_{\rm rad}) \cap \partial^2X$. 
\end{lemm}

%
\section{Strongly positively recurrent actions}\label{sec: spr}
%

The presentation is strongly inspired from Schapira-Tapie \cite{Schapira:2018ti} but has been slightly modified and simplified to adapt in an easier way to less smooth actions on general hyperbolic spaces.

\subsection{Entropy outside a compact set}

Let $(X,d)$ be a proper geodesic metric space, and $\Gamma$  a group acting properly by isometries on $X$.
Given a compact subset $K$ of $X$, let  $\Gamma_K$ be the set of elements $\gamma \in \Gamma$ for which there exists two points $x,y \in K$ and a geodesic $c \colon [a,b] \to X$ joining $x$ to $\gamma y$ such that $c \cap \Gamma \cdot K$ is contained in $K \cup \gamma K$.
We call the critical exponent $h_{\Gamma_K}$ of the Poincar\'e series $\mathcal P_{\Gamma_K}$ the \emph{entropy outside $K$}.
Given any two compact subsets $k\subset K$ of $X$, observe that $\Gamma_K \subset \Gamma_k$, whence $h_{\Gamma_K} \leq h_{\Gamma_k}$.

\begin{defi}
	\label{res: entropy at infinity}
	The \emph{entropy at infinity} $h_\Gamma^\infty$ is the quantity
	\begin{equation*}
		h_\Gamma^\infty = \inf_K h_{\Gamma_K}
	\end{equation*}
	where the infimum runs over all compact subsets of $X$.
\end{defi}

\begin{defi}
	\label{def: spr}
	The action of $\Gamma$ on $X$ is \emph{strongly positively recurrent} if $h_\Gamma^\infty < h_\Gamma$.
	We also say that the action has a \emph{growth gap at infinity}.
\end{defi}

%
\subsection{Examples}
%

We present some examples of  strongly positively recurrent  actions. 
\autoref{exa: spr - hyperbolic group} is  a trivial one. 
The simplest non trivial example is a geometrically finite group acting on a negatively  curved manifold with a \emph{parabolic gap}, as studied by Dal'bo et al. in \cite{Dalbo:2000eh}, see \autoref{res: rel hyp - critical exponent at infinity}.  
We refer to \cite{Schapira:2018ti} for more examples in a Riemannian setting such as geometrically finite manifolds, Schottky products, infinite genus Ancona surfaces, etc.
If one does not assume that the space $X$ on which $\Gamma$ acts is hyperbolic, Arzhantseva et al. \cite{Arzhantseva:2015cl} and Yang  \cite{Yang:2016wa} produce other examples, e.g. some rank one actions on CAT($0$) spaces and some actions of subgroups of mapping class groups.

	\medskip

\begin{exam}[Non elementary hyperbolic groups]
	\label{exa: spr - hyperbolic group}
	Let $\Gamma$ be a group acting properly cocompactly on a geodesic $\delta$-hyperbolic space $X$ (in particular $\Gamma$ is a hyperbolic group).
	If $\Gamma$ is non elementary, this action is always strongly positively recurrent.
	Indeed, as the action is cocompact, there exists a compact subset $K$ of $X$ such that $\Gamma K$ covers $X$.
	Thus, $\Gamma_K$ is contained in 
	\begin{equation*}
		\set{\gamma \in \Gamma}{K \cap \gamma K \neq \emptyset}.
	\end{equation*}
	Since the action is proper, the latter set is finite, hence $h_{\Gamma_K} = 0$.
	As $\Gamma$ is non-elementary $h_\Gamma > 0$.
	Thus the action is strongly positively recurrent.
\end{exam}

\begin{exam}[Relative hyperbolic groups]
	There exist many equivalent definitions of relative hyperbolic groups.
	Let us recall the one that fits to our context, see for instance Bowditch \cite{Bowditch:2012ga} or Hruska \cite[Definition~3.3]{Hruska:2010iw}.

	\medskip
	Let $\Gamma$ be a group and $\mathcal P$ a finite collection of finitely generated subgroups of $G$.
	Assume that $\Gamma$ acts properly by isometries on a geodesic hyperbolic space $X$.
	We say that the action of $(\Gamma, \mathcal P)$ on $X$ is \emph{cusp-uniform} if there exists a $\Gamma$-invariant family $\mathcal Z$ of pairwise disjoint horoballs in $X$ with the following properties.
	\begin{enumerate}
		\item The action of $\Gamma$ on $X\setminus U$ is cocompact, where $U$ stands for the union of all horoballs $Z \in \mathcal Z$.
		\item For every $Z \in \mathcal Z$, the stabilizer of $Z$ is conjugated to some $P \in \mathcal P$.
	\end{enumerate}
	The group $\Gamma$ is \emph{hyperbolic relative to $\mathcal P$} if $(\Gamma, \mathcal P)$ admits a cusp-uniform action on a hyperbolic space.

	\medskip
	The definition of cusp-uniform action mimics the decomposition of finite volume hyperbolic manifolds as the union of a compact part and finitely many cusps.
	Hence the proof of the next statement works as in Schapira-Tapie \cite[Proposition~7.16]{Schapira:2018ti}.
	The details are left to the reader.

	\begin{prop}
		\label{res: rel hyp - critical exponent at infinity}
		Let $\Gamma$ be a group and $\mathcal P$ a finite collection of finitely generated subgroups of $\Gamma$.
		Let $X$ be a hyperbolic space, endowed with a cusp-uniform action of $(\Gamma, \mathcal P)$.
		The critical exponent at infinity for this action is
		\begin{equation*}
			h_\Gamma^\infty = \max_{P \in \mathcal P} h_P.
		\end{equation*}
		In particular the action of $\Gamma$ on $X$ is strongly positively recurrent if $h_P < h_\Gamma$, for every $P \in \mathcal P$.
	\end{prop}

	\paragraph{Remark.}
	We recover here the parabolic gap condition, introduced by Dal'bo, Otal and Peign\'e \cite{Dalbo:2000eh}.
	 It also follows from this statement that if any group $\Gamma$ (not necessarily a relatively hyperbolic one) admits a strongly positively recurrent action, then it is non-elementary (for this action).

	\medskip
	We now focus on a specific cusp-uniform action, following with minor variations the Groves-Manning construction \cite{Groves:2008ip}.
	Given a geodesic metric space $Y$, the \emph{horocone} over $Y$ is the space
	$Z(Y) = Y \times \R_+$ whose metric is modelled on the standard hyperbolic plane $\H^2$ as follows:
	if $x = (y,r)$ and $x' = (y',r')$ are two points of $Z(Y)$, then
	\begin{equation*}
		\cosh \dist x{x'} = \cosh (r-r') + \frac 12 e^{-(r+r')} \dist y{y'}^2.
	\end{equation*}
	It is a geodesic hyperbolic space.
	It comes with a natural $1$-Lipschitz embedding $\iota \colon Y \to Z(Y)$ sending $y$ to $(y,0)$.

	\medskip
	Let $\Gamma$ be a group and $\mathcal P$ a finite collection of finitely generated subgroups of $G$.
	Let $S$ be a finite generating subset of $G$ such that for every $P \in \mathcal P$, the set $S \cap P$ generates $P$.
	Let $X$ (\resp $Y_P$) be the Cayley graph of $\Gamma$ (\resp $P$) with respect to $S$ (\resp $S \cap P$).
	It follows from our assumption that $Y_P$ isometrically embeds in $X$.
	The \emph{cone-off} space $\dot X$ is the space obtained by attaching for every $P \in \mathcal P$ and $\gamma \in \Gamma$, the horocone $Z(\gamma Y_P)$ onto $X$ along $\gamma Y_P$ according to the canonical embedding $\gamma Y_P \to Z(\gamma Y_P)$.
	We endow this space with the largest pseudo-metric such that the maps $X \to \dot X$ and $Z(\gamma Y_P) \to \dot X$ are $1$-Lipschitz.
	It turns out that this pseudo-metric is actually a distance.
	Moreover the space $\dot X$ is proper and geodesic.
	 In addition, the action of $\Gamma$ on $X$ extends to a proper action on $\dot X$.
	As $\Gamma$ is hyperbolic relative to $\mathcal P$, the space $\dot X$ is hyperbolic, moreover the action of $(\Gamma, \mathcal P)$ on $\dot X$ is cusp-uniform.

	\begin{prop}
		\label{res: rel hyp group are spr}
		Assume that every $P \in \mathcal P$ is virtually nilpotent.
		If the action of $\Gamma$ on $\dot X$ is non elementary then it is strongly positively recurrent.
	\end{prop}

	\paragraph{Remarks.}
	Being a strongly positively recurrent is a property of the \emph{action} of $\Gamma$ and not of the group $\Gamma$ itself.
	 The proposition states  that the action of $\Gamma$ on the \emph{cone-off space $\dot X$} is strongly positively recurrent.
	However this is not the case of any cusp-uniform action of $(\Gamma, \mathcal P)$ on a $\delta$-hyperbolic space.
	Indeed Dal'bo, Otal and Peign\'e produced an example of a geometrically finite manifold $M$ with pinched negative curvature whose fundamental group $\Gamma = \pi_1(M)$ contains a parabolic subgroup $P$ (isomorphic to $\Z$) whose critical exponent is the same as the one of $\Gamma$ \cite[Th\'eor\`eme~C]{Dalbo:2000eh}.
	In particular, this action is not strongly positively recurrent.
	Their construction strongly relies on the fact that the curvature of $M$ is not constant.  Indeed, an explicit computation shows that in locally symmetric spaces with negative curvature, all parabolic groups have a divergent Poincar\'e series (cf \cite{Dalbo:2000eh} for the case of real hyperbolic surfaces).
		By \autoref{res: hyperbolic rel divergent implies SPR} below, this implies that all groups acting on a locally symmetric space with a geometrically finite action have a growth gap at infinity. 

	In the above construction, the metric on each horocone $Z(Y)$ is modelled on the one of the standard hyperbolic plane $\mathbb H^2$.
	Hence, although there is no appropriate notion of sectional curvature in this context, it is natural to think of $\dot X$ as a space with constant curvature equal to $-1$.

	\medskip
	A variation of \autoref{res: rel hyp group are spr} already appears in the course of the proof of \cite[Theorem~8.1]{Arzhantseva:2015cl}.
	However the argument is rather terse.
	For the convenience of the reader, we expose an alternative approach, which is of independent interest.
	We start with the following statement.

	\begin{lemm}
		\label{res: rel hyp - divergent parabolic}
		Let $P \in \mathcal P$.
		If $P$ is virtually nilpotent, then the action of $P$ on $\dot X$ is divergent.
	\end{lemm}

	\begin{proof}
		For simplicity we let $Y = Y_P$.
		Since $\dot X$ is hyperbolic, there exists $r \geq 0$ such that the subspace $Z_r(Y) = Y \times [r, \infty)$ of $Z(Y)$ \emph{isometrically} embeds in $\dot X$ \cite[Proposition~3.12]{Coulon:2015jr}.
		Hence it suffices to prove that $P$ is divergent for its action on $Z(Y)$.
		We denote by $o$ the image in $Z(Y)$ of the vertex of $Y$ corresponding to the trivial element in $P$.
		For every $\gamma \in P$ we have
		\begin{equation*}
			\cosh \dist {\gamma o}o = 1 + \frac 12 \left|\gamma\right|^2,
		\end{equation*}
		where $\abs{\gamma}$ stands for the length of $\gamma$ with respect to the word metric on $P$ induced by $S \cap P$.
		A direct computation shows that
		\begin{equation*}
			e^{-\dist{\gamma o}o} = \frac 14 \left( \sqrt{\left|\gamma\right|^2 +4} - \abs{\gamma}\right)^2
		\end{equation*}
		Hence the Poincar\'e series of $P$ for its action on $Z(Y)$ computed at $s$ is
\begin{equation*}
\mathcal P_P(s) = \sum_{k \in \N} \card{S(k)}a_k
\quad \text{where} \quad
a_k = \left(\frac{\sqrt{k^2+4} - k}2 \right)^{2s},
\end{equation*}
and $S(k)$ stands for the sphere  of radius $k$ of $P$ with respect to the word metric induced by $S\cap P$. 
	Using Abel's transformation we compute the partial series associated to $\mathcal P_P(s)$.
		More precisely, for every $n \in \N$, we have
\begin{equation*}
\sum_{k =0}^n \card{S(k)} a_k
= 
\sum_{k = 0}^{n-1} \card{B(k)}\left(a_k - a_{k+1}\right) + \card{B(n)}a_n,
\end{equation*}
		where $B(k)$ stands for the \emph{ball} of radius $k$ of $P$ with respect to the word metric induced by $S\cap P$.
		A simple asymptotic expansion yields
		\begin{equation*}
			a_k \asim_{k \to \infty} \frac 1{k^{2s}}
			\quad \text{and} \quad
			\left(a_k - a_{k+1}\right) \asim_{k \to \infty} \frac {2s}{k^{2s+1}}.
		\end{equation*}
		Recall that $P$ is virtually nilpotent.
		According to Bass \cite{Bass:1972eu} and Guivarc'h \cite{Guivarch:1971up},
		there exist $A,B \in \R_+^*$ and $d \in \N$ such that for every $k \in \N$,
		\begin{equation*}
			Ak^d \leq \card{B(k)} \leq B k^d.
		\end{equation*}
		Combining this estimate with the  previous asymptotic expansion, we deduce that $\mathcal P_P(s)$ converges if and only if $s > d/2$.
		In particular the group $P$ is divergent.
	\end{proof}

		\begin{proof}[Proof of \autoref{res: rel hyp group are spr}]
			By \autoref{res: rel hyp - critical exponent at infinity}, there exists $P \in \mathcal P$ such that $h_\Gamma^\infty = h_P$.
			Moreover, since the action of $\Gamma$ on $X$ is non-elementary, the limit set $\Lambda(\Gamma)$ is infinite, whereas $\Lambda(P)$ is a single point. 
			Recall also that $P$ is divergent (\autoref{res: rel hyp - divergent parabolic}).
			By \cite[Proposition~2]{Dalbo:2000eh}, we get $h_\Gamma > h_P = h_\Gamma^\infty$ (this reference is  written in the context of negatively curved manifolds, but the proof  applies verbatim to our setting).
		\end{proof}

	\begin{rema}
	\label{res: hyperbolic rel divergent implies SPR}
		The previous proof can be adapted, using a variation of the construction of Abbott-Hume-Osin \cite{Abbott:2017tw} to get the following combination result.

		Let  $\mathcal P$ be a collection of finitely generated subgroups of $\Gamma$. 		
		Assume that $\Gamma$ is hyperbolic relative to $\mathcal P$ and non-elementary.
		If each parabolic group $P \in \mathcal P$ admits a divergent action on a hyperbolic space $X_P$, then there exists a hyperbolic space $X$ on which $\Gamma$ admits a  strongly positively recurrent  action.
		The proof is left to the interested reader.  
	\end{rema}
\end{exam}

%
\subsection{Radial limit set.}
%
\label{sec: spr radial limit set}

Let $\Gamma$ be a group with a strongly positively recurrent action on a hyperbolic space $X$. 
This assumption has a key consequence: the Patterson-Sullivan measure gives full measure to $\Lambda_{\rm{rad}}^r$ for some $r \in \R_+$, see \autoref{res: PS charges radial limit set}.  
As mentioned in the introduction,  being strongly positively recurrent  is useful but not necessary here, see \autoref{res: noatom}. 
It will be crucial in \autoref{res: twisted PS charges radial limit set}.
Several results in this section  have been proven in \cite{Schapira:2018ti} in a Riemannian setting. In our Gromov-hyperbolic setting,   some arguments need to be slightly adapted.

\medskip

Let $K$ be a compact subset of $X$.
Denote by $\mathcal L_K$ the set of points $\xi \in \partial X$ for which there exists a geodesic ray $c \colon \R_+ \to X$ starting in $K$, ending at $\xi$ and such that $c \cap \Gamma K$ is contained in $K$.
For every $T \in \R_+$, define $U_K^T$ by
\begin{equation*}
	U_K^T = \set{ x \in \bar X}{ \exists \xi \in \mathcal L_K,\ \gro x{\xi}o \geq T}.
\end{equation*}

\begin{lemm}[Compare with {\cite[Proposition~7.29]{Schapira:2018ti}}]
	\label{res: K-radial limit set vs AK}
	For every compact   set $K\subset X$, we have
	\begin{equation*}
		\partial X \setminus  \Lambda_{\rm{rad}}^K \subset \Gamma \mathcal L_K.
	\end{equation*}
\end{lemm}

\paragraph{Remark.}
In comparison with \cite{Schapira:2018ti} observe that $\Lambda_{\rm rad}^K\cap \Gamma \mathcal L_K$ could be non-empty.
This follows from the fact that two points in $X \cup \partial X$ may be joined by several geodesics, one intersecting infinitely many translates of $K$ and an the other not.

\begin{proof}
	Let $\xi \in \partial X \setminus  \Lambda_{\rm{rad}}^K$.
	Let $c \colon \R_+ \to X$ be any geodesic ray starting in $K$ and ending at $\xi$.
	Since $\xi$ does not belong to $\Lambda_{\rm{rad}}^K$, there exists $t\in \R_+$ and $\gamma \in \Gamma$ such that $c(t)$ belongs to $\gamma K$ and $c$ restricted to $(t, \infty)$ does not intersects $\Gamma K$.
	It follows that $\gamma^{-1}\xi$ belongs to $\mathcal L_K$, hence the result.
\end{proof}

\begin{lemm}
	\label{res: thickening lemma}
	Let $k\subset X$ be a compact  set containing the base point $o$ and $K$ its $6\delta$-neighbourhood.
	There exist a finite subset $S \subset \Gamma$ and  $r_0 \in \R_+^*$ with the following property.
	Let $x \in K$, $y \in X \cup \partial X$ and $c \colon I \to X$ a geodesic joining $x$ to $y$ such that $c \cap \Gamma K \subset K$.
	For every $\gamma \in \Gamma$, there exists $\beta \in S\Gamma_k$ such that
	\begin{enumerate}
		\item $\gro x{\gamma o}{\beta o} \leq r_0$ and $\gro y{\gamma o}{\beta o} \leq r_0$,
		\item $\dist x{\beta o} \geq \gro y{\gamma o}x - r_0$.
	\end{enumerate}
\end{lemm}

\paragraph{Remark.}
Working with the Gromov product is very convenient when geodesics are not unique, but sometimes confusing at the first sight.
The above statement has the following geometrical meaning.
If one approximates the triangle $[x,y,\gamma o]$ by a tripod, then $\beta o$ lies close to the branch joining $\gamma o$ to the centre of the tripod (see \autoref{fig: tripod}).

\begin{figure}[htpb]
	\begin{center}
		\includegraphics[page=4, width=0.9\linewidth]{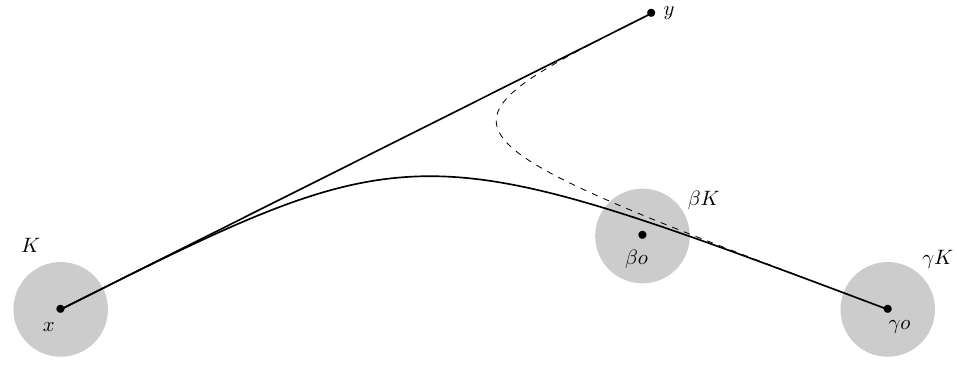}
		\caption{A geodesic tripod}
		\label{fig: tripod}
	\end{center}
\end{figure}

\begin{proof} Let $D$ be the diameter of $K$ and  $r_0 = D + 2\delta$.
	Since the action of $\Gamma$ on $X$ is proper, the set
	\begin{equation*}
		S = \set{\gamma \in \Gamma}{\dist o{\gamma o} \leq 3D + 6\delta}
	\end{equation*}
	is finite.
	Let $x \in K$, $y \in X \cup \partial X$ and $c \colon I \to X$ be a geodesic joining $x$ to $y$ such that $c \cap \Gamma K\subset K$.
	Let $\gamma \in \Gamma$.
	We fix a geodesic $c_\gamma \colon [0,a] \to X$ joining $x$ to $\gamma o$.
	Let $s$ be the largest time in $[0,a]$ such that $s \leq D + 6\delta$ and $c_\gamma(s)$ belongs to $\alpha k$, for some $\alpha \in \Gamma$.

	Similarly, let $t> D+ 6\delta$ be the smallest time in $[s,a]$ such that $c_\gamma (t)$ belongs to $\beta  k\setminus \alpha k$, for some $\beta\in \Gamma$.
	As $x$ and $o$ both belong to $K$, $\dist o{\alpha o} \leq 3D + 6\delta$, so that $\alpha$ belongs to $S$.
	It follows from the construction that $\alpha^{-1}\beta$ belongs to $\Gamma_k$, hence $\beta \in S\Gamma_k$.
	Let us now prove that $\beta$ satisfies the announced inequalities.
	For simplicity, set $z = c_\gamma(t)$.
	Observe first that
	\begin{equation*}
		\gro  x{\gamma o}{\beta o} \leq \gro x{\gamma o}z + \dist z{\beta o} \leq D \leq r_0.
	\end{equation*}
	Applying twice the four point inequality (\ref{eqn: four point inequality}) we get
	\begin{equation}
		\label{eqn: thickening lemma}
		\min\left\{ \gro xy z, \gro y{\gamma o}z \right\} \leq \gro x{\gamma o}z + 2\delta \leq 2\delta.
	\end{equation}
	Assume now that the minimum is achieved by $\gro xyz$.
	In particular $z$ is $6\delta$-close to point $z'$ on $c$ \cite[Chapitre~3, Lemme~3.7]{CooDelPap90}.
	Note also that $z'$ belongs to $\beta K$.
	As $c \cap \Gamma K$ is contained in $K$, the point $z'$ actually belongs to $K$.
	It forces $\dist xz \leq D + 6\delta$, which contradicts the definition of $t$.
	Consequently the minimum in (\ref{eqn: thickening lemma}) is achieved by $\gro y{\gamma o}z$ which yields $\gro y{\gamma o}z \leq 2\delta$.
	Hence $\gro y{\gamma o}{\beta o} \leq D + 2\delta$, which completes the proof of the first point.
	It follows from the triangle inequality that
	\begin{equation*}
		\gro y{\gamma o}x \leq \dist xz + \gro y{\gamma o}z
		\leq \dist x{\beta o} + D + 2\delta,
	\end{equation*}
	which corresponds to the second point.
\end{proof}

\begin{lemm}[Compare with {\cite[Equation~(27)]{Schapira:2018ti}}]
	\label{res: open around AK covered by shadow}
	Let $k \subset X$ be a compact subset containing $o$ and $K$ its $6\delta$-neighbourhood.
	There exists a finite subset $S$ of $\Gamma$ and $r \in \R_+$ with the following properties.
	For every $T \in \R_+$,
	\begin{equation*}
		U_K^T \cap \Gamma o \subset \bigcup_{\substack{\beta \in S\Gamma_k, \\ \dist o{\beta o} \geq T - r}} \mathcal O_o(\beta o ,r).
	\end{equation*}
\end{lemm}

\begin{proof}
Let $S \subset \Gamma$ and $r_0 \in \R_+$ be given by \autoref{res: thickening lemma}. 
Set $r = r_0 + 2D + 4 \delta$, where $D$ is the diameter of $K$. 
	Let $T \in \R_+$.
	Let $\gamma \in \Gamma$ such that $\gamma o$ belongs to $U_K^T$.
	In particular there exists $\xi \in \mathcal L_K$ such that $\gro \xi{\gamma o} o \geq T$.
	By   definition of $\mathcal L_K$, there exists a geodesic ray $c \colon \R_+ \to X$ starting in $K$, ending at $\xi$ such that $c\cap \Gamma K$ is contained in $K$.
	For simplicity, set $x = c(0)$.
	By \autoref{res: thickening lemma}, there exists $\beta \in S\Gamma_k$ such that
	\begin{equation*}
		\dist x{\beta o} \geq \gro \xi{\gamma o}x - r_0
		\quad \text{and} \quad
		\gro x{\gamma o}{\beta o} \leq r_0.
	\end{equation*}
	Observe that $\dist ox \leq D$, as $o$ and $x$ both belongs to $K$.
	It follows from the triangle inequality that
	\begin{equation*}
		\dist o{\beta o}
		\geq \gro \xi{\gamma o}o - r_0 - 2D
		\geq T -r
		\quad \text{and} \quad
		\gro o{\gamma o}{\beta o} \leq r_0 + D < r-4\delta
	\end{equation*}
	The latter point implies that $\gamma o$ belongs to $\mathcal O_o(\beta o, r)$ \cite[Chapitre~3, Lemme~3.7]{CooDelPap90}, whence the result.
\end{proof}

\begin{prop}[Compare with {\cite[Proposition~7.31]{Schapira:2018ti}}]
	\label{res: measure of UTK}
	Assume that the action of $\Gamma$ on $X$ is strongly positively recurrent.
	There exists a compact subset $K$ of $X$ and numbers $a, C, T_0 \in \R_+^*$ such that for every $T \geq T_0$, for every non-negative function $f \in C^+(\bar X)$ whose support is contained in $U_K^T$,
	\begin{equation*}
		\int f d\nu_o \leq C\norm[\infty] f e^{-aT}.
	\end{equation*}
\end{prop}
\begin{proof}
	Since the action of $\Gamma$ is strongly positively recurrent, there exists a compact subset $k$ of $X$ such that $h_{\Gamma_k} < h_\Gamma$.
	Up to enlarging $k$, we may assume that $o$ belongs to $k$.
	Let $K$ be the $6\delta$-neighbourhood of $k$.
	By \autoref{res: open around AK covered by shadow}, there exists a finite subset $S$ of $\Gamma$ and a number $r \in \R_+$ such that for every $T \in \R_+$,
	\begin{equation}
		\label{eqn: measure of UTK}
		U_K^T \cap \Gamma o \subset \bigcup_{\substack{\beta \in S\Gamma_k \\ \dist o{\beta o} \geq T - r}} \mathcal O_o(\beta o ,r).
	\end{equation}
	Let $\varepsilon > 0$ be such that $h_\Gamma -2\varepsilon > h_{\Gamma_k}$.
	Since $S$ is finite, $S\Gamma_k$ and $\Gamma_k$ have the same critical exponent.
	In particular the Poincar\'e series associated to $S\Gamma_k$ converges at $h_\Gamma - \varepsilon$.
	More precisely there exists $B \in \R_+$ such that for every $T \geq 0$, we have
	\begin{equation}
		\label{eqn: measure of UTK - remainder}
		\sum_{\substack{\beta \in S\Gamma_k, \\ \dist o{\beta o} \geq T-r}}e^{-(h_\Gamma - \varepsilon) \dist o {\beta o}} \leq Be^{-(h_\Gamma - h_{\Gamma_k} - 2\varepsilon)T}.
	\end{equation}
	Recall that $\theta_0 \colon \R_+\to \R_+$ is the slowly increasing function used in (\ref{eq: Patterson trick}) to define the Patterson-Sullivan measure $\nu_o$.
	There exists $t_0 \geq 0$ such that for every $t \geq t_0$ and $u \geq 0$, we have $\theta_0(t+u) \leq e^{\varepsilon u}\theta_0(t)$.
	Define
	\begin{equation*}
		F = \set{\gamma \in \Gamma}{\dist o{\gamma o} < t_0}.
	\end{equation*}
	Now, let   $T \geq t_0 + r$ and  $f \in C^+(\bar X)$ be a map supported in  $U^K_T$. 
	We may assume that $\norm[\infty] f = 1$.
	Let $s > h_\Gamma$.
	It follows from (\ref{eqn: measure of UTK}) that
	\begin{equation}
		\label{eqn: measure of UTK - preintegral f}
		L_{o,s}(f)
		\leq
		\frac 1{\mathcal P'_\Gamma(s)} \sum_{\substack{\beta \in S\Gamma_k, \\ \dist o{\beta o} \geq T-r}} \sum_{\substack{\gamma \in \Gamma \\ \gamma o \in \mathcal O_o(\beta o,r)}} \theta_0(\dist o{\gamma o})e^{-s\dist{\beta o}{\gamma o}}.
	\end{equation}
	Let $\beta \in S\Gamma_k$ such that $\dist o{\beta o} \geq T - r$.
	We are going to estimate the second sum appearing above.
	For every $y \in \mathcal O_o(\beta o,r)$
	we have
	\begin{equation*}
		\dist o{\beta o} + \dist{\beta o}y- 2r
		\leq \dist oy
		\leq \dist o{\beta o} + \dist{\beta o}y.
	\end{equation*}
	Moreover if $\dist {\beta o}y \geq t_0 $, then
	\begin{equation*}
		\theta_0(\dist oy) \leq e^{\varepsilon \dist o{\beta o}} \theta_0(\dist {\beta o}y),
	\end{equation*}
	otherwise, since $\dist o{\beta o} \geq t_0$, we get
	\begin{equation*}
		\theta_0(\dist oy) \leq e^{\varepsilon [\dist o{\beta o} + \dist {\beta o}y - t_0 ]}\theta_0(t_0) \leq e^{\varepsilon \dist o{\beta o}}\theta_0(t_0).
	\end{equation*}
	Consequently
	\begin{equation*}
		\sum_{\substack{\gamma \in \Gamma \\ \gamma o \in \mathcal O_o(\beta o,r)}} \theta_0(\dist o{\gamma o})e^{-s\dist{\beta o}{\gamma o}}
		\leq  e^{2sr}e^{-(s-\varepsilon) \dist o{\beta o}}\left(\Sigma_1 + \Sigma_2\right),
	\end{equation*}
	where
	\begin{align*}
		\Sigma_1 & = \sum_{\substack{\gamma \in \Gamma \\ \gamma o \in \mathcal O_o(\beta o,r),\ \dist{\beta o}{\gamma o} < t_0}} \theta_0(t_0)e^{-s\dist{\beta o}{\gamma o}}, \\
		\Sigma_2 & = \sum_{\substack{\gamma \in \Gamma \\ \gamma o \in \mathcal O_o(\beta o,r),\ \dist{\beta o}{\gamma o} \geq t_0}} \theta_0(\dist {\beta o}{\gamma o})e^{-s\dist{\beta o}{\gamma o}}.
	\end{align*}
	The number of terms in $\Sigma_1$ is at most $\card F$, so that
	$\Sigma_1 \leq \card F \theta(t_0)$.
	On the other hand  $\Sigma_2$ is bounded from above by $\mathcal P'_\Gamma(s)L_{\beta o,s}(\mathbf 1)=\mathcal P'_\Gamma(s)$.
	Combining these inequalities, we get
	\begin{equation*}
		L_{o,s}(f) \leq C(s) \sum_{\substack{\beta \in S\Gamma_k, \\ \dist o{\beta o} \geq T-r}}e^{-(s-\varepsilon) \dist o{\beta o}},
	\end{equation*}
	with
	\begin{equation*}
		C(s) = e^{2sr}\left(1 + \frac{\card F \theta_0(t_0)}{\mathcal P'_\Gamma(s)}\right).
	\end{equation*}
	After passing to the limit, it becomes
	\begin{equation*}
		\int f d\nu_o \leq e^{2h_\Gamma r}\sum_{\substack{\beta \in S\Gamma_k, \\ \dist o{\beta o} \geq T - r}}e^{-(h_\Gamma-\varepsilon) \dist o{\beta o}}.
	\end{equation*}
	Recall that  $h_\Gamma-\varepsilon > h_{\Gamma_k}$. 
	Thus, we also get from (\ref{eqn: measure of UTK - remainder})
	\begin{equation*}
		\int f d\nu_o \leq B e^{2h_\Gamma r}e^{-(h_\Gamma - h_{\Gamma_k} - 2\varepsilon)T}.
	\end{equation*}
	Recall that $B$, $k$, $r$ and $\varepsilon$ do not depend on $T$ or $f$.
	The result follows.
\end{proof}

\begin{coro}[Compare with {\cite[Corollary~7.32]{Schapira:2018ti}}]
	\label{res: PS charges radial limit set}
	Assume that the action of $\Gamma$ on $X$ is strongly positively recurrent.
	There exists a compact subset $K$ of $X$ such that $\nu_o(\Lambda_{\rm{rad}}^K) = 1$.
\end{coro}

Recall that this conclusion is also true under the weaker assumption that $\nu_o$ gives full measure to $\Lambda_{\rm rad}$, as shown in \autoref{res: noatom}.

\begin{proof}
	Let $K$ be the compact subset of $X$ given by \autoref{res: measure of UTK}.
	By definition $(U^T_K)$ is a family of neighbourhoods of $\mathcal L_K$.
	Since $\nu_0$ is inner regular, it follows from \autoref{res: measure of UTK}, that $\nu_o(\mathcal L_K)=0$.
	Therefore $\nu_o(\Gamma.\mathcal L_K)=0$, whence $\nu_o(\Lambda_{\rm rad}^K)=1$.
\end{proof}

%
\section{Ergodicity of the Bowen-Margulis current}
%
\label{sec: BMCurrent}

The aim of this section is to prove the following theorem, which will be of crucial importance in the proof of \autoref{res: main}.

\begin{theo}
\label{res: ergodicity BM}
	Let $\Gamma$ be a discrete group with a proper strongly positively recurrent  action  on a Gromov-hyperbolic space $X$. 
	Then the diagonal action of $\Gamma$ on $\partial^2 X$ is ergodic with respect to the Bowen-Margulis current $\mu$.
\end{theo}

The statement is well-known if $X$ is a CAT($-1$) space.
Indeed the Hopf-Tsuji-Sullivan Theorem states that the action of $\Gamma$ on $(\partial^2X, \mu)$ is ergodic if and only if the Patterson-Sullivan measure $\nu_o$ gives full measure the the radial limit set $\Lambda_{\rm rad}$ \cite[Chapter 1]{Roblin:2003vz}.
The proof goes through the ergodicity of the geodesic flow on the quotient space $X/\Gamma$ with respect to the Bowen-Margulis measure.
A key ingredient is the exponential contraction/expansion of the geodesic flow along stable/unstable manifolds.

The reader  may  know that when $X$ is Gromov hyperbolic, the definition of a good geodesic flow may be a problem. 
A option to bypass this difficulty would be to use the construction of either Gromov \cite{Gro87,Champetier:1994us} or Mineyev \cite{Mineyev:2005fda}.
They both define a metric geodesic flow, with the needed exponential contraction/expansion properties. 
However, the statements available in the literature require some additional assumptions.
Although it is likely that the proof would adapt to strongly positively recurrent actions on Gromov hyperbolic spaces, it is probably a long technical work.

Instead, we follow with little variations the strategy developed by Bader and Furman for hyperbolic groups \cite{Bader:2017te}.
We first define a measurable action of $\Gamma$ on the abstract space $\partial^2X \times \R$ and then prove a version of the Hopf-Tsuji-Sullivan theorem, involving the \og geodesic flow \fg\ on the quotient $(\partial^2X\times\R)/\Gamma$ (\autoref{res: hopf tsuji}).
In this approach the contraction property of the geodesic flow is replaced by a contraction property for the action of $\Gamma$ on the boundary $\partial X$.

%
\subsection{The measurable geodesic flow of Bader-Furman}
%
\label{sec: bader furman flow}

\paragraph{The space of the geodesic flow.}
Recall that $o \in X$ is a fixed base point.
The space $X$ does not come with a well behaved geodesic flow.
Instead we consider the abstract topological space
\begin{equation*}
	SX = \partial^2 X \times \R,
\end{equation*}
and denote by $\mathcal B$  its Borel $\sigma$-algebra.
As suggested by the notation, it should be thought as the analogue of the unit tangent bundle of $X$.
For this reason we slightly abuse terminology 
by calling the elements of $SX$ \emph{vectors}.

\paragraph{Measure, flow and action.}
Recall that $(\nu_x)$ is the $\Gamma$-invariant
$h_\Gamma$-quasi-conformal Patterson-Sullivan density on $\partial X$
and $\mu$ is the associated $\Gamma$-invariant Bowen-Margulis current, which belongs to the same measure class as $\nu_o \otimes \nu_o$.
The Bowen-Margulis measure on $SX$ is defined as the product measure
\begin{equation*}
	m = \mu \otimes dt,
\end{equation*}
where $dt$ is the Lebesgue measure on $\R$.
The translation on the $\R$ component defines a flow $(\phi_t)_{t \in \R}$ on $SX$ which preserves the Bowen-Margulis measure $m$.
 If $v = (\eta, \xi,t)$ is vector of $SX$, the points $\eta$ and $\xi$ are its respective (asymptotic) \emph{past} and \emph{future}.
We now endow $SX$ with a measurable $\Gamma$-action.
To that end, we define a map $\beta \colon \Gamma \times \partial X \to \R$ by
\begin{equation*}
	\beta(\gamma, \xi) = h_\Gamma^{-1} \ln \left( \frac{d \gamma^{-1}_\ast \nu_o}{d\nu_o} (\xi)\right).
\end{equation*}
It satisfies the following cocycle relation $\nu_o$-a.s.
\begin{equation}
	\label{eqn: flow2 - cocyle RN derivative}
	\beta(\gamma_2\gamma_1, \xi) = \beta(\gamma_2, \gamma_1\xi) + \beta(\gamma_1, \xi),\quad  \forall \gamma_1,\gamma_2 \in \Gamma.
\end{equation}
As $(\nu_x)$ is the push-forward by $\pi \colon \partial_h X \to \partial X$ of the $h_\Gamma$-conformal Patterson-Sullivan density $(\tilde \nu_x)$   on $\partial_h X$, for any  cocycle  $b \in \pi^{-1}(\xi)$ and $\gamma \in \Gamma$,
\begin{equation}
	\label{eqn: flow2 - comparison cocycle}
	\abs{ \beta(\gamma, \xi) - b(\gamma^{-1}o,o)} \leq 100\delta.
\end{equation}
For every point $v = (\eta, \xi, t)$ in $SX$ and any element $\gamma \in \Gamma$, define
\begin{equation}
	\label{eqn: flow2 - def Gamma action}
	\gamma v = \left( \gamma \eta, \gamma \xi , t + \kappa_\gamma (\eta, \xi)\right),
	\quad \text{where} \quad
	\kappa_\gamma (\eta, \xi) = \frac {\beta(\gamma, \xi) - \beta(\gamma, \eta)}2 .
\end{equation}
It defines a measurable action of $\Gamma$ on $SX$ which preserves the Bowen-Margulis measure $m$.  
Indeed, by (\ref{eqn: flow2 - cocyle RN derivative}), as $\Gamma$ is countable,   the set
\begin{equation}
	\label{eqn: flow2 - full measure inv subset}
	 S_0 X  = \set{v \in SX}{\forall \gamma_1, \gamma_2 \in \Gamma,\ \gamma_1(\gamma_2v) = (\gamma_1\gamma_2)v}
\end{equation}
is a $\Gamma$-invariant Borel subset of $SX$, with full $m$-measure. 
The action of $\Gamma$ commutes with the flow $(\phi_t)$.
In particular the set $S_0 X$ defined above is invariant under the flow $(\phi_t)$.

\paragraph{Quotient space.}
By analogy with the Riemannian  setting, we wish to study the geodesic flow on the quotient space $SX/\Gamma$, viewed as an analogue of the unit tangent bundle on $M = X/\Gamma$. 
The action of $\Gamma$ on $SX$ is only a measurable action, but we could  work in the quotient space $S_0 X/\Gamma$ where $S_0 X$ is the subset defined in (\ref{eqn: flow2 - full measure inv subset}).
We prefer to use a slightly different approach and keep working in $SX$.
Let $\mathcal B_\Gamma$, be the sub-$\sigma$-algebra of all Borel subsets which are $\Gamma$-invariant (up to measure zero).
Let $D$ be a Borel fundamental domain for the action of $\Gamma$ on $SX$.
We endow $(SX, \mathcal B_\Gamma)$ with the restriction $\bar m$ of the measure $m$ to $D$.
More precisely, for every $B \in \mathcal B_\Gamma$, we let
\begin{equation*}
	\bar m (B) = m \left( B \cap D\right).
\end{equation*}
This definition of $\bar m$ does not depend on the choice of $D$.
As $\Gamma$ is countable we observe that $\bar m(B) = 0$ if and only if $m(B) = 0$.
Since the flow $(\phi_t)$ commutes with the action of $\Gamma$, it induces a measure preserving flow on $(SX, \mathcal B_\Gamma, \bar m)$.
We think of this new dynamical system as the geodesic flow on $SX/\Gamma$.

\paragraph{The Hopf-Tsuji-Sullivan theorem.}
\autoref{res: ergodicity BM} is a direct consequence of \autoref{res: PS charges radial limit set} and the following statement. 

\begin{theo}[Hopf-Tsuji-Sullivan theorem on $\delta$-hyperbolic spaces] 
\label{res: hopf tsuji}
	Let $\Gamma$ be a discrete group acting  properly by isometries
 on a Gromov hyperbolic space $X$. 
	The following assertions are equivalent.
	\begin{enumerate}
		\item \label{enu: hopf tsuji - rad}
		The Patterson-Sullivan  measure  $\nu_o$ only charges the radial limit set.
\item  \label{enu: hopf tsuji - conservative}
The geodesic flow on $(SX, \mathcal B_\Gamma, \bar m)$ is conservative.
\item  \label{enu: hopf tsuji - geo flow}
 The geodesic flow on $(SX, \mathcal B_\Gamma, \bar m)$ is ergodic.
\item  \label{enu: hopf tsuji - boundary}
The diagonal action of $\Gamma$ on $(\partial^2 X, \mu)$ is ergodic.
\end{enumerate}
Moreover, if any of these assertions is satisfied, then $\Gamma$ is divergent.
\end{theo}

\paragraph{Remark.}
If $X$ is CAT($-1$), Roblin shows  that the above items are \emph{equivalent} to the divergence of the group $\Gamma$ \cite[Chapter 1]{Roblin:2003vz}. 
His proof would adapt to our setting, but is long and useless for our purpose, so we omit it here.

\medskip
Note that the equivalence (\ref{enu: hopf tsuji - geo flow})~$\Leftrightarrow$~(\ref{enu: hopf tsuji - boundary}) follows immediately from the definition. 
We will see in \autoref{sec: conservative} that (\ref{enu: hopf tsuji - rad})~$\Leftrightarrow$~(\ref{enu: hopf tsuji - conservative}) is also rather easy.
As $\Gamma$ is non-elementary, the Bowen-Margulis measure is not supported on a single orbit, so that  (\ref{enu: hopf tsuji - geo flow})~$\Rightarrow$~ (\ref{enu: hopf tsuji - conservative}), see  \cite[Proposition~1.2.1]{Aaronson:1997gv}.
The core of the proof is (\ref{enu: hopf tsuji - conservative})~$\Rightarrow$~(\ref{enu: hopf tsuji - boundary}), which is shown in \autoref{sec: Hopf}.

\paragraph{Projection from $SX$ to $X$.}
In order to prove the Hopf-Tsuji-Sullivan Theorem, we need to relate the dynamical properties of the abstract space $SX$ to the geometry of the original space $X$.
To that end, we build a ``projection'' map $\proj \colon SX \to X$ as follows.
For every $(\eta, \xi) \in \partial^2X$ we choose first a bi-infinite geodesic  $\sigma_{(\eta,\xi)} \colon \R \to X$  joining $\eta$ to $\xi$.
Without loss of generality we can assume that $\sigma_{(\xi,\eta)}$ is   obtained from $\sigma_{(\eta, \xi)}$ by reversing the orientation.
The image $\proj(v)$ of a vector $v = (\eta, \xi,t)$ in $SX$ is now defined as the unique point $x$ on $\sigma_{(\eta,\xi)}$ such that
\begin{equation*}
	\frac 12 \left[ b^+_\xi(o,x) - b^-_\eta(o,x) \right] = t,
\end{equation*}
where $b^+_\xi$ and $b^-_\eta$ stand for the Busemann cocycle along $\sigma_{(\eta,\xi)}$  and  $\sigma_{(\xi,\eta)}$ respectively.
This definition of $\proj(v)$ involves many choices.
However, any another choice would lead to a point $x'$ such that $\dist x{x'} \leq 100\delta$.
It is a standard exercise of hyperbolic geometry to prove that for every vector $v = (\eta, \xi, t)$ in $SX$,
\begin{equation}
	\label{eqn : flow2 - distance projection}
	\abs{\gro \eta\xi o + \abs t - \dist o{\proj(v)}} \leq 20\delta.
\end{equation}
It follows from the construction that for every $v = (\eta, \xi, t)$ in $SX$ the map
\begin{equation*}
	\begin{array}{ccc}
		\R & \to & X                     \\
		s  & \mapsto & \proj \circ \phi_s(v)
	\end{array}
\end{equation*}
is a (up to changing the origin) the bi-infinite geodesic $\sigma_{(\eta, \xi)}$. 
The projection $\proj \colon SX \to X$ is not $\Gamma$-invariant in general.
However, for every $v \in SX$, for every $\gamma \in \Gamma$,
\begin{equation}
	\label{eqn: flow2 - proj equiv}
	\dist{ \gamma \proj(v)}{\proj(\gamma v)} \leq 200\delta.
\end{equation}
Combined with (\ref{eqn : flow2 - distance projection}) we get the following useful estimate.
For every $v = (\eta, \xi, t)$ in $SX$, for every $\gamma \in \Gamma$,
\begin{equation}
	\label{eqn: flow2 - proj equiv + dist}
	\abs{ \gro {\gamma \eta}{\gamma \xi}o + \abs{t + \kappa_\gamma(\eta, \xi)} - \dist o{\gamma \proj(v)}} \leq 220\delta,
\end{equation}
where $\kappa_\gamma(\eta, \xi)$ has been defined in \eqref{eqn: flow2 - def Gamma action}.

%
\subsection{Changing spaces}
%

We use the strategy of Bader-Furman to go back and forth between the spaces  $(\partial^2X, \mu)$, $(SX, \mathcal B,m)$ and $(SX, \mathcal B_\Gamma, \bar m)$.
We now work at the level of function spaces.
We consider first the following operation
\begin{equation}\label{eqn : ergo2 - tensor product}
	\begin{array}{ccc}
		L^1(\partial^2X,\mu) \times L^1(\R,dt) & \to & L^1(SX,m) \\
		(f,\vartheta)                             & \mapsto & f_\vartheta
	\end{array}
\end{equation}
where $f_\vartheta = f\otimes \vartheta$, i.e. for every $(\eta,\xi,t) \in SX$,
$\displaystyle 	f_\vartheta(\eta, \xi,t) = f(\eta,\xi) \vartheta(t)$. 

\medskip
Let $f \in L^1_+(SX,m)$ be a non-negative summable function.
We define a $\Gamma$-invariant function $\hat f \colon SX \to \R_+ \cup \{\infty\}$ by
\begin{equation}
	\label{eqn: ergo2 - averaging function}
	\hat f(v) = \sum_{\gamma \in \Gamma} f(\gamma v).
\end{equation}
Recall that $D$ stands for a Borel fundamental domain for the action of $\Gamma$ on $SX$.
As $m$ is $\Gamma$-invariant, we have
\begin{equation}
	\label{eqn: ergo2 - averaging vs integration}
	\int_{SX} \hat f d \bar m
	= \sum_{\gamma \in \Gamma} \int_D (f\circ \gamma) dm
	= \sum_{\gamma \in \Gamma} \int_{\gamma D} f dm
	= \int_{SX} f dm.
\end{equation}
In particular, $\hat f\in L^1_+(SX,\bar{m})$.
It follows that the map $f\mapsto \hat{f}$ is a well defined isometric embedding  of $L^1(SX,\mathcal{B},m)$ into $L^1(SX,\mathcal{B}_\Gamma,\bar m)$.

%
\subsection{Conservativity of the flow} \label{sec: conservative}
%

This section is devoted to the proof of (\ref{enu: hopf tsuji - rad})~$\Leftrightarrow$~(\ref{enu: hopf tsuji - conservative}) in \autoref{res: hopf tsuji}.
For a precise definition of \emph{conservativity} we refer the reader to \cite{Hopf:1937kk,Aaronson:1997gv}.
In this article we will only use the following properties.
Assume that $T$ is an inversible measure preserving map acting on a Borel space $(Y, \mathcal B,m)$.
If for $m$-almost every $y \in Y$ there exists $B \in \mathcal B$ with $0 < m(B) < \infty$ such that 
\begin{equation*}
	\sum_{n = 1}^\infty \mathbf 1_B \circ T^n (y) = \infty,
\end{equation*}
then $T$ is conservative \cite[Proposition~1.1.6]{Aaronson:1997gv}.
Reciprocally, by Halmos' recurrence theorem \cite[Theorem~1.1.1]{Aaronson:1997gv}, if $T$ is conservative, then for every $B \in \mathcal B$, with $m(B) > 0$, for $m$-almost every $y \in Y$,
\begin{equation*}
	\sum_{n = 1}^\infty \mathbf 1_B \circ T^n (y) = \infty.
\end{equation*}
A measure preserving flow $(\phi_t)$ on $(Y,\mathcal B,m)$ is \emph{conservative} if its time-one map $T = \phi_1$ is conservative.

\medskip
For every $r \in \R_+$,  we define two subsets of $\partial^2X$ and $SX$ respectively by 
\begin{equation}
	Z(r) = \set{(\eta, \xi) \in \partial^2 X}{\gro \eta\xi o \leq r},
	\quad \text{and} \quad 
	B(r)=Z(r) \times [0,1].
\end{equation}
Note that $B(r)$ need not be a compact subset of $SX$ (the Gromov product is not necessarily continuous).
Still it has positive finite $m$-measure.

\begin{lemm}
	\label{res: ergo2 - conservativity preparation} 	
For $\bar m$-almost every vector $v = (\eta, \xi, t)$ in $SX$, the future $\xi$ of $v$ belongs to the radial limit set $\Lambda_{\rm rad}$ if and only if there exists $r \in \R_+^*$ such that 
	\begin{equation}
		\label{eqn: ergo2 - conservativity preparation}
		\int_0^\infty \mathbf 1_{\Gamma   B(r)}\circ \phi_s(v) ds = \infty.
	\end{equation}
\end{lemm}

\begin{proof}	
	Recall that $S_0 X$ is the  $\Gamma$- and flow-invariant subset of $SX$ of full measure given in (\ref{eqn: flow2 - full measure inv subset}).
	Let $v = (\eta, \xi,t)$ be a vector in $S_0 X$.
	As we noticed earlier, the path $\sigma \colon \R \to X$ sending $s$ to $\proj \circ \phi_s(v)$ is a bi-infinite geodesic joining $\eta$ to $\xi$.
	Assume first that $\xi$ belongs to $\Lambda_{\rm{rad}}$.
	There exists $r\in \R_+^*$, and an infinite sequence $(\gamma_n)$ of elements of $\Gamma$ such that $(\gamma_n o)$ converges to $\xi$ and for every $n \in \N$, we have $\gro \eta\xi {\gamma_n o} \leq r$.
	Set
	\begin{equation*}
		s_n = -t - \kappa_{\gamma_n^{-1}}(\eta, \xi).
	\end{equation*}
	so that for all $u \in [0,1]$, the vector $\gamma_n^{-1}\phi_{s_n+u}(v)$ belongs to $B(r)$.
	Since $v\in S_0X$, the vector $\phi_{s_n+u}(v)$ belongs to $\gamma_n B(r)$, hence to $\Gamma B(r)$.
	By (\ref{eqn: flow2 - proj equiv + dist}), the point $\proj\circ \phi_{s_n}(v)$ is approximatively a projection of $\gamma_n o$ onto $\sigma$.
	As $(\gamma_n o)$ converges to $\xi$, the sequence $(s_n)$ diverges to infinity.
	Consequently the positive orbit of $v$ spends an infinite amount of time in $\Gamma B(r)$, whence
	\begin{equation*}
		\int_0^\infty \mathbf 1_{\Gamma B(r)}\circ \phi_s(v) ds = \infty.
	\end{equation*}
	Conversely assume that there exists $r \in \R_+^*$ such that the above equality holds.
	As $v$ belongs to $S_0X$, it means that there exists a sequence $(s_n)$ diverging to infinity as well as a sequence $(\gamma_n)$ of elements of $\Gamma$ such that $\gamma_n^{-1}\phi_{s_n}(v)$ belongs to $B(r)$ for every $n \in \N$.
	Hence $\gro \eta\xi{\gamma_no} \leq r$.
	By (\ref{eqn: flow2 - proj equiv + dist}) we also get that $\dist {\gamma_no}{\proj \circ \phi_{s_n}(v)}$ is uniformly bounded.
	Consequently $(\gamma_no)$ converges to $\xi$, hence $\xi$ belongs to $\Lambda_{\rm rad}$.
\end{proof}

\begin{prop}
	\label{res: ergo2 - conservativity}
	The Patterson-Sullivan measure $\nu_o$ gives full support to the radial limit set if and only if the flow $(\phi_t)$ on $(SX, \mathcal B_\Gamma, \bar m)$ is conservative.
\end{prop}

\begin{proof}
	The measure $\nu_0$ gives full measure to the radial limit set 
if and only if the measures $m$ and thus $\bar m$ give full measure to
	\begin{equation*}
		 \set{(\eta,\xi,t)\in SX}{(\eta, \xi) \in \Lambda_{\rm rad}\times \Lambda_{\rm rad}}.
	\end{equation*}
	 In view of the properties of conservative systems recalled above, \autoref{res: ergo2 - conservativity preparation} tells us that the $\nu_0$ gives full measure to the radial limit set if and only if the flow $(\phi_t)$ is conservative.  
\end{proof}

%
\subsection{The Hopf argument}\label{sec: Hopf}
%

This section is devoted to  the implication (\ref{enu: hopf tsuji - conservative})~$\Rightarrow$~(\ref{enu: hopf tsuji - boundary})  of \autoref{res: hopf tsuji}, proven in \autoref{res: conservative implies double ergodic}.

\paragraph{Preliminaries.}
We  fix once for all a number $a > 2h_\Gamma$ and a map
\begin{equation*}
	\begin{array}{rccc}
		\vartheta \colon & \R & \to & \R                       \\
		              & t  & \mapsto & \frac 12 a e^{-a\abs t}.
	\end{array}
\end{equation*}
For every $T_1,T_2 \in \R_+$, with $T_1 \leq T_2$, we write $\Theta_{T_1}^{T_2} \colon \R \to [0,1]$ for the map defined by
\begin{equation*}
	\Theta_{T_1}^{T_2}(u)
	= \int_{T_1}^{T_2} \vartheta(u + t)dt
	= \int_{T_1+u}^{T_2+u} \vartheta(t)dt.
\end{equation*}
This function is \og almost constant \fg on $[-T_2,-T_1]$ and decays exponentially outside this interval.
More precisely we have the following useful estimates.
\begin{enumerate}
	\item For every $u \in \R$,
	      \begin{equation}
		      \label{eqn : ergo2 - tail Theta}
		      \Theta_{T_1}^{T_2}(u) \leq \frac 12\min \left\{ e^{a(T_2+u)}, e^{-a(T_1+u)} \right\}.
	      \end{equation}
	\item For every $u \in [-T_2,-T_1]$,
	      \begin{equation*}
		      \Theta_{T_1}^{T_2}(u) = 1 - \frac 12\left[ e^{a(T_1+u)} + e^{-a(T_2+u)}\right].
	      \end{equation*}
	      Consequently, for every $u, u' \in [-T_2, -T_1]$,
	      \begin{equation}
		      \label{eqn : ergo2 - belly Theta}
		      \abs{\Theta_{T_1}^{T_2}(u) - \Theta_{T_1}^{T_2}(u')}
		      \leq \left[ e^{a\left(T_1 + \frac{u+u'}2\right)} +e^{-a\left(T_2 + \frac{u+u'}2\right)} \right]\sinh \left( \frac a2 \abs{u - u'}\right).
	      \end{equation}
\end{enumerate}
See \autoref{fig: graph Theta} for a sketch of the graph of $\Theta_{T_1}^{T_2}$.

\begin{figure}[htbp]
\centering
	\includegraphics[page=3, width = \linewidth]{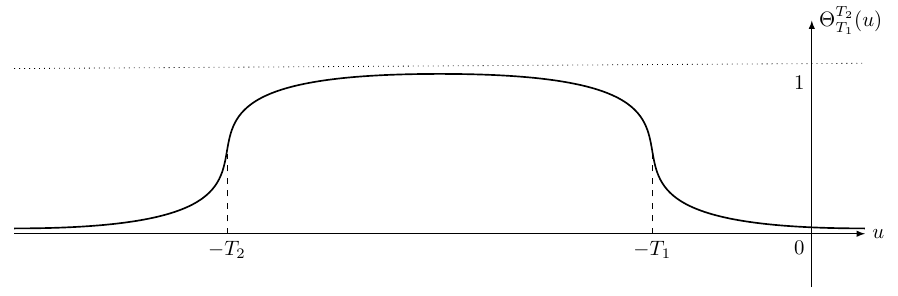}
	\caption{Graph of the map $\Theta_{T_1}^{T_2}$.}
	\label{fig: graph Theta}
\end{figure}

\medskip
We say that a function $f \colon \partial^2X \to \R$ has \emph{exponential decay} if there exists $C \in \R_+$ such that for every $\mu$-almost every ${(\eta, \xi) \in \partial^2 X}$ we have
\begin{equation}
	\label{eqn: ergo2 - def fast decay}
	\abs{f(\eta,\xi)} \leq Ce^{-a\gro\eta \xi o}.
\end{equation}
Any such function belongs to $L^1(\mu)$.
Indeed, as $a > 2h_\Gamma$, Inequality~(\ref{eqn: bm product}) yields
\begin{equation*}
	\int \abs{f} d\mu \leq CC_0 \int e^{(2h_\Gamma-a)\gro \eta\xi o}  d\nu_o(\eta)d\nu_o(\xi) \leq CC_0.
\end{equation*}
Recall that the boundary $\partial X$ is endowed with a visual metric $\distV[\partial X]$ for which there exists $a_0, \varepsilon_0 \in (0,1)$ such that for every $\eta, \xi \in \partial X$, we have
\begin{equation}
	\label{eqn: ergo2 - recall visual metric}
	\abs{\dist[\partial X] \eta \xi + a_0 \gro \eta \xi o} \leq \varepsilon_0.
\end{equation}
The product metric induces a distance on $\partial^2X$.
We write $\mathcal D^+(\partial^2X)\subset L^1_+(\mu)$ for the set of all Lipschitz functions $f \colon \partial^2X \to \R_+$ with exponential decay.

\medskip
We complete this preliminary discussion with the following easy but useful statements.

\begin{lemm}
	\label{res: ergo2 - prelim card ball}
	There exists $C \in \R_+$ such that for every $x \in X$, for every $r \in \R_+$,
	\begin{equation*}
		\card{\set{\gamma \in \Gamma}{\dist o{\gamma x} \leq r}} \leq Ce^{2h_\Gamma r}.
	\end{equation*}

\end{lemm}

\begin{proof}
	If $\dist x{\Gamma o} > r$, then $\set{\gamma \in \Gamma}{\dist o{\gamma x} \leq r}$ is empty.
	If $\dist x{\Gamma o} \leq r$, there exists $\alpha \in \Gamma$ such that $\dist x{\alpha o} = \dist x{\Gamma o}$.
	Triangle inequality implies that $\set{\gamma \in \Gamma}{\dist o{\gamma x} \leq r}$ is contained in $B_\Gamma(o,2r) \alpha^{-1}$, where
	\begin{equation*}
		B_\Gamma(o,2r) = \set{ \gamma \in \Gamma}{\dist o{\gamma o} \leq 2r}.
	\end{equation*}
	It is well-known that
	$\card{B_\Gamma(o,2r)} \leq Ce^{2h_\Gamma r}$,
	for some   universal constant $C$,  see for instance \cite[Corollaire~6.8]{Coo93}, which completes the proof.
\end{proof}

\begin{lemm}
	\label{eqn: ergo2 - prelim control Poincare series}
	There exists $C \in \R_+$ such that for every $x \in X$, for every $r \in \R_+$,
	\begin{equation*}
		\sum_{\substack{\gamma \in \Gamma,\\ \dist o{\gamma x} \geq r}} e^{-a \dist o{\gamma x}}
		\leq C e^{- (a - 2h_\Gamma)r}.
	\end{equation*}
\end{lemm}

\begin{proof}
	Let $x \in X$ and $r \in \R_+$.
	We split the sum as follows:
	\begin{equation*}
		\sum_{\substack{\gamma \in \Gamma,\\ \dist o{\gamma x} \geq r}} e^{-a \dist o{\gamma x}}
		\leq
		\sum_{\substack{\ell \in \N,\\ \ell \geq r}} \card{\set{\gamma \in \Gamma}{ \ell \leq \dist o{\gamma x} \leq \ell + 1}} e^{-a\ell}.
	\end{equation*}
	By \autoref{res: ergo2 - prelim card ball}, there exists $C \in \R_+$ (independent of $x$ and $r$) such that
	\begin{equation*}
		\sum_{\substack{\gamma \in \Gamma,\\ \dist o{\gamma x} \geq r}} e^{-a \dist o{\gamma x}}
		\leq C \sum_{\substack{\ell \in \N,\\ \ell \geq r}} e^{-(a-2h_\Gamma)\ell}.
	\end{equation*}
	As $a > 2h_\Gamma$, we get (up to changing the constant $C$)
	\begin{equation*}
		\sum_{\substack{\gamma \in \Gamma,\\ \dist o{\gamma x} \geq r}} e^{-a \dist o{\gamma x}}
		\leq C e^{-(a-2h_\Gamma)r}. \qedhere
	\end{equation*}
\end{proof}

\begin{lemm}
\label{res: f-bornee} 
	If $f\in \mathcal D^+(\partial^2X)$, then $\hat{f}_\vartheta$ is bounded, where $\hat f_\vartheta$ was defined in \eqref{eqn : ergo2 - tensor product} and \eqref{eqn: ergo2 - averaging function}.
\end{lemm}

\begin{proof}
	As $f$ has exponential decays, there exists $C \in \R_+$ such that for every $v \in SX$,
	\begin{equation*}
		\hat f_\vartheta(v)
		\leq C \sum_{\gamma\in\Gamma}e^{-a \gro {\gamma \xi}{\gamma \eta}o}e^{-a\abs{t+\kappa_{\gamma}(\xi,\eta)}}.
	\end{equation*}
	Set $x = \proj(v)$.
	Recall that by $\gro {\gamma \xi}{\gamma \eta}o + \abs{t+\kappa_{\gamma}(\xi,\eta)}$ is approximatively the distance between $o$ and $\gamma x$, see (\ref{eqn: flow2 - proj equiv + dist}).
	Up to increasing $C$, we get
	\begin{equation*}
		\hat f_\vartheta(v)\leq C \sum_{\gamma \in \Gamma}e^{-a d(o,\gamma x)},
	\end{equation*}
	Recall that $a > 2h_\Gamma$. 
	It follows from  \autoref{eqn: ergo2 - prelim control Poincare series} that this last sum is bounded independently of $v$.
\end{proof}

\paragraph{Contraction property.}
In a CAT($-1$) space $X$, a key fact when running the Hopf argument is that two geodesic rays $\sigma, \sigma' \colon \R \to X$ with the same point at infinity satisfy a contraction property, namely there exists $u \in \R$, such that $t \to \dist{\sigma(t)}{\sigma'(t+u)}$ converges exponentially fast to zero. 
As a consequence, if $f:X\to \R$ is a H\"older continuous map, the difference
\begin{equation*}
	\int_0^T \left[f\left(\sigma(t)\right)-f\left(\sigma'(t)\right)\right]dt
\end{equation*}
converges when $T\to +\infty$. 
Exponential decay of the distance along asymptotic geodesics is no longer true when $X$ is Gromov hyperbolic.
Indeed two geodesic rays may have the same endpoint at infinity, but only stay at bounded distance one from the other.
In this setting, the contraction of geodesics is replaced by the following fact.

\begin{prop}[Contraction lemma]
\label{res: ergo2 - contraction lemma}
Let $f\in \mathcal D^+(\partial^2X)$.
Let $v = (\eta, \xi, 0)$ and $v' = (\eta', \xi,0)$ be two vectors of $SX$ with the same future.
The map
\begin{equation*}
\begin{array}{ccl}
\R_+ & \to & \R_+                                                                                                 \\
T    & \to & \displaystyle\int_0^T \left(\hat f_\vartheta \circ \phi_s(v) - \hat f_\vartheta\circ \phi_s(v') \right) ds
\end{array}
\end{equation*}
is bounded.
\end{prop}

\begin{proof}
	Note that since $f$ is continuous, the function $\hat f_\vartheta$ is defined everywhere (and not just $\bar m$-almost everywhere).
	Recall that the map $\sigma \colon \R \to X$ sending $s$ to $\proj\circ \phi_s(v)$ is a bi-infinite geodesic joining $\eta$ to $\xi$.
	Similarly, using the vector $v'$, we define a geodesic $\sigma' \colon \R \to X$ from $\eta'$ to $\xi$.
We start by a defining a time shift, to make sure that $\sigma$ and $\sigma'$ fellow travel.

	By hyperbolicity of $X$,  there exists $u, T_0 \in \R$, such that for every $s \geq T_0$, we have $\dist{\sigma(s)}{\sigma'(s + u)} \leq 16\delta$ \cite[Chaptitre~7, Proposition~2]{Ghys:1990ki}.
	To have enough flexibility, we let $T_1 = T_0 + 10^{10}\delta$.
	As we already observed from (\ref{eqn : flow2 - distance projection}) the point $\sigma(0) = \proj(v)$ is approximately a projection of $o$ on $\sigma$.
	Similarly, by (\ref{eqn: flow2 - proj equiv + dist}) for every $\gamma \in \Gamma$, $\sigma(- \kappa_\gamma (\eta, \xi))$ is approximately a projection of $\gamma^{-1}o$ on $\sigma$.
	The same interpretation holds for $\sigma'$.
	It follows that for every $\gamma \in \Gamma$ such that $\kappa_\gamma(\eta, \xi) \leq - T_1 + 5000\delta$ or $u + \kappa_\gamma(\eta',\xi') \leq -T_1 + 5000\delta$, the following holds
	\begin{enumerate}
		\item $\abs{\gro {\eta'}{\xi}{\gamma^{-1}o} - \gro \eta \xi{\gamma^{-1}o}} \leq 500\delta$,
		\item $\abs{\kappa_\gamma(\eta', \xi) + u - \kappa_{\gamma}(\eta,\xi)} \leq 2000\delta$,
		\item $\gro \eta{\eta'}{\gamma^{-1}o} \geq - \kappa_\gamma(\eta, \xi) - T_1 $.
	\end{enumerate}
	This general configuration is sketched on \autoref{fig: ergo2 - contraction lemma - 1}.
	
	\begin{figure}[htbp]
	\centering
		\includegraphics[page=1, width=\linewidth]{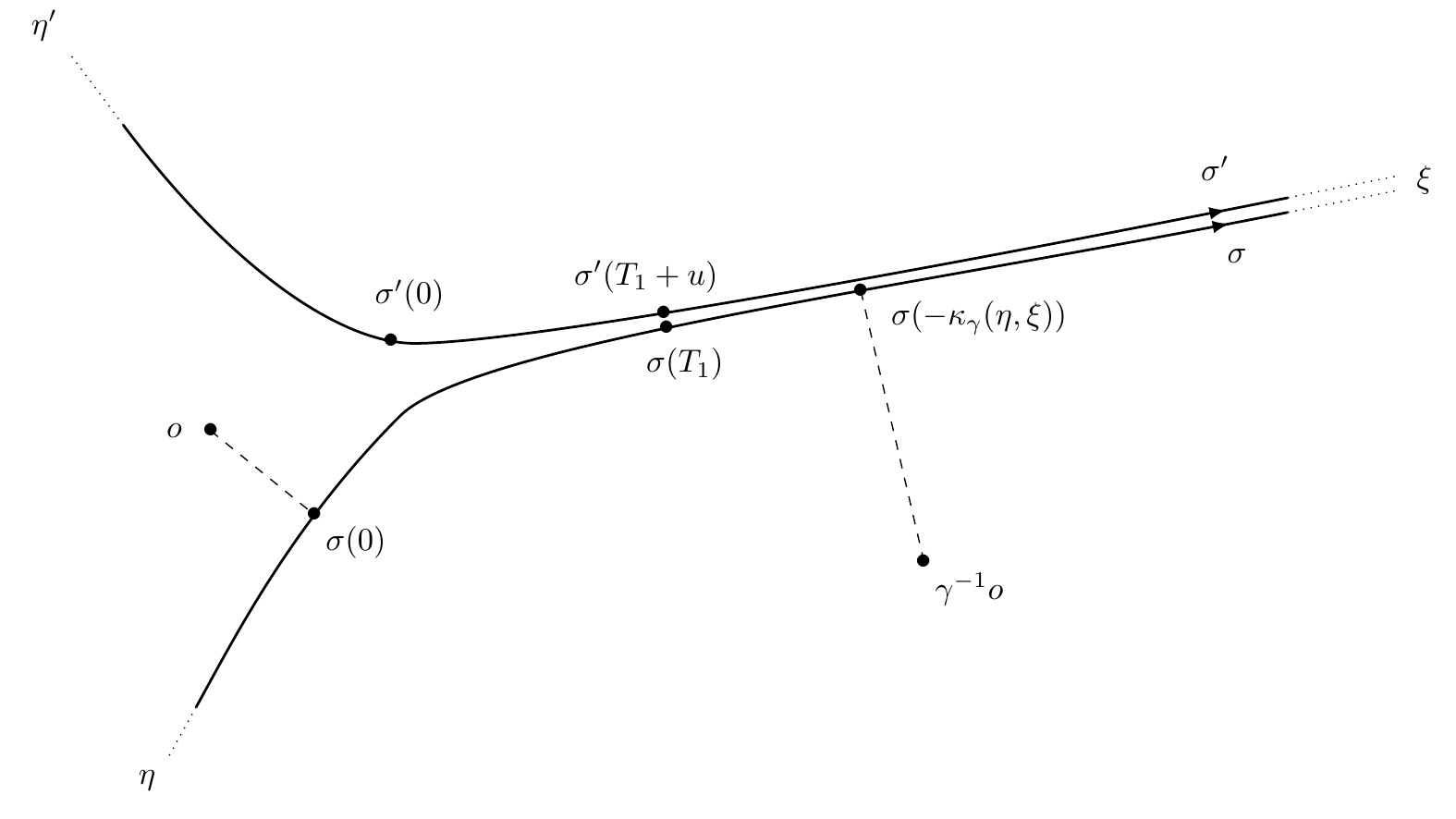}
	\caption{General configuration of $\sigma$ and $\sigma'$.}
	\label{fig: ergo2 - contraction lemma - 1}
	\end{figure}

	Since the map $\hat f_\vartheta$ is bounded (\autoref{res: f-bornee}) there exists $C_0 \in \R_+$, such that for every $T \geq T_0$,
	\begin{equation*}
		\abs{\int_0^T \left(\hat f_\vartheta \circ \phi_s(v) - \hat f_\vartheta\circ \phi_s(v') \right) ds
			-\int_{T_1}^T \left(\hat f_\vartheta \circ \phi_s(v) - \hat f_\vartheta\circ \phi_{s+u}(v') \right) ds} \leq C_0.
	\end{equation*}
	Set $v'_u = \phi_u(v')$ and define $F_T$ by
	\begin{equation*}
		\begin{array}{rccc}
			F_T \colon & SX & \to     & \R                                                         \\
			           & w  & \mapsto & \displaystyle\int_{T_1}^T \hat f_\vartheta \circ \phi_s(v)ds,
		\end{array}
	\end{equation*}
	To get \autoref{res: ergo2 - contraction lemma}, it suffices to show that the map $T\to F_T(v)-F_T(v'_u)$ is bounded.
	A Fubini argument gives
	\begin{equation}
		\label{eqn: ergo2 - contraction lemma - F as averaged function}
		F_T(w) = \sum_{\gamma \in \Gamma} f \otimes \Theta_{T_1}^T\left(\gamma w\right) = \widehat{f\otimes\Theta^T_{T_1}}(w).
	\end{equation}
	\autoref{fig: ergo2 - contraction lemma - 2} represents the value of $F_T$.
	
	\begin{figure}[htbp]
	\centering
		\includegraphics[page=2, width=\linewidth]{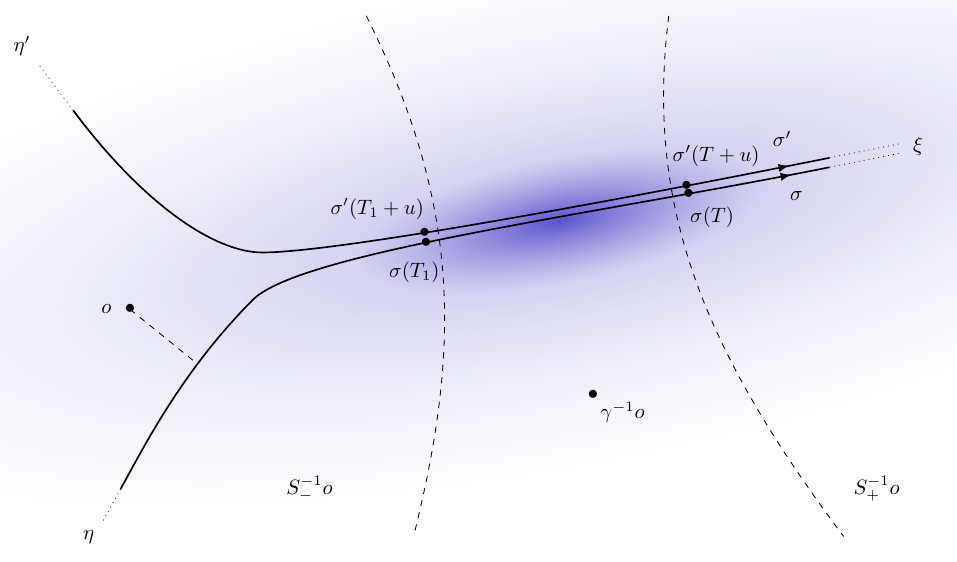}
	\caption{
		The function $F_T$.
		The shade represents the magnitude of $F_T$.
		Dark areas (\resp light) corresponds to vectors $w \in SX$ for which $\abs{F_T(w)}$ is large (\resp small).
		The dashed lines split the orbit $\set{\gamma ^{-1} o}{\gamma \in \Gamma}$ in three parts according to whether $\gamma$ belongs to $S_-$, $S_+$ or $\Gamma \setminus (S_- \cup S_+)$.
	}
	\label{fig: ergo2 - contraction lemma - 2}
	\end{figure}
	We are now going to decompose the sum in (\ref{eqn: ergo2 - contraction lemma - F as averaged function}) according to the value of $\kappa_\gamma (\eta, \xi)$.
	For every $t \in \R_+$, we define a subset $S(t)$ of $\Gamma$ as follows.
	\begin{align*}
		S(t) & = \set{\gamma \in \Gamma}{ -\delta < t + \kappa_\gamma(\eta, \xi) \leq  0  }.
	\end{align*}
	Roughly speaking $S(t)$ corresponds to the set of all elements $\gamma \in \Gamma$ such that the projection of $\gamma^{-1}o$ on $\sigma$ is approximatively $\sigma(t) = \proj \circ \phi_t(v)$.
	The sets $(S(n \delta))_{n \in \Z}$ form a partition of $\Gamma$.
	In particular
	\begin{equation}
		\label{eqn: ergo2 - contraction lemma - FT sum}
		F_T(w) = \sum_{n \in \Z} \sum_{\gamma \in S(n\delta)} f \otimes \Theta_{T_1}^T\left(\gamma w\right).
	\end{equation}
	The first lemma handles the tails of this sum.

	\begin{lemm}
		\label{res: ergo2 - tails FT}
		There exists $C_1 \in \R_+$ such that for every $T \geq T_1$, for every $w \in \{ v, v'_u\}$, we have
		\begin{equation*}
			\max\left\{
			\sum_{ n \delta \leq T_1 + 2000\delta} \sum_{\gamma \in S(n\delta)} f \otimes \Theta_{T_1}^T\left(\gamma w\right),
			\sum_{ n \delta \geq  T - 2001\delta} \sum_{\gamma \in S(n\delta)} f \otimes \Theta_{T_1}^T\left(\gamma w\right)
			\right\}\leq C_1.
		\end{equation*}
	\end{lemm}

	\begin{proof}
		Let $T \geq T_1$.
		Observe that
		\begin{equation*}
			\bigcup_{n\delta \geq T - 2001\delta} S(n\delta) = \set{ \gamma \in \Gamma}{\kappa_\gamma (\eta, \xi) \leq - T + 2001\delta}.
		\end{equation*}
		For simplicity we denote this set by $S_+$ (see \autoref{fig: ergo2 - contraction lemma - 2}).
		Similarly, set
		\begin{equation*}
			S_- = \bigcup_{ n \delta \leq T_1 + 2000\delta} S(n\delta) = \set{ \gamma \in \Gamma}{\kappa_\gamma (\eta, \xi) \geq  - T_1 - 2001\delta}.
		\end{equation*}
		We focus now on the right tail of $F_T(v)$.
		Recall that $f$ has exponential decay, whereas the tails of $\Theta_{T_1}^T$ decay exponentially -- see (\ref{eqn : ergo2 - tail Theta}).
		It follows that
		\begin{align*}
			\sum_{ n \delta \geq  T - 2001\delta} \sum_{\gamma \in S(n\delta)} f \otimes \Theta_{T_1}^T\left(\gamma v\right)
			 & \leq \sum_{\gamma \in S_+}f(\gamma \eta, \gamma \xi) \Theta_{T_1}^T\left(\kappa_\gamma (\eta, \xi) \right)                             \\
			 & \leq \frac 12\sum_{\gamma \in S_+} e^{-a \gro {\gamma \eta}{\gamma\xi}o} e^{a\left[T + \kappa_\gamma(\eta, \xi\right)]}.
		\end{align*}
		However $(T -2001\delta)+ \kappa_\gamma(\eta, \xi) \leq 0$, for every $\gamma \in \Gamma$.
		Consequently
		\begin{equation*}
			\gro {\gamma \eta}{\gamma\xi}o - \left[(T -2001\delta) + \kappa_\gamma(\eta, \xi)\right]
			= \gro {\gamma \eta}{\gamma\xi}o+ \abs{(T -2001\delta) + \kappa_\gamma(\eta, \xi)},
		\end{equation*}
		which, according to (\ref{eqn: flow2 - proj equiv + dist}), differs from $\dist o{\gamma \sigma(T-2001\delta)}$ by at most $\delta$.
		Hence there exists a constant $C$ (which does not depends on $T$) such that
		\begin{equation*}
			\sum_{ n \delta \geq  T - \delta} \sum_{\gamma \in S(n\delta)} f \otimes \Theta_{T_1}^T\left(\gamma v\right)
			\leq C \sum_{\gamma \in S_+} e^{-a\dist o{\gamma \sigma(T-2001\delta)}}
		\end{equation*}
		It follows from \autoref{eqn: ergo2 - prelim control Poincare series} that the latter sum is bounded from above independently of $T$.
		The upper bound for the left tail of $F_T(v)$ follows the exact same strategy.
		For the tails of $F_T(v'_u)$ we have to be slightly more careful.
		Indeed the sets $S(t)$ were defined according to $v$ (the definition involves its past $\eta$) and not $v'_u$.
		Nevertheless, as we observed at the beginning of the proof if either $\kappa_\gamma(\eta, \xi) \leq - T_1 + 5000\delta$ or $u + \kappa_\gamma(\eta',\xi') \leq -T_1 + 5000\delta$, then these two quantities differ by as most $2000\delta$.
		Consequently
		\begin{align*}
			S_- & \subset \set{ \gamma \in \Gamma}{u + \kappa_\gamma (\eta', \xi') \geq  - T_1 - 4001\delta}, \\
			S_+ & \subset \set{ \gamma \in \Gamma}{u + \kappa_\gamma (\eta', \xi') \leq - T + 4001\delta}.
		\end{align*}
		The estimation of the tails of $F_T(v'_u)$ now works as for the one of $F_T(v)$.
	\end{proof}

	The next step is to estimate in (\ref{eqn: ergo2 - contraction lemma - FT sum}) each sum over $S(n\delta)$ whenever $n\delta$ belongs to $[T_1+ 2000\delta, T- 2001\delta]$.

	\begin{lemm}
		\label{res: ergo2 - belly FT}
		There exists $C_2 \in \R_+$ such that for every $T \geq T_1$, for every $n \in \Z$ such that $T_1+ 2000\delta \leq n\delta \leq T - 2001\delta$, we have
		\begin{align*}
			 \Delta(n):= \MoveEqLeft{\sum_{\gamma \in S(n \delta)} \abs{f \otimes \Theta_{T_1}^T\left(\gamma v\right) - f \otimes \Theta_{T_1}^T\left(\gamma v'_u\right)}} \\
			 & \leq C_2\left[\left(e^{a(T_1-n\delta)}+ e^{-a(T- n \delta)}\right)+\frac 1{n^2} + n^q e^{-a_0n\delta}\right],
		\end{align*}
		where 
		\begin{equation*}
			q = \frac {4h_\Gamma}{a - 2h_\Gamma}.
		\end{equation*}
	\end{lemm}

	\begin{proof}
		Let $n \in \Z$, such that $T_1+ 2000\delta \leq n\delta \leq T - 2001\delta$.
		Observe that $\Delta(n) \leq h_F(n) + h_\Theta(n)$ where
		\begin{align*}
			h_F(n)      & = \sum_{\gamma \in S(n \delta)} \abs{f(\gamma \eta, \gamma \xi) - f(\gamma \eta', \gamma \xi)}\Theta_{T_1}^T\left(u + \kappa_\gamma(\eta', \xi')\right), \\
			h_\Theta(n) & =
			\sum_{\gamma \in S(n \delta)} f(\gamma \eta, \gamma \xi) \abs{\Theta_{T_1}^T\left(\kappa_\gamma(\eta, \xi)\right) - \Theta_{T_1}^T\left(u + \kappa_\gamma(\eta', \xi)\right)}.
		\end{align*}

		\medskip
		We start with the term $h_\Theta(n)$.
		Let $\gamma \in S(n\delta)$.
		As observed at the beginning of the proof, since $\kappa_\gamma(\eta, \xi) \leq -T_1$, this quantity differs from $u + \kappa_\gamma(\eta', \xi)$ by at most $2000\delta$.
		In particular $u + \kappa_\gamma(\eta', \xi')$ belongs to $[-n\delta - 2001\delta, -n \delta + 2000\delta]$, hence to $[-T,-T_1]$.
		On this interval the function $\Theta_{T_1}^T$ is almost constant.
		More precisely, using (\ref{eqn : ergo2 - belly Theta}) we get
		\begin{equation*}
			\abs{\Theta_{T_1}^T\left(\kappa_\gamma(\eta, \xi)\right) - \Theta_{T_1}^T\left(u + \kappa_\gamma(\eta', \xi')\right)} 
			\leq C\left( e^{a(T_1-n\delta)}+ e^{-a(T- n \delta)}\right),
		\end{equation*}
		for some parameter  $C$, which does not depends on $n$ or $T$.
		Consequently
		\begin{equation*}
			h_\Theta(n) \leq C\left( e^{a(T_1-n\delta)}+ e^{-a(T- n \delta)}\right) \sum_{\gamma \in S(n\delta)} f(\gamma \eta,\gamma \xi).
		\end{equation*}
		Recall that $-\delta \leq n\delta + \kappa_\gamma(\eta, \xi) \leq 0$, for every $\gamma \in S(n\delta)$.
		Consequently the latter sum can be bounded above as follows
		\begin{equation*}
			\sum_{\gamma \in S(n\delta)} f(\gamma \eta,\gamma \xi)
			\leq e^{a\delta} \sum_{\gamma \in S(n\delta)} f(\gamma \eta,\gamma \xi) e^{-a\abs{n\delta + \kappa_\gamma(\eta, \xi)}}.
		\end{equation*}
		Since $f$ decays exponentially, we prove as in \autoref{res: ergo2 - tails FT} that this sum is bounded from above independently of $n$ and $T$.
		To summarize, we have proved that there exists $C_\Theta \in \R_+$ (which does not depend on $n$ or $T$) such that
		\begin{equation}
			\label{res: ergo2 - belly FT - Theta diff}
			h_\Theta(n) \leq C_\Theta\left( e^{a(T_1-n\delta)}+ e^{-a(T- n \delta)}\right).
		\end{equation}

		\medskip
		Let us now focus on $h_F(n)$.
		First, as $\Theta_{T_1}^T\leq 1$, we have
		\begin{equation*}
			h_F(n) \leq  \sum_{\gamma \in S(n \delta)} \abs{f(\gamma \eta, \gamma \xi) - f(\gamma \eta', \gamma \xi)}.
		\end{equation*}
		We  split again this sum in two parts according to the value of $\gro \eta\xi{\gamma^{-1} o}$. More precisely, we set
		\begin{equation*}
			p = \frac 2{a - 2h_\Gamma},
		\end{equation*}
		and
		\begin{align*}
			S_0(n\delta)       & = \set{\gamma \in S(n\delta)}{ \gro \eta\xi{\gamma^{-1} o} \leq p \ln (n \delta)}, \\
			S_\infty (n\delta) & = \set{\gamma \in S(n\delta)}{ \gro \eta\xi{\gamma^{-1} o} > p \ln (n \delta)}.
		\end{align*}
		Roughly speaking, $S_0(n\delta)$ is the set of all $\gamma \in S(n\delta)$ such that $\gamma^{-1}o$ stay close to $\sigma$.
		We will bound the corresponding sum using the regularity of $f$.
		On the other hand $S_\infty (n\delta)$ is the set of all elements $\gamma \in S(n\delta)$ such that $\gamma^{-1}o$ is far from $\sigma$.
		The corresponding sum will be controlled using the exponential decay of $f$.
		We split the details in three claims.

		\begin{clai}
			\label{cla: ergo2 - belly FT - card}
			There exists $C \in \R_+$ (which does not depend on $n$ or $T$) such that $\card{S_0(n\delta)} \leq Cn^q$.
		\end{clai}

		Let $\gamma \in S_0(n \delta)$.
		Using (\ref{eqn: flow2 - proj equiv + dist}) we observe that, up to $220\delta$ the distance between $o$ and $\gamma \sigma(n\delta)$ is at most
		\begin{equation*}
			\gro {\gamma \eta}{\gamma \xi} o + \abs{n \delta + \kappa_\gamma(\eta, \xi)} \leq p \ln (n \delta) + \delta.
		\end{equation*}
		Consequently $S_0(n\delta)$ is contained in
		\begin{equation*}
			U = \set{\gamma \in \Gamma}{\dist o{\gamma \sigma(n\delta)} \leq r},
			\quad \text{where} \quad
			r = p \ln (n\delta) + 221\delta.
		\end{equation*}
		By \autoref{res: ergo2 - prelim card ball}, there exists $C \in \R_+$ (independent of $n$ or $T$) such that
		\begin{equation*}
			\card U \leq Ce^{2h_\Gamma r} \leq Ce^{221h_\Gamma \delta} (n\delta)^{2ph_\Gamma},
		\end{equation*}
		which completes the proof of the first claim.

		\medskip
		Recall that $a_0$ denotes the parameter which allows to approximate the visual metric on $\partial X$ by Gromov products (\ref{eqn: ergo2 - recall visual metric}).
		\begin{clai}
			\label{cla: ergo2 - belly FT - reg f}
			There exists $C \in \R_+$ (which does not depend on $n$ or $T$) such that
			\begin{equation*}
				\sum_{\gamma \in S_0(n \delta)} \abs{f(\gamma \eta, \gamma \xi) - f(\gamma \eta', \gamma \xi)} \leq C n^qe^{-a_0n\delta}.
			\end{equation*}
		\end{clai}

		According to our assumption $f$ is Lipschitz with respect to the product metric on $\partial^2X$.
		Moreover, $v$ and $v'$ have the same future, namely $\xi$.
		These observations together with (\ref{eqn: ergo2 - recall visual metric}) imply that there exists $M \in \R_+$ such that for every $\gamma \in S_0(n\delta)$.
		\begin{equation*}
			\abs{f(\gamma \eta, \gamma \xi) - f(\gamma \eta', \gamma \xi)}
			\leq
			M e^{-a_0\gro \eta{\eta'}{\gamma^{-1}o}}.
		\end{equation*}
		However, as $\kappa_\gamma(\eta,\xi) \leq -T_1 - \delta$, we observed at the beginning of the proof that
		\begin{equation*}
			\gro \eta{\eta'}{\gamma^{-1}o} \geq - \kappa_\gamma(\eta, \xi) - T_1 \geq n\delta - T_1.
		\end{equation*}
		Consequently
		\begin{equation*}
			\sum_{\gamma \in S_0(n \delta)} \abs{f(\gamma \eta, \gamma \xi) - f(\gamma \eta', \gamma \xi)}
			\leq M\card{S_0(n \delta)} e^{-a_0n\delta}.
		\end{equation*}
		\autoref{cla: ergo2 - belly FT - reg f} now follows from the estimate of $\card{S_0(n \delta)}$ given by \autoref{cla: ergo2 - belly FT - card}.

		\begin{clai}
			\label{cla: ergo2 - belly FT - decay f}
			There exists $C \in \R_+$ (independent of $n$ or $T$) such that
			\begin{equation*}
				\sum_{\gamma \in S_\infty(n \delta)} \abs{f(\gamma \eta, \gamma \xi) - f(\gamma \eta', \gamma \xi)} \leq \frac C{n^2}.
			\end{equation*}
		\end{clai}

		We split this sum in two parts as follows.
		\begin{equation*}
			\sum_{\gamma \in S_\infty(n \delta)} \abs{f(\gamma \eta, \gamma \xi) - f(\gamma \eta', \gamma \xi)}
			\leq \sum_{\gamma \in S_\infty(n \delta)} f(\gamma \eta, \gamma \xi) + \sum_{\gamma \in S_\infty(n \delta)}  f(\gamma \eta', \gamma \xi).
		\end{equation*}
		Proceeding as for $h_\Theta$, we observe that
		\begin{equation*}
			\sum_{\gamma \in S_\infty(n \delta)} f(\gamma \eta, \gamma \xi) \leq e^{a\delta} \sum_{\gamma \in S_\infty(n \delta)}f(\gamma \eta, \gamma \xi)e^{-a\abs{n\delta + \kappa_\gamma(\eta, \xi)}}
		\end{equation*}
		We now argue as in \autoref{res: ergo2 - tails FT} and prove that there exists a constant $C \in \R_+$ (which does not depends on $n$ or $T$) such that
		\begin{equation*}
			\sum_{\gamma \in S_\infty(n \delta)} f(\gamma \eta, \gamma \xi)
			\leq C  \sum_{\gamma \in S_\infty(n \delta)}e^{-a\dist o{\gamma \sigma(n\delta)}}.
		\end{equation*}
		As usual the distance $\dist o{\gamma \sigma(n\delta)}$ can be approximated by
		\begin{equation*}
			\gro \eta\xi{\gamma^{-1}o} + \abs{n\delta + \kappa_\gamma(\eta, \xi)}
		\end{equation*}
		It follows from the very definition of $S_\infty(n\delta)$ that $\dist o{\gamma \sigma(n\delta)} > p \ln (n\delta) - 220\delta$, for every $\gamma \in S_\infty(n\delta)$.
		Hence
		\begin{equation*}
			\sum_{\gamma \in S_\infty(n \delta)} f(\gamma \eta, \gamma \xi)
			\leq C  \sum_{\substack{\gamma \in \Gamma, \\ \dist o{\gamma \sigma(n\delta)} > p \ln (n\delta) - 220\delta}}e^{-a\dist o{\gamma \sigma(n\delta)}}.
		\end{equation*}
		An upper bound of the last sum is given by \autoref{eqn: ergo2 - prelim control Poincare series}.
		More precisely, up to replacing $C$ by a larger constant (which still does not depend on $n$ or $T$) we get
		\begin{equation*}
			\sum_{\gamma \in S_\infty(n \delta)} f(\gamma \eta, \gamma \xi)
			\leq C  e^{-(a-2h_\Gamma) p \ln(n\delta)}
			\leq \frac C{(n\delta)^2}.
		\end{equation*}
		The last inequality is just the definition of $p$.
		Recall that whenever $\kappa_\gamma(\eta, \xi) \leq - T_1$, then $\kappa_\gamma(\eta, \xi)$ and $u + \kappa_\gamma(\eta', \xi)$ differ by at most $2000\delta$.
		Following the exact same argument we get a similar upper bound for
		\begin{equation*}
			\sum_{\gamma \in S_\infty(n \delta)}  f(\gamma \eta', \gamma \xi'),
		\end{equation*}
		which completes the proof of \autoref{cla: ergo2 - belly FT - decay f}.
		To summarize, the last two claims tell us that there exists $C_f$ (which does not depend on $n$ or $T$) such that
		\begin{equation}
			\label{res: ergo2 - belly FT - f diff}
			h_F(n) \leq C_f\left( \frac 1{n^2} + n^q e^{-a_0n\delta}\right)
		\end{equation}
		\autoref{res: ergo2 - belly FT} is the combination of (\ref{res: ergo2 - belly FT - Theta diff}) and (\ref{res: ergo2 - belly FT - f diff}).
	\end{proof}

	Recall that we need to estimate $F_T(v) - F_T(v'_u)$.
	According to \autoref{res: ergo2 - tails FT} there exists $C_1 \in \R_+$ such that for every $T \geq 0$,
	\begin{equation*}
		\abs{F_T(v) - F_T(v'_u)}
		\leq C_1 + \sum_{T_1+ \delta \leq n \delta \leq T - \delta} \sum_{\gamma \in S(n \delta)} \abs{f \otimes \Theta_{T_1}^T\left(\gamma v\right) - f \otimes \Theta_{T_1}^T\left(\gamma v'_u\right)}
	\end{equation*}
	Combined with \autoref{res: ergo2 - belly FT}, we see that there exists $C_2 \in \R_+$ such that
	\begin{align*}
		\MoveEqLeft{\abs{F_T(v) - F_T(v'_u)}}                                                                                                                                  \\
		 & \leq C_1 + C_2\sum_{T_1+ \delta \leq n \delta \leq T - \delta}\left[\left(e^{a(T_1-n\delta)}+ e^{-a(T- n \delta)}\right)+\frac 1{n^2} + n^q e^{-a_0n\delta}\right].
	\end{align*}
	Observe that for every integer $n$ indexing the sum $T_1 - n \delta$ is negative, whereas $T-n\delta$ is positive.
	Consequently, the latter sum is bounded from above independently of $T$, which completes the proof of the proposition. 
\end{proof}

\paragraph{Running the Hopf argument.}
 We fix until the end of this section  a \emph{bounded positive} function $g \in \mathcal D^+(\partial^2X)$, i.e. $g$ is Lipschitz with exponential decay.
For instance one can chose $g(\eta, \xi) = \dist[\partial X] \eta\xi^p$ for a sufficiently large $p \in \R_+$.
Recall that $g$ belongs to $L^1(\mu)$.
Up to rescaling $g$ we can assume that
\begin{equation*}
	\int \hat g_\vartheta d \bar m
	= \int  g_\vartheta d  m
	= \int g d\mu = 1.
	\end{equation*}
In addition we define an auxiliary map
\begin{equation*}
	\begin{array}{rccc}
		g' \colon & \partial^2 X & \to & \R_+^* \\
		& (\eta, \xi) & \to & \displaystyle \int_{\R} \hat g_\vartheta(\eta, \xi, t) \theta(t) dt.
	\end{array}
\end{equation*}
Note that as $\hat g_\vartheta$ is bounded (\autoref{res: f-bornee}), $g'$ is a bounded positive map.

\begin{prop}
	\label{res: ergo2 - hopf argt}
	\label{res: ergo - hopf argt boundary quantified}
	Assume that the geodesic flow on $(SX, \mathcal B_\Gamma, \bar m)$ is conservative.
	If  $f \in L^1(\mu)$, then for $\bar m$-almost every $v \in SX$,
	\begin{equation*}
		\lim_{T \to \pm \infty}\frac {\int_0^T \hat f_\vartheta \circ \phi_t(v) dt}{\int_0^T \hat g_\vartheta \circ \phi_t(v) dt}
		=  \int_{\partial^2X}fg'd\mu.
	\end{equation*}
	The same conclusion holds with $v=(\eta,\xi, 0)$, for $\mu$-almost all $(\eta,\xi)\in\partial^2X$.
\end{prop}

\begin{proof}
	Recall that the map $\hat f_\vartheta \colon SX \to \R$ defined as in (\ref{eqn: ergo2 - averaging vs integration}) is $\Gamma$-invariant and $\bar m$-integrable.
	Since the geodesic flow on $(SX, \mathcal B_\Gamma, \bar m)$ is conservative, the Hopf ergodic theorem \cite{Hopf:1937kk} tells us that for $\bar m$-almost every $v \in SX$,
	\begin{equation}
		\label{eqn: ergo2 - hopf argt - hopf}
		\lim_{T \to \pm \infty}\frac {\int_0^T \hat f_\vartheta \circ \phi_t(v) dt}{\int_0^T \hat g_\vartheta \circ \phi_t(v) dt}
		= f_\infty(v),\quad\text{where}\quad
		f_\infty(v) = \mathbb E_{\hat g_\vartheta\bar m}\left(\hat f_\vartheta\ \middle|\ \mathcal I \right) (v)
	\end{equation}
	is the conditional expectation of $\hat f_\vartheta$ with respect to the sub-$\sigma$-algebra $\mathcal I$ of $\mathcal B_\Gamma$ of all $(\phi_t)$-invariant Borel subsets.

	\medskip
	Assume that $f$ belongs to $\mathcal D^+(\partial^2X)$.
	As the geodesic flow on $(SX, \mathcal B_\Gamma, \bar m)$  is conservative, both the numerator and the denominator in (\ref{eqn: ergo2 - hopf argt - hopf}) diverge to infinity.
	Since  $\hat f_\vartheta$  and $\hat g_\vartheta$  are bounded (\autoref{res: f-bornee}), the map $f_\infty(v)$  does not depend on the time coordinate of $v = (\eta, \xi, t)$, hence we write $f_\infty(v) = f_\infty(\eta, \xi)$.
	The crucial ingredient is \autoref{res: ergo2 - contraction lemma}, which implies that the map $f_\infty$ only depends on the future, $\bar m$- or $m$- or  $\mu$-almost surely.
	As the flow  is flip invariant, the map $f_\infty$ depends also only on the past, $\mu$-almost surely.
	Since $\mu$ is equivalent to a product measure, the standard Hopf argument (based on Fubini Theorem) shows that $f_\infty$ is constant $\bar m$- or $m$- or  $\mu$-almost surely.

	By construction $\hat g_\vartheta$ is bounded (see \autoref{res: f-bornee}) so that $f_\vartheta  \hat{g}_\vartheta\in L^1(m)$. 
	As $\hat{g}_\vartheta$ is $\Gamma$-invariant, (\ref{eqn: ergo2 - averaging vs integration}) yields
	\begin{equation*}
		\int fg' d\mu 
		= \int f_\vartheta \hat g_\vartheta dm
		=\int \widehat{f_\vartheta \hat g_\vartheta}d\bar m
		=\int \hat f_\vartheta \hat g_\vartheta d\bar m.
	\end{equation*}
	By  definition of conditional expectation, we deduce that the almost sure value of $f_\infty$, say $M\in\R$, satisfies
	\begin{equation*}
		M = \int f_\infty \hat g_\vartheta d\bar m
		=\int \hat f_\vartheta \hat g_\vartheta d\bar m
		= \int fg' d\mu.
	\end{equation*}
	As $g'$ is bounded, both maps
	\begin{equation*}
		f \mapsto \mathbb E_{\hat g_\vartheta\bar m}\left(\hat f_\vartheta\ \middle|\ \mathcal I \right)
		\quad \text{and} \quad
		f \mapsto \int_{\partial^2X} fg'd\mu
	\end{equation*}
	are bounded linear functionals, which coincide on $\mathcal D^+(\partial^2X)\subset L^1_+(\mu)$. 
	As it is a dense subset of $L^1_+(\mu)$, they coincide everywhere. 
	It completes the proof of the main statement. 
	The proof of the last statement is a direct corollary of the previous argument. 	We omit it. 
\end{proof}

We have not quite proved yet that the measure $\bar m$ is ergodic for the flow $(\phi_t)$.
Indeed \autoref{res: ergo2 - hopf argt} does not a priori apply for any function in $L^1(\bar m)$.
Nevertheless it is sufficient to deduce that $\mu$ is ergodic for the diagonal action of $\Gamma$ on $\partial^2 X$.
The next statement completes the proof of \autoref{res: hopf tsuji}.

\begin{coro}\label{res: conservative implies double ergodic}
	Assume that the geodesic flow on $(SX, \mathcal B_\Gamma, \bar m)$ is conservative.
	The action of $\Gamma$ on $(\partial^2X, \mu)$ is ergodic.
\end{coro}

\begin{proof}
	Let $B$ be a $\Gamma$-invariant subset of $\partial^2X$ such that $\mu(B) > 0$.
	We want to prove that $\mu (\partial^2X \setminus B)=0$.
	Let $K\subset B$ be a compact set with $\mu(K)>0$.
	By \autoref{res: ergo - hopf argt boundary quantified} applied to $f = \mathbf 1_K$, for $\mu$-almost every $(\eta,\xi)$, for every sufficiently large $T\in \R_+$, 
	\begin{equation*}
		\int_0^T \hat f_\vartheta\circ\phi_t(v)dt>0, 
		\quad \text{where} \quad 
		v = (\eta, \xi, 0).
	\end{equation*}
	It implies that for $\mu$-almost every $(\eta,\xi)$, some $(\gamma \xi,\gamma \eta)$ lies in $K$, and therefore $B$. As $B$ is $\Gamma$-invariant, it means that $\mu$-almost every $(\eta,\xi)$ belongs to $B$, i.e. $B$ has full measure. 
\end{proof}

\subsection{Finiteness of the Bowen-Margulis measure}
%

As a by-product of our technique, we will show that, when the action of $\Gamma$ has a growth gap at infinity, the Bowen-Margulis measure $\bar m$ on $(SX, \mathcal B_\Gamma)$ is finite. This statement is not needed for the proof of \autoref{res: main}. We include it because it is an important dynamical result, which follows easily  from the previous material.
In fact, we prove the following more general statement, inspired from the work of Pit and Schapira \cite[Section 5]{Pit:2018eg}.

\begin{theo}
	\label{res: finite bm - thm} 
	Let $\Gamma$ be a discrete group acting properly  by isometries
 on a Gromov-hyperbolic space $X$.
	Assume that the Patterson-Sullivan measure $\nu_0$ gives full measure to the radial limit set $\Lambda_{\rm rad}(\Gamma)$.
	Then the Bowen-Margulis measure $\bar m$ on $(SX, \mathcal B_\Gamma)$ is finite if and only if there exists a compact subset $K$ of $X$ such that the series
	\begin{equation*}
		\sum_{\gamma \in \Gamma_K} \dist o{\gamma o} e^{-h_\Gamma \dist o{\gamma o}}
	\end{equation*}
	converges.
\end{theo}

\medskip

Recall that if the action of $\Gamma$ on $X$ is strongly positively recurrent, then $\nu_0$ gives full measure to the radial limit set (\autoref{res: PS charges radial limit set}).
Moreover there exists a compact subset $K$ of $\bar X$ such that $h_{\Gamma_K}  < h_\Gamma$.
Therefore \autoref{res: finite bm - thm} has the following immediate corollary.
\begin{coro}
\label{res: finite bm - finiteness final}
	Let $\Gamma$ be a discrete group acting properly  by isometries  on a Gromov-hyperbolic space $X$.
	If the action of $\Gamma$ on $X$ is strongly positively recurrent, then the Bowen-Margulis measure $\bar m$ on $(SX, \mathcal B_\Gamma)$ is finite.
\end{coro} 
\medskip
From now on, we only assume that $\nu_0$ gives full measure to the radial limit set.
By definition, $\Lambda_{\rm rad}$ is the increasing union of all $\Lambda_{\rm rad}^K$ where $K$ runs over all compact subsets of $X$.
As already noticed before, there exists a compact subset $k \subset X$, such that $\nu_o(\Lambda_{\rm rad}^k)=1$ (\autoref{res: noatom}).
Up to enlarging $k$ we may assume that $o$ belongs to $k$.
We now fix a parameter $r \geq \diam k + 1000\delta$.
For the moment $r$ is fixed, it will vary only at the very end of the proof.
For simplicity let
\begin{equation*}
	Z =Z(r)= \set{(\eta, \xi) \in \partial^2X}{\gro \eta\xi o \leq r},
\end{equation*}
and define
\begin{equation*}
	\Sigma = \set{(\eta, \xi,0) \in SX}{(\eta,\xi) \in Z},
\end{equation*}
which we think of as a ``compact'' subset of a section of the flow.
As in the preceding section, we work in $SX$ modulo $\Gamma$.
This motivates the next definition.
Given a vector $v = (\eta, \xi, 0)$ in $\Sigma$, the \emph{first return time} of $v$ in $\Sigma$ (modulo $\Gamma$), denoted by $\tau(v)$, is defined by
\begin{equation*}
	\tau(v) = \inf\set {t > 2r + 500\delta}{\exists \gamma \in \Gamma, \ \gamma^{-1}\phi_t(v) \in \Sigma}.
\end{equation*}

\paragraph{Remark.}
As $X$ is Gromov hyperbolic, we only control its large scale geometry, which causes some edge effects.
For this reason, it is convenient to require the first return time to be larger that $2r + 500\delta$ (see for instance the proof of \autoref{res: finite bm - second comparison series}).

\medskip
Define now
\begin{equation*}
	\Sigma'=\{v\in \Sigma, \tau(v)<+\infty\}.
\end{equation*}
Finally, the \emph{first return core} is defined by
\begin{equation*}
	W = \set{\phi_t(v)}{v \in \Sigma',\ 0 \leq t \leq \tau(v)}.
\end{equation*}
We are going to prove that $\Gamma W$ has full $\bar m$-measure (\autoref{res: finite bm - first return set has full measure}) and that its measure $\bar m(\Gamma W)$ is finite if and only if a certain series converges (Propositions~\ref{res: finite bm - upper bound} and \ref{res: finite bm - lower bound}).
We start with the following lemma which provides a useful criterion \emph{in the space $X$} to determine when a vector $v \in SX$ belongs to a translate of $\Sigma$.
\begin{lemm}
	\label{res: finite bm - criterion to be in Sigma}
	Let $v \in SX$ and $\gamma \in \Gamma$.
	If $\dist {\gamma o}{\proj(v)} \leq r - 220\delta$,
	then there exists $s \in \R$, with $\abs s \leq r$ such that $\gamma^{-1} \phi_s(v) \in \Sigma$.
\end{lemm}

\begin{proof}
	We write $v = (\eta, \xi, t)$.
	Combining our assumption with (\ref{eqn: flow2 - proj equiv + dist}) we get
	\begin{equation*}
		\gro {\gamma^{-1}\eta}{\gamma^{-1}\xi}o + \abs{t + \kappa_{\gamma^{-1}}(\eta, \xi)} 
		\leq \dist {\gamma o}{\proj(v)} + 220\delta 
		\leq r.
	\end{equation*}
	It follows first that $\gro {\gamma^{-1}\eta}{\gamma^{-1}\xi}o \leq r$, i.e. the pair $(\gamma^{-1}\eta, \gamma^{-1}\xi)$ belongs to $Z$.
	Moreover $s = - \kappa_{\gamma^{-1}} (\eta, \xi) - t$ satisfies $\abs s \leq r$.
	One easily checks that $\gamma^{-1} \phi_s(v) = \left(\gamma^{-1} \eta, \gamma^{-1} \xi, 0\right)$, which, according to our previous observation, belongs to $\Sigma$.
\end{proof}

\begin{prop}
\label{res: finite bm - first return set has full measure} 
	The set $\Gamma W$ is a $\Gamma$-invariant set of  full $\bar m$-measure.
	In particular, $\bar m(SX) \leq m(W)$.
\end{prop}

\begin{proof}
	By assumption, $\nu_0(\Lambda_{\rm rad}^k) = 1$.
	Since $\mu$ belongs to the same measure class as $\nu_o \otimes \nu_o$, it gives full measure the the set $(\Lambda_{\rm rad}^k \times \Lambda_{\rm rad}^k) \cap \partial^2X$.
	It follows from \autoref{res: ergo2 - conservativity preparation}, that $m$-almost every $v \in SX$, for every $T \geq 0$, there exists $t \geq T$ and $\gamma \in \Gamma$ such that $\gamma^{-1}\phi_t(v) \in \Sigma$.
	The same holds for negative times.
	Hence $W$ contains a Borel fundamental domain for the action of $\Gamma$ on $SX$.
	Consequently $\Gamma W$ has full $m$-measure and thus full $\bar m$-measure.
	The inequality $\bar m(SX) \leq m(W)$ directly follows from the definition of $\bar m$.
\end{proof}

In order to estimate the measure of $W$ it will be convenient to decompose it according to which translates of $\Sigma$ the first return map falls in.
This motivates the next definitions.
For all $\gamma \in \Gamma$, we define 
\begin{align*}
	\Sigma'_\gamma & = \set{v \in \Sigma'}{\exists s,\ \tau(v) \leq s \leq \tau(v)+2r+500\delta\ \text{and}\ \gamma^{-1} \phi_s(v) \in \Sigma} \\
	Z'_\gamma & = \set{(\eta, \xi) \in \partial^2X}{(\eta, \xi,0)\in\Sigma'_\gamma} \\
	W_\gamma & = \set{\phi_t(v)}{v \in \Sigma'_\gamma,\ 0 \leq t \leq \tau(v)}.
\end{align*}

Finally we denote by $\Gamma(\Sigma')$ the set of all elements $\gamma \in \Gamma$ for which $\Sigma'_\gamma$ is non-empty.
It follows from these definitions that 
\begin{equation}
\label{eqn: finite bm - decomposition first return core}
	W\subset\bigcup_{\gamma\in\Gamma(\Sigma')}W_\gamma.
\end{equation}
Let us study the properties of these sets.
We start with a series of lemmas that will provide an \emph{upper} bound of $\bar m(SX)$.

\begin{lemm}
	\label{res: finite bm - first return vs distance}
	For every $\gamma \in \Gamma(\Sigma')$, for every $v \in \Sigma'_\gamma$, the vectors $v$ and $v' = \phi_{\tau(v)}(v)$ satisfy
	\begin{equation*}
		\dist o{\proj (v)} \leq r + 20\delta 
		\quad \text{and} \quad
		\dist {\gamma o}{ \proj(v')} \leq 3r + 720\delta
	\end{equation*}
	Moreover $\abs{\dist o{\gamma o} - \tau(v)} \leq 4r + 740\delta$.
\end{lemm}

\begin{proof} Note that the proof would be rather obvious if the projection $SX\to X$ were $\Gamma$-equivariant.
	Let $v = (\eta, \xi, 0)$ in $\Sigma'_\gamma$.
	As  observed in (\ref{eqn : flow2 - distance projection}), the quantity $\gro \eta\xi o$ roughly measures the distance between $o$ and $\proj(v)$.
	Since $v\in\Sigma$, we get
	\begin{equation*}
		\dist o{\proj(v)}\leq \gro \eta\xi o + 20\delta \leq r + 20\delta.
	\end{equation*}
	By definition of $\Sigma'_\gamma$, there exists $t \in [\tau(v), \tau(v) + 2r + 500 \delta]$ such that $\gamma^{-1}\phi_t(v)$ belongs to $\Sigma$.
	As before, we get from (\ref{eqn: flow2 - proj equiv + dist}) 
	\begin{equation*}
		\dist{\gamma o}{\proj(\phi_t(v))}
		\leq \gro{\gamma^{-1}\eta}{\gamma^{-1}\xi}o + 220\delta
		\leq r + 220\delta.
	\end{equation*}
	The map $\proj \circ \phi_s:\R\to X$   is a bi-infinite geodesic, so that
	\begin{equation*}
		\dist {\proj(v')}{\proj(\phi_t(v))} \leq 2r + 500\delta
		\quad \text{and} \quad
		\dist {\proj(v)}{\proj(v')} = \tau(v).
	\end{equation*}
	It yields $\dist{\gamma o}{\proj(v')} \leq 3r + 720\delta$, which completes the first part of the lemma.
	The second part follows from the triangle inequality.
\end{proof}

\begin{lemm}
	\label{res: finite bm - shadow}
	For every $\gamma \in \Gamma(\Sigma')$, the set $Z'_\gamma$ is contained in the product $\mathcal O_{\gamma o}(o,r+30\delta) \times \mathcal O_o(\gamma o,r+30\delta)$.
\end{lemm}

\begin{proof}
	Let $(\eta, \xi) \in \partial^2X$ and $v = (\eta, \xi, 0)$.
	As usual we write  $\sigma \colon \R \to X$ for the bi-infinite geodesic sending $s$ to $\phi_s(v)$.
	Assume firs that $(\eta,\xi) \in Z'_\gamma$, i.e. the vector $v$ belongs to $\Sigma'_\gamma$.
	It follows that $\gro \eta\xi o \leq r$ and $\gro \eta \xi {\gamma o} \leq r$.
	In particular $o$ and $\gamma o$ are $(r+6\delta)$-close to $\sigma$.
	As the ideal geodesic triangles in $X$ are $24\delta$-thin, $\gamma o$ (\resp $o$) is $(r+ 30\delta)$-close to any geodesic joining $o$ to $\xi$  (\resp $\gamma o$ to $\eta$).
	Whence the result.
\end{proof}

\begin{lemm}
\label{res: finite bm - first comparison series}
	Assume that $K$ is a compact subset contained in $B(r - 300 \delta)$.
	There exist two finite subsets $S_1$ and $S_2$ of $\Gamma$ such that $\Gamma(\Sigma') \setminus S_1$ is contained in $S_2\Gamma_KS_2$.
\end{lemm}

\begin{proof}
	Set
	\begin{align*}
		S_1  &= \set{\gamma \in \Gamma}{\dist o{\gamma o} \leq 8r + 2000\delta}, \\
		S_2  &= \set{\gamma \in \Gamma}{\dist o{\gamma o} \leq 5r + 1000\delta}.
	\end{align*}
	Let $\gamma \in \Gamma(\Sigma')\setminus S_1$ and choose an arbitrary $v \in \Sigma'_\gamma$.
	For simplicity, set $\tau = \tau(v)$ for the first return time of $v$ in $\Sigma$.
	Recall that the map $\sigma \colon \R \to X$ sending $t$ to $\proj \circ \phi_t(v)$ is a bi-infinite geodesic of $X$ joining $\eta$ to $\xi$.
	According to \autoref{res: finite bm - first return vs distance}
	\begin{equation*}
		\dist o{\sigma(0)} \leq r + 20\delta
		\quad \text{and} \quad
		\dist {\gamma o}{ \sigma(\tau)} \leq 3r + 720\delta
	\end{equation*}
	Fix a geodesic $c \colon [0, \ell] \to X$ joining $o$ to $\gamma o$.
	Let $s_1,s_2 \in [0, 4r + 1000\delta]$ be the largest times such that $c(s_1)$ (\resp $c(\ell - s_2)$) belongs $\alpha_1 K$ for some $\alpha_1\in \Gamma$ (\resp  $\gamma \alpha_2 K$ for some $\alpha_2\in\Gamma$).
	It follows from the previous claim combined with the triangle inequality that $\dist o{\alpha_i o} \leq 5r + 1000\delta$.
	In other words $\gamma$ can be written $\gamma = \alpha_1 (\alpha_1^{-1}\gamma \alpha_2)\alpha_2^{-1}$ where $\alpha_1$ and $\alpha_2^{-1}$ belong to $S_2$.
	Thus we are left to prove that $\alpha_1^{-1}\gamma \alpha_2$ belongs to $\Gamma_K$.
	As $\gamma$ does not belong to $S_1$, the point $x$, $c(s_1)$, $c(\ell - s_2)$ and $\gamma y$ are aligned in this order along $c$.
	Hence it suffices to prove that $c$ restricted to $(s_1, \ell - s_2)$ does not intersect $\Gamma K$.
	Assume on the contrary that there exists $s \in (s_1, \ell - s_2)$ 
such that $y = c(s)$ belongs to $\beta K$ for some $\beta \in \Gamma$.
	By construction $\dist o{c(s)} > \dist o{\sigma(0)} +3r+ 510\delta$ 
and $\dist {\gamma o}{c(s)} > \dist{\gamma o}{\sigma(\tau)} +r+ 10\delta$.
	It is a standard exercise in hyperbolic geometry to observe that $y$ is $6\delta$-close to a point $x = \sigma(t)$ with $t \in (3r+500\delta,\tau-r)$.
	In particular, $\dist {\beta o}x \leq r -220\delta$.
	It follows then from \autoref{res: finite bm - criterion to be in Sigma} 
that there exists $t' \in (2r +500\delta, \tau(v))$ such that $\beta^{-1}\phi_{t'}(v)$ belongs to $\Sigma$.
	This contradicts the definition of the first return time and completes the proof of the claim.
\end{proof}

\begin{prop}
\label{res: finite bm - upper bound}
	Assume that $K$ is a compact subset contained in $B(r - 300 \delta)$.
	There exists $C \in \R_+$ such that 
	\begin{equation*}
		\bar m(SX) \leq C \sum_{\gamma \in \Gamma_K} \dist o{\gamma o} e^{-h_\Gamma \dist o{\gamma o}}
	\end{equation*}
\end{prop}

\begin{proof}  
	As we observed earlier $\bar m (SX) \leq m(W)$ (\autoref{res: finite bm - first return set has full measure}).
	For every $(\eta, \xi) \in Z$ define $\tau(\eta,\xi ) = \tau (v)$ where $v = (\eta, \xi)$.
	Recall that $m = \mu \otimes dt$.
	Thus the decomposition of the first return core $W$ given in (\ref{eqn: finite bm - decomposition first return core}) yields
	\begin{equation*}
		\bar m (SX) 
		\leq \sum_{\gamma \in \Gamma(\Sigma')} m(W_\gamma) 
		\leq \sum_{\gamma \in \Gamma(\Sigma')} \int \mathbf 1_{Z'_\gamma}(\eta, \xi) \tau(\eta, \xi) d\mu(\eta, \xi),
	\end{equation*}
	By \autoref{res: finite bm - first return vs distance}, the first return time $\tau$ is approximatively $\dist o{\gamma o}$ when restricted to $Z'_\gamma$.
	Moreover by (\ref{eqn: bm product}) $\mu$ restricted to $Z'$ is comparable to $\nu_o \otimes \nu_o$.
	Hence there exists $C \in \R_+$ such that 
	\begin{equation*}
		\bar m (SX)  
		\leq C\sum_{\gamma \in \Gamma(\Sigma')} \dist o{\gamma o} (\nu_o \otimes \nu_o)(Z'_\gamma).
	\end{equation*}
	According to \autoref{res: finite bm - shadow},  $Z'_\gamma$ is contained in $\partial X \times \mathcal O_o(\gamma o, r +30\delta)$.
	Hence (up to increasing $C$) the Shadow Lemma (\autoref{res: shadow lemma}) gives
	\begin{equation*}
		\bar m (SX) 
		\leq C\sum_{\gamma \in \Gamma(\Sigma')} \dist o{\gamma o} e^{-h_\Gamma \dist o{\gamma o}}.
	\end{equation*}
	The conclusion now follows from \autoref{res: finite bm - first comparison series}.
\end{proof}

Let us now provide a \emph{lower} bound of $\bar m(SX)$.
To that end we define
\begin{equation*}
	W^0 = \set{\phi_s(v)}{v \in \Sigma',\ 0 \leq s < \tau(v)} 
\end{equation*}
The first step is to estimate the multiplicity of certain families.

\begin{lemm}
\label{res: finite bm - multiplicity 1}
	There exists $N \in \N$ such that for $\bar m$-almost every $v \in SX$ 
the set $\set{\gamma \in \Gamma}{v \in \gamma W^0}$ contains at most $N$ elements.
\end{lemm}

\begin{proof}
	Recall that $S_0X$ is the full measure, $\Gamma$ and flow invariant subset defined in (\ref{eqn: flow2 - full measure inv subset}).
	Let $v \in S_0X$.
	Let $\alpha, \beta \in \Gamma$ such that $v$ belongs to $\alpha W^0 \cap \beta W^0$.
	We can write $\alpha \phi_s(u) = v = \beta \phi_t(w)$, where $u,w \in \Sigma'$,
	\begin{equation*}
		0 \leq s < \tau(u),
		\quad \text{and} \quad
		0 \leq t < \tau(w).
	\end{equation*}
	In particular $\alpha^{-1} \beta \phi_{t-s}(w) = u$ and $\beta^{-1}\alpha \phi_{s-t}(u) = w$ both belong to $\Sigma$.
	By construction either $0 \leq t-s <\tau(w)$ or $0 \leq s - t < \tau(u)$.
	It follows from our definition of first return time that $\abs{t-s} \leq 2r + 500\delta$.
	Since $\proj \colon SX \to X$ is almost $\Gamma$-equivariant (\ref{eqn: flow2 - proj equiv}) and maps orbits of the flow to geodesics we get
	\begin{equation*}
		\dist {\alpha \proj(u)}{\beta \proj(w)} \leq 2r + 700\delta.
	\end{equation*}
	On the other hand, since $w$ belongs to $\Sigma$, we have $\dist o{\proj(w)} \leq r + 20\delta$ (\autoref{res: finite bm - first return vs distance}).
	Consequently 
	\begin{equation*}
		\set{\gamma \in \Gamma}{v \in \gamma W^0} \subset \set{\gamma \in \Gamma}{\dist x {\gamma o} \leq 3r + 800\delta}  
	\end{equation*}
	where $x = \alpha\proj(u)$.
	The conclusion follows from \autoref{res: ergo2 - prelim card ball}.
\end{proof}

\begin{lemm}
\label{res: finite bm - multiplicity 2}
	There exists $N \in \N$ such that for every $v \in SX$ the set $\set{\gamma \in \Gamma}{v \in W_\gamma}$ contains at most $N$ elements.
\end{lemm}

\begin{proof}
	Set $v' = \phi_{\tau(v)}(v)$.
	It follows from \autoref{res: finite bm - first return vs distance} that 
	\begin{equation*}
		\set{\gamma \in \Gamma}{v \in W_\gamma} \subset \set{\gamma \in \Gamma}{ \dist {\proj(v')}{\gamma o} \leq 3r + 720\delta}.
	\end{equation*}
	Hence the result follows from \autoref{res: ergo2 - prelim card ball}.
\end{proof}

\begin{lemm}
\label{res: finite bm - second comparison series}
	There exists a compact subset $K\subset X$ and a finite subset $S \subset \Gamma$ such that for every $\gamma \in \Gamma_K \setminus S$, the product $\mathcal O_{\gamma o}(o,r-\delta) \times \mathcal O_o(\gamma o,r-\delta)$ is contained in $Z'_\gamma$.
	In particular $\Gamma_K \setminus S \subset \Gamma(\Sigma')$
\end{lemm}

\begin{proof}
	Let $K$ be the closed ball $K = \bar B(o, r + 250\delta)$ and set 
	\begin{equation*}
		S = \set{\gamma \in \Gamma}{\dist o{\gamma o} \leq 6r + 1000\delta}.
	\end{equation*}
	Let $\gamma \in \Gamma_K \setminus S$ and $(\eta, \xi) \in \mathcal O_{\gamma o}(o,r-\delta) \times \mathcal O_o(\gamma o,r-\delta)$.
	Since $\eta$ belongs to $\mathcal O_{\gamma o}(o,r-\delta)$, it follows from the four point inequality (\ref{eqn: hyperbolicity condition with boundary}) that 
	\begin{equation*}
		\min\left\{ \gro \eta \xi o , \gro {\gamma o}\xi o \right\} \leq \gro {\gamma o}\eta o + \delta \leq r.
	\end{equation*}
	As $\xi$ belongs to $\mathcal O_{\gamma o}(o,r-\delta)$ and $\gamma \notin S$, we have
	\begin{equation*}
		\gro {\gamma o}\xi o \geq \dist o{\gamma o} - r - \delta >r,
	\end{equation*}
	thus the minimum cannot be achieved by $\gro {\gamma o}\xi o$.
	Hence $\gro \eta\xi o \leq r$, which means that $v = (\eta, \xi,0)$ lies in $\Sigma$.
	Similarly we prove that $\gro \eta \xi {\gamma o} \leq r$, thus there exists $t \in \R$ such that $\gamma^{-1} \phi_t(v)$ belongs to $\Sigma$.
	Since $\dist o{\gamma o} > 6r + 1000\delta$, we can assume that $t > 0$.
	In particular $\tau(v) \leq t$.
	We now need to prove that $t \leq \tau(v) + r + \delta$.
	Assume on the contrary that is its not the case.
	In particular there exists $s \in [\tau(v), t - r - \delta)$ such that $\alpha^{-1}\phi_s(v) \in \Sigma$, for some $\alpha \in \Gamma$. 
	For simplicity we let $z_0 = \proj(v)$, $z_s = \proj\circ \phi_s(v)$ and $z_t = \proj \circ \phi_t(v)$.
	By (\ref{eqn: flow2 - proj equiv + dist}) we have 
	\begin{equation*}
		\max\left\{\dist o{z_0}, \dist{\alpha o}{z_s}, \dist{\gamma o}{z_t} \right\} \leq r + 220\delta.
	\end{equation*}
	Since $\gamma$ belongs to $\Gamma_K$, there exists $x,y \in K$ and a geodesic $c \colon [0, \ell] \to X$ joining $x$ to $\gamma y$ such that $c \cap \Gamma K \subset K \cup \gamma K$.
	It follows then from the triangle inequality that $\dist x{z_0} \leq 2r + 470\delta$ and $\dist {\gamma y}{z_t} \leq 2r + 470\delta$.
	On the other hand since $\proj \colon SX \to X$ maps orbits of the flow to geodesics, hence $\abs{s - t} \leq r + \delta$, we have
	\begin{equation}
	\label{eqn: finite bm - second comparison series}
		\dist{z_0}{z_s} \geq \tau(v) > 2r + 500\delta
		\quad \text{and} \quad
		\dist{z_t}{z_s} \geq t - s > 2r+ 500\delta.
	\end{equation}
	A standard exercise of hyperbolic geometry show that $z_s$ it $6\delta$-close to a point $c(s)$ on $c$.
	In particular $\dist{\alpha o}{c(s)} \leq r + 250\delta$, i.e. $c(s) \in \alpha K$.
	It follows from the definition of $c$ that $c(s)$ belongs to $K \cup \gamma K$.
	Consequently either $\dist {z_0}{z_s} \leq 2r + 500\delta$ or $\dist{z_t}{z_s} \leq 2r + 500\delta$, which violates (\ref{eqn: finite bm - second comparison series}).
\end{proof}

\begin{prop}
\label{res: finite bm - lower bound}
	There exist $C \in \R_+^*$ and a compact subset $K \subset X$ such that 
	\begin{equation*}
		C \sum_{\gamma \in \Gamma_K} \leq \bar m(SX).
	\end{equation*}
\end{prop}

\begin{proof}
	We write $K$ for the compact subset of $X$ given by \autoref{res: finite bm - second comparison series}.
	Obviously $\bar m(\Gamma W^0) \leq \bar m(SX)$.
	Note that the collection $(\gamma W^0)$ may not be pairwise disjoint, nevertheless thanks to \autoref{res: finite bm - multiplicity 1} we control its multiplicity.
	Thus there exists $C \in \R_+^*$ such that 
	\begin{equation*}
		C m(W^0) \leq \bar m(\Gamma W^0) \leq \bar m(SX).
	\end{equation*}
	Similarly (up to decreasing $C$) we get by \autoref{res: finite bm - multiplicity 1}
	\begin{equation*}
		C \sum_{\gamma \in \Gamma(\Sigma')} m\left(W^0 \cap W_\gamma\right)
		\leq \bar m(SX).
	\end{equation*}
	Reasoning as in \autoref{res: finite bm - upper bound}, we get 
	\begin{equation*}
		C \sum_{\gamma \in \Gamma(\Sigma')} \dist o{\gamma o} (\nu_o \otimes \nu_o) (Z'_\gamma) 
		\leq \bar m(SX).
	\end{equation*}
	By \autoref{res: finite bm - second comparison series}, $\mathcal O_{\gamma o}(o,r-\delta) \times \mathcal O_o(\gamma o,r-\delta)$ is contained in $Z'_\gamma$, for all but finitely many $\gamma \in \Gamma_K$.
	Combined with the Shadow Lemma (\autoref{res: shadow lemma}) it yields
	\begin{equation*}
		C \sum_{\gamma \in \Gamma_K} \dist o{\gamma o}e^{-h_\Gamma \dist o{\gamma o}} \leq \bar m(SX). \qedhere.
	\end{equation*}
 	
\end{proof}

We complete this section with the proof of \autoref{res: finite bm - thm}.
\begin{proof}[Proof of \autoref{res: finite bm - thm}]
	Assume first that the Bowen-Margulis measure $\bar m$ is finite.
	It follows form \autoref{res: finite bm - lower bound} that there exists a compact subset $K \subset X$ such that the series 
	\begin{equation*}
		\sum{\gamma \in \Gamma_K} \dist o{\gamma o} e^{-h_\Gamma \dist o{\gamma o}}
	\end{equation*}
	converges.
	Assume on the contrary that there exists a compact subset $K$ for which the above series converges.
	Up to enlarging the value of $r$, we can always assume that $K$ is contained in $B(o, r - 300\delta)$.
	It follows from \autoref{res: finite bm - upper bound} that $\bar m$ is finite.
\end{proof}

%
\section{A twisted Patterson-Sullivan measure}
%
\label{sec: twisted ps}

%
\subsection{Main theorem}
%

\paragraph{Setting}
Let $(X,d)$ be a proper geodesic $\delta$-hyperbolic space.
We fix once and for all a base point $o \in X$.
Let $\Gamma$ be a group acting properly by isometries on $X$.
Recall that $h_\Gamma$ stands for the critical exponent of the Poincar\'e series of $\Gamma$.

\medskip
Let $(\mathcal H, \prec)$ be a \emph{Hilbert lattice}, i.e. a Hilbert space endowed with a partial order $\prec$, compatible with the vector space structure as well as the norm, which induces a lattice structure on $\mathcal H$.
We refer the reader to \autoref{sec: banach lattices} for a precise definition.
All properties of Hilbert lattices that we will use are also recalled in this appendix.
Denote by $\mathcal H^+$ its positive cone, i.e. the set of elements $\phi\in\mathcal H$ such that $0\prec \phi$. 
Let $\rho \colon \Gamma \to \mathcal U(\mathcal H)$ be a \emph{positive} unitary representation, i.e. $\rho(\gamma) \phi \in \mathcal H^+$, for every $\gamma \in \Gamma$ and every $\phi \in \mathcal H^+$.

\paragraph{Twisted Poincar\'e series.}
For every $s \in \R_+$ we consider the formal series $A(s)$ defined as follows
\begin{equation*}
	A(s) = \sum_{\gamma \in \Gamma} e^{-s\dist{\gamma o}o} \rho(\gamma).
\end{equation*}
We say that this series is \emph{bounded} if there exists 
$M \in \R_+$ such that for every finite subset $S$ of $\Gamma$,
	\begin{equation*}
		\norm{\sum_{\gamma \in S} e^{-s\dist{\gamma o}o} \rho(\gamma)} \leq M.
	\end{equation*}
	The \emph{critical exponent} of the representation $\rho$ is defined as
	\begin{equation*}
		h_\rho = \inf \set{ s \in \R_+}{ A(s)\ \text{is bounded}}.
	\end{equation*}
	According to \autoref{res: app - converging operator net}, for every $s > h_\rho$, the series pointwise converges to a bounded operator of $\mathcal H$.
The following lemma is straightforward.

\begin{lemm}
	\label{res: upper bound delta rho}
	For every $s > h_\Gamma$,  the series $A(s)$ is bounded and  $\norm{A(s)} \leq \mathcal P_\Gamma(s)$.
	In particular, $h_\rho \leq h_\Gamma$.
\end{lemm}

\paragraph{Almost invariant vectors.}
Let $S$ be a finite subset of $\Gamma$ and $\varepsilon \in \R_+^*$.
A vector $\phi \in \mathcal H$ is \emph{$(S,\varepsilon)$-invariant} (with respect to $\rho$) if
\begin{displaymath}
	\sup_{\gamma \in S} \norm{\rho(\gamma)\phi - \phi} < \varepsilon \norm{\phi}.
\end{displaymath}
The representation $\rho \colon \Gamma \to  \mathcal U(\mathcal H)$  \emph{almost has invariant vectors} if for every finite subset $S$ of $\Gamma$, for every $\varepsilon \in \R_+^*$, 
there exists an $(S, \varepsilon)$-invariant vector.
The goal of this section is to prove the following statement.

\begin{theo}
	\label{res: growth vs almost inv vec - quantified} 
	Let $\Gamma$ be a discrete group acting properly by isometries on a hyperbolic space $(X,d)$.
	Assume that the action of $\Gamma$ on $X$ is strongly positively recurrent.
	For every finite subset $S$ of $\Gamma$, for every $\varepsilon \in \R_+^*$, there exists $\eta \in \R_+^*$ with the following property.
	Let $\rho \colon \Gamma \to \mathcal U(\mathcal H)$ be a unitary positive representation of $\Gamma$ into a Hilbert lattice.
	If $h_\rho \geq (1- \eta) h_\Gamma$, then $\rho$ has an $(S, \varepsilon)$-invariant vector.
\end{theo}

The proof of this result is given in Sections~\ref{sec: conf family of hilbert functionals}~-~\ref{sec: invariant vectors}.
For the moment let us mention a first consequence of this statement.

\begin{coro}
	\label{res: growth vs almost inv vec}
	Let $\rho \colon \Gamma \to \mathcal U(\mathcal H)$ be a unitary positive representation of $\Gamma$ into a Hilbert lattice.
	The representation $\rho$ almost has invariant vectors if and only if $h_\rho = h_\Gamma$.
	In this case, $\norm{A(s)} = \mathcal P_\Gamma(s)$, for every $s > h_\Gamma$.
\end{coro}

\begin{proof}	Assume first that the representation almost has invariant vectors.
	Let $s > h_\rho$, and $S$ be a finite subset of $\Gamma$ and $\varepsilon \in \R_+^*$.
	There exists a vector $\phi \in \mathcal H\setminus \{0\}$ such that for every $\gamma \in S$, we have $\norm{\rho(\gamma) \phi - \phi} < \varepsilon \norm \phi$.
	It  yields
	\begin{equation*}
		\norm{\sum_{\gamma \in S}e^{-s \dist o{\gamma o}}\rho(\gamma)\phi - \sum_{\gamma \in S}e^{-s \dist o{\gamma o}}\phi }
		\leq
		\varepsilon \sum_{\gamma \in S} e^{-s\dist o{\gamma o}}\norm \phi,
	\end{equation*}
	whence
	\begin{equation*}
		(1- \varepsilon) \sum_{\gamma \in S} e^{-s\dist o{\gamma o}} \norm \phi
		\leq \norm{\sum_{\gamma \in S} e^{-s\dist o{\gamma o}} \rho(\gamma) \phi }
		\leq \norm{\sum_{\gamma \in S} e^{-s\dist o{\gamma o}} \rho(\gamma)} \norm \phi.
	\end{equation*}
	Since $\phi$ is a non-zero vector we get
	\begin{equation*}
		(1- \varepsilon) \sum_{\gamma \in S} e^{-s\dist o{\gamma o}}
		\leq \norm{\sum_{\gamma \in S} e^{-s\dist o{\gamma o}} \rho(\gamma)}.
	\end{equation*}
	This inequality holds for all $\varepsilon > 0$.
	Hence, for any finite subset $S$ of $\Gamma$, we have 
	\begin{equation*}
		\sum_{\gamma \in S} e^{-s\dist o{\gamma o}}
		\leq
		\norm{\sum_{\gamma \in S} e^{-s\dist o{\gamma o}} \rho(\gamma)}.
	\end{equation*}
	We deduce that for all $s > h_\rho$, $\mathcal P_\Gamma(s) \leq \norm{A(s)}$.
	It follows that $h_\Gamma \leq h_\rho$.
	By \autoref{res: upper bound delta rho},
	we get $h_\Gamma = h_\rho$ and $\norm{A(s)} = \mathcal P_\Gamma(s)$.
	The converse implication follows from of \autoref{res: growth vs almost inv vec - quantified}.
\end{proof}

%
\subsection{Ultra-limit of Hilbert spaces}
%
\label{sec: ultra-limits}

Inspired by the standard Patterson-Sullivan construction we are going to build in the next section a linear functional on $C(\bar X_h)$ which we think of as an operator valued measure on $\bar X_h$.
Since $\bar X_h$ is compact, the set of positive \emph{real valued} measures on $\bar X_h$ is compact for the weak-$\ast$ topology.
This is no more the case for general vector valued measure.
To bypass this difficulty we let our measures converge in a bigger space obtained as the ultra-limit of a sequence of Banach spaces.
This section reviews the main properties of ultra-limit of Banach spaces.
For more details see Dru\c tu-Kapovich \cite[Chapter~19]{Drutu:2018vu}.

\medskip
A \emph{non-principal ultra-filter} is a finitely additive map $\omega : \mathcal P(\N) \rightarrow \{0,1\}$ such that $\omega(\N) = 1$ and which vanishes on every finite subset of $\N$.
A property $P_n$ is true \oas if
\begin{equation*}
	\omega \left( \set{ n \in \N }{ P_n\ \text{is true} } \right) = 1.
\end{equation*}
A real sequence $\left(u_n\right)$ is \oeb if there exists $M$ such that $\left|u_n\right| \leq M$, \oas.
Given $\ell \in \R$, we say that the $\omega$-limit of $\left(u_n\right)$ is $\ell$ and write $\limo u_n = \ell$ if for all $\varepsilon >0$, we have $\abs{u_n -\ell} \leq \varepsilon$, \oas.
Any sequence which is \oeb admits a $\omega$-limit \cite{Bou71}.

\medskip
Let $(E_n)$ a sequence of Banach spaces.
We define a restricted product by
\begin{equation*}
	\prod_\omega E_n
	= \set{(\phi_n) \in \prod_{n \in \N}E_n}{\norm{\phi_n}\ \text{is \oeb}}
\end{equation*}
Pointwise addition and scalar multiplication define a vector space structure on this set.
We define a pseudonorm by
\begin{equation*}
	\norm{(\phi_n)} = \limo \norm{\phi_n}.
\end{equation*}

\begin{defi}
	\label{def: ultra-limit banach space}
	The \emph{$\omega$-limit} of $(E_n)$, denoted by $\limo E_n$ or simply $E_\omega$, is the quotient of $\prod_\omega E_n$ by the equivalence relation which identifies two sequences $(\phi_n)$ and $(\phi'_n)$ whenever $\norm{(\phi_n) - (\phi'_n)} = 0$.
\end{defi}

The vector space structure on $\prod_\omega E_n$ passes to the quotient and turns $E_\omega$ into a vector space.
Similarly the pseudonorm on $\prod_\omega E_n$ defines a norm on $\limo E_n$ for which $E_\omega$ is complete \cite[Preliminaries]{Papasoglu:1996tc}.
Hence $E_\omega$ is a Banach space.
In addition, if for every $n \in \N$, the space $E_n$ is a Hilbert space, then so is $E_\omega$.
Indeed, the parallelogram law only involves four points, thus it passes to the limit \cite[Corollary~19.3]{Drutu:2018vu}.

\paragraph{Notation.}
If $(\phi_n)$ is a sequence in $\prod_\omega E_n$ we denote its image in $\limo E_n$ by $\limo \phi_n$.

\medskip
Let $(E_n)$ and $(F_n)$ be two sequences of Banach spaces.
Let
\begin{equation*}
	E_\omega = \limo E_n \quad \text{and}\quad F_\omega = \limo F_n.
\end{equation*}
For every $n \in \N$, the space $\mathcal B(E_n,F_n)$ of bounded linear operator from $E_n$ to $F_n$ is a Banach space.
In particular, we can consider the limit space $\limo \mathcal B(E_n,F_n)$.
Given an element $A = \limo A_n$ in $\limo \mathcal B(E_n,F_n)$, one defines an operator $\iota(A)$ in $\mathcal B(E_\omega, F_\omega)$ as follows.
For every $\phi = \limo \phi_n$ in $E_\omega$, we let
\begin{equation*}
	\iota(A) \phi = \limo \left[A_n\phi_n\right].
\end{equation*}
One checks easily that $\iota(A)$ is well-defined.
In particular it does not depend on the choice of the sequences $(A_n)$ or $(\phi_n)$.
The resulting map
\begin{equation*}
	\iota \colon \limo \mathcal B(E_n,F_n) \to \mathcal B(E_\omega, F_\omega)
\end{equation*}
is both a linear map and an isometric embedding.
As $\limo \mathcal B(E_n,F_n)$ is complete, its image is closed.
In this article, we will omit the map $\iota$ and see $\limo \mathcal B(E_n,F_n)$ as a closed linear subspace of $\mathcal B(E_\omega, F_\omega)$.
Similarly $\limo \mathcal B(E_n)$ embeds as a closed \emph{subalgebra} of $\mathcal B(E_\omega)$.
This leads to the following statement.

\begin{prop}
	\label{res: limit of rep}
	Let $\Gamma$ be a group.
	Let $(\mathcal \rho_n)$ be a sequence of unitary representations of $\Gamma$ into  a Hilbert space $\mathcal H_n$.
	There exists a unique unitary representation ${\rho_\omega \colon \Gamma \to \mathcal U(\mathcal H_\omega)}$ such that for every $\gamma \in \Gamma$, for every element $\phi = \limo \phi_n$ of $\mathcal H_\omega$ we have
	\begin{equation*}
		\rho_\omega(\gamma)\phi = \limo \left[ \rho_n(\gamma) \phi_n \right].
	\end{equation*}
	It is denoted by $\rho_\omega=\lim_\omega\rho_n$, and called the \emph{(ultra-)limit representation}.
\end{prop}

\paragraph{Lattice structure.}
Assume now that each space $E_n$ comes with a partial order $\prec$ that turns $E_n$ into a Banach lattice.
We define $E_\omega^+$ as the set
\begin{equation*}
	E_\omega^+ = \set{\limo \phi_n \in E_\omega}{\phi_n \in E_n^+,\ \text{\oas}}.
\end{equation*}
It is a positive convex cone.
Hence one can define a partial order on $E_\omega$ by declaring that $\phi \prec \phi'$ if $\phi' - \phi \in E_\omega^+$.

\begin{lemm}[Dru\c tu-Kapovich {\cite[Proposition~19.12]{Drutu:2018vu}}]
	\label{res: limit lattice}
	The ordered vector space ${(E_\omega, \normV, \prec)}$ is a Banach lattice.
\end{lemm}

\begin{lemm}
	\label{res: limit of positive rep}
	Let $\Gamma$ be a group.
	Let $(\mathcal \rho_n)$ be a sequence of unitary representations of $\rho_n \colon \Gamma \to \mathcal U(\mathcal H_n)$ into a Hilbert lattice $\mathcal H_n$.
	If $\rho_n$ is positive \oas, then so is the limit representation $\rho_\omega = \limo \rho_n$.
\end{lemm}

\begin{proof}
	It directly follows from the definition of $\rho_\omega$.
\end{proof}

%
\subsection{Conformal family of operator valued measures}
%
\label{sec: conf family of hilbert functionals}

The next sections are dedicated to the proof of \autoref{res: growth vs almost inv vec - quantified}.

\paragraph{Setting.}
Let $(X,d)$ be a proper geodesic hyperbolic space.
We fix once and for all a base point $o \in X$.
Recall that $\bar X$ stands for the Gromov compactification of $X$ whereas $\bar X_h$ is its horocompactification.
Let $\Gamma$ be a group acting properly by isometries on $X$.
We assume that this action is strongly positively recurrent.
Let $\omega$ be a non-principal ultra-filter.
For every $n \in \N$, we fix a Hilbert lattice $\mathcal H_n$, as well as a unitary positive representation ${\rho_n \colon \Gamma \to \mathcal U(\mathcal H_n)}$.
We denote by $A_n(s)$ the formal series
\begin{equation*}
	A_n(s) = \sum_{\gamma \in \Gamma} e^{-s\dist{\gamma o}o} \rho_n(\gamma)
\end{equation*}
We set
\begin{equation*}
	h_\omega = \limo h_{\rho_n}.
\end{equation*}
We are going to prove that if $h_\omega = h_\Gamma$, then $\rho_\omega$ has a non-zero invariant vector (\autoref{res: non-zero inv vector limit rep}). 
This will imply \autoref{res: growth vs almost inv vec - quantified}.

\paragraph{Remark.}
To prove \autoref{res: main}, we  can assume  $\rho_n$ is constantly equal to the Koopmann representation associated to the right action of $\Gamma$ on $\Gamma' \backslash \Gamma$. 
In this case $h_\omega = h_\rho$.
Nevertheless, the quantified version of our main theorem as stated in \autoref{res: growth vs almost inv vec - quantified} requires this level of full generality.
Note that even if $(\rho_n)$ is a constant sequence, we cannot avoid using ultra-filters.
Indeed, as our Hilbert spaces are not locally compact, ultra-limit of Hilbert spaces provides a convenient tool to make bounded sequences converge.

\paragraph{Weighted Poincar\'e series.}
Since the action of $\Gamma$ is strongly positively recurrent, the standard Poincar\'e series $\mathcal P_\Gamma(s)$ is divergent at the critical exponent $s = h_\Gamma$, see \autoref{res: PS charges radial limit set}.
However there is no reason that the sequence $\norm{A_n(s)}$ should diverge, at $s = h_{\rho_n}$.
We bypass this difficulty by adapting the usual Patterson argument \cite[Lemma~3.1]{Patterson:1976hp}.

\begin{lemm}
	\label{res: patterson trick}
	Let $(s_n)$ be a sequence converging to $h_\omega$.
	There exists a non decreasing map $\theta \colon \R_+ \to \R_+$ with the following properties.
	\begin{enumerate}
		\item \label{enu: patterson trick - slowly increasinsing}
		      For every $\varepsilon \in \R_+^*$, there exists $t_0 \in \R_+$ such that for every $u \in \R_+$ and $t \geq t_0$, one has $ \theta(t+u) \leq e^{\varepsilon u} \theta(t)$.
		\item \label{enu: patterson trick - div}
		      The operator series
		      \begin{equation*}
			      A'_n(s) = \sum_{\gamma \in \Gamma} \theta(\dist{\gamma o}o) e^{-s \dist{\gamma o}o} \rho_n(\gamma)
		      \end{equation*}
		      is bounded  whenever $s > h_{\rho_n}$ and unbounded  whenever $s < h_{\rho_n}$.
		\item \label{enu: patterson trick - norm}
		      The sequence $\norm{A'_n(s_n)}$ diverges as $n$ approaches infinity.
	\end{enumerate}
\end{lemm}

\begin{proof}
	Recall that both $(h_{\rho_n})$ and $(s_n)$ converge to $h_\omega$.
	Hence we can find be a decreasing sequence $(\varepsilon_n)$ of positive numbers converging to zero, such that $s_n - \varepsilon_n < h_{\rho_n}$ for every $n \in \N$.
	We are going to build by induction an increasing sequence $(t_n)$ diverging to infinity and a map $\theta \colon \R_+ \to \R_+$ whose restriction to $[t_n,t_{n+1}]$ is logarithmically affine with slope $\varepsilon_n$.
	We start by letting $t_0 = 0$ and $\theta(t_0) = 1$.
	Let $n \in \N$.
	Assume now that $t_n$ and $\theta$ restricted to $[t_0,t_n]$ have already be defined.
	By assumption, the series $A_n(s)$ is divergent at $s = s_n - \varepsilon_n$.
	Consequently there exists $t_{n+1} > t_n +1$ such that
	\begin{equation}
		\label{eqn: patterson trick}
		\norm{\sum_{\gamma \in S_n } e^{-(s_n-\varepsilon_n) \dist{\gamma o}o} \rho_n(\gamma) } \geq n,
	\end{equation}
	where
	\begin{equation*}
		S_n  = \set{\gamma \in \Gamma}{t_n < \dist{\gamma o}o \leq t_{n+1}}.
	\end{equation*}
	We define $\theta$ on $]t_n,t_{n+1}]$ by $\theta(t) = e^{\varepsilon_n (t-t_n)}\theta(t_n)$.
	This complete the induction step.
Points~(\ref{enu: patterson trick - slowly increasinsing}) and (\ref{enu: patterson trick - div}) are proved exactly as for regular Patterson-Sullivan measures. 
By construction,
	\begin{equation*}
		\norm{\sum_{\gamma \in S_n} \theta(\dist{\gamma o}o) e^{-s_n\dist{\gamma o}o} \rho_n(\gamma)}
		= \norm{\sum_{\gamma \in S_n} \theta(t_n)e^{-(s_n-\varepsilon_n)\dist{\gamma o}o} \rho_n(\gamma)} .
	\end{equation*}
	Consequently, (\ref{eqn: patterson trick}) yields
	\begin{equation*}
		\norm{A'_n(s_n)}
		\geq \norm{\sum_{\gamma \in S_n} \theta(\dist{\gamma o}o) e^{-s_n\dist{\gamma o}o} \rho_n(\gamma)}
		\geq n.
	\end{equation*}
	Hence the sequence $\norm{A'_n(s_n)}$ diverges as $n$ approaches infinity, whence (\ref{enu: patterson trick - norm}).
\end{proof}

\paragraph{A limit of bounded operators.}
We fix once for all a sequence $(s_n)$ converging to $h_\omega$ as well as a slowly increasing function $h \colon \R_+ \to \R_+$ as in \autoref{res: patterson trick}.
Following the exposition of \autoref{sec: ultra-limits}, let
$\mathcal H_\omega = \limo \mathcal H_n$ be the limit Hilbert space (\autoref{def: ultra-limit banach space}) and $\rho_\omega = \limo \rho_n$ the limit representation (\autoref{res: limit of rep}).

\medskip
Let $x \in X$.
For every $n \in \N$, we define a linear map
\begin{equation*}
	a^\rho_{x,n} \colon C(\bar X_h) \to \mathcal B(\mathcal H_n)
\end{equation*}
as follows. 
For every $f \in C(\bar X_h)$, define
\begin{equation}
	\label{eqn: def bxn}
	a^\rho_{x,n} (f) = \frac 1{\norm{A'_n(s_n)}} \sum_{\gamma \in \Gamma} \theta(\dist x{\gamma o}) e^{-s_n\dist x{\gamma o}} f(\gamma o) \rho(\gamma).
\end{equation}
By \autoref{res: patterson trick},  
there exists a parameter $C(x)$ (which does not depends on $n$) such that for every $\gamma \in \Gamma$,
\begin{equation*}
	\theta(\dist x{\gamma o}) \leq C(x) \theta(\dist o{\gamma o}).
\end{equation*}
Let $f \in C(\bar X_h)$.
Using the previous inequality, we observe that the series defining $a^\rho_{x,n} (f)$ is  bounded  for every $n \in \N$.
Moreover its norm is bounded above by
\begin{equation}
	\label{eqn: bxn unif bounded}
	\norm{ a^\rho_{x,n}(f)} \leq C(x)e^{s_n\dist xo} \norm[\infty] f.
\end{equation}
We  define a bounded operator of $\mathcal H_\omega$ by
\begin{equation*}
	a^\rho_x (f) = \limo a^\rho_{x,n}(f).
\end{equation*}
This provides a positive continuous linear functional
\begin{equation}
	\label{eqn: map bx}
	a^\rho_x \colon C(\bar X_h) \to \mathcal B(\mathcal H_\omega).
\end{equation}

\paragraph{Remark.}
This functional can be interpreted as an \emph{operator-valued} measure on $\bar X_h$.
Indeed by construction for every continuous function $f \in C(\bar X_h)$ with $f \geq 0$, the associated operator $a_x^\rho(f)$ is positive.
It follows that $a_x^\rho$ takes its values in the set $\mathcal L_r(\mathcal H_\omega)$ of regular
 operators on $\mathcal H_\omega$, which is an order-complete vector lattice (\autoref{res: app - regular op lattice}). 
By  Wright \cite[Theorem~1]{Wright:1971bu}, there exists a unique quasi-regular $\mathcal L_r(\mathcal H_\omega)$-valued Borel measure on $\bar X_h$ such that for every $f \in C(\bar X_h)$, the operator $a_x^\rho(f)$ equals the integral of $f$ against this measure.
We refer the reader to \cite{Wright:1971bu} and the references therein for the theory of lattice-valued measures.
We can hence see $a_x^\rho$ as
an operator-valued measure. However, it is  simpler to
express all the properties of this measure in terms of the functional $a_x^\rho$.
Still, it justifies the following terminology.

\begin{defi}\label{def: twisted Patterson Sullivan}
	We call the family $(a^\rho_x)_{x\in X}$ the \emph{twisted Patterson-Sullivan measure}  associated to $\Gamma$ and $\rho = (\rho_n)_{n\in \mathbb N}$.
\end{defi}

The core of the proof of \autoref{res: growth vs almost inv vec - quantified} consists in understanding the properties of this twisted Patterson-Sullivan measure. 
 This new powerful definition and the study below are the main novelty  of our paper. It has been inspired from the weighted Patterson-Sullivan measures of the thermodynamical formalism on the one hand, see \cite{Babillot:1998wd, BabillotTSG, Paulin:2015ud} but also several papers of Sambarino as 
\cite{Sambarino:2014} or all his later works, and from weighted Ruelle operators on the other hand, as for example \cite{Bowen:2008bx} and particularly the twisted Ruelle operators used in \cite{Coulon:2017vz}.

\paragraph{First properties.}

The following statement translates the well-known properties of usual Patterson-Sullivan measures in this context.

\begin{theo}
\label{res: standard properties twisted PS} 
	Let $\Gamma$ be a discrete group acting properly by isometries on a Gromov-hyperbolic space $X$.
	With the previous notations, the family $(a^\rho_x)_{x\in X}$ is $\rho_\omega$-equivariant, $h_{\omega}$-conformal, gives full support to $\partial_h X$ and is normalized at $o$.
\end{theo}

The proof of this theorem, as well as its precise meaning, is detailed in Lemmas \ref{res: b - total mass} to \ref{res: b - conformality}.

\begin{lemm}[Normalization]
	\label{res: b - total mass}
	The operator $a^\rho_o(\mathbf 1)$ has norm $1$.
\end{lemm}

\begin{proof}
	By construction, for every $n \in \N$, if $f=\mathbf 1\in C(\bar X_h)$, the operator
	\begin{equation*}
		a^\rho_{o,n}(\mathbf 1) = \frac 1{\norm{A'_n(s_n)}} A'_n(s_n)
	\end{equation*}
	has norm $1$.
	We get the result by passing to the limit.
\end{proof}

\begin{lemm}[Support]
	\label{res: b - support}
	Let $x \in X$.
	Let $f \in C(\bar X_h)$.
	If its support is contained in $X$, then $a^\rho_x(f) = 0$.
\end{lemm}

\begin{proof}
	As the support of $f$ is a compact subset of $X$, there exists a finite subset $S$ of $\Gamma$ such that $\supp f \cap \Gamma o \subset S o$.
	Consequently, for every $n \in \N$, we have
	\begin{equation*}
		a^\rho_{x,n}(f) = \frac 1{\norm{A'_n(s_n)}} \sum_{\gamma \in S} \theta(\dist x{\gamma o})e^{-s_n\dist x{\gamma o}}f(\gamma o) \rho(\gamma).
	\end{equation*}
	This finite sum is uniformly bounded whereas $\norm{A'_n(s_n)}$ diverges to infinity (\autoref{res: patterson trick}).
	Passing to the limit, we get $a^\rho_x(f) = 0$.
\end{proof}

\begin{lemm}[$ \rho_\omega$-equivariance]
	\label{res: b - equivariance}
	Let $x \in X$ and $\gamma \in \Gamma$.
	For every $f \in C(\bar X_h)$ we have
	\begin{equation*}
		a^\rho_{\gamma x} (f)=\rho_\omega(\gamma)a^\rho_x(f \circ \gamma).
	\end{equation*}
\end{lemm}

\begin{proof}
	Let $f \in C(\bar X_h)$.
	A direct computation shows that for every $n \in \N$, we have
	\begin{equation*}
		a^\rho_{\gamma x, n}(f)=\rho_n(\gamma)a^\rho_{x,n}(f\circ \gamma).
	\end{equation*}
	The result follows by taking the $\omega$-limit.
\end{proof}

Let $x,y \in X$.
Recall that a point in the horoboundary $\partial_h X$ of $X$ can  be seen as a cocycle $b$.
With this in  mind, we define for  $s \in \R_+$, a map $\chi^s_{x,y} \in C(\bar X_h)$.
If $z \in X$, then
\begin{equation*}
	\chi^s_{x,y}(z) = \frac{\theta(\dist xz)}{\theta(\dist yz)} e^{-s[\dist xz - \dist yz]}
\end{equation*}
If $b \in \partial_h X$, then
\begin{equation*}
	\chi^s_{x,y} (b) = e^{-sb(x,y)}.
\end{equation*}
The sequence $(\chi^{s_n}_{x,y})$ uniformly converges to $\chi^{h_\omega}_{x,y}$.

\begin{lemm}[$h_\omega$-conformality]
	\label{res: b - conformality}
	Let $x,y \in X$.
	For every $f \in C(\bar X_h)$, we have
	\begin{equation*}
		a^\rho_x(f)=a^\rho_y\left(\chi^{h_\omega}_{x,y}f\right).
	\end{equation*}
\end{lemm}

\begin{proof}
	A standard computation shows that for every $n \in \N$,
	\begin{equation*}
		a^\rho_{x,n}\left(f\right) = a^\rho_{y,n}\left(\chi^{s_n}_{x,y}f\right).
	\end{equation*}
	Consequently
	\begin{equation*}
		a^\rho_xf = \limo a^\rho_{y,n}\left(\chi^{s_n}_{x,y}f\right).
	\end{equation*}
	On the other hand, by the very definition of $a^\rho_y$ we have
	\begin{equation*}
		a^\rho_y\left(\chi^{h_\omega}_{x,y}f\right) = \limo a^\rho_{y,n}\left(\chi^{h_\omega}_{x,y}f\right).
	\end{equation*}
	Hence it suffices to prove that
	\begin{equation*}
		\limo a^\rho_{y,n} \left( \left[\chi^{h_\omega}_{x,y} - \chi^{s_n}_{x,y}\right]f \right) = 0.
	\end{equation*}
	The norm of $a^\rho_{y,n}$, as a linear map form $C(\bar X_h)$ to $\mathcal B(\mathcal H_n)$, is uniformly bounded -- see (\ref{eqn: bxn unif bounded}).
	In particular, there exists $M \in \R_+$ such that for every $n \in \N$,
	\begin{equation*}
		\norm {a^\rho_{y,n} \left( \left[\chi^{h_\omega}_{x,y} - \chi^{s_n}_{x,y}\right]f \right)}
		\leq M  \norm[\infty] {\left[\chi^{h_\omega}_{x,y} - \chi^{s_n}_{x,y}\right]f}.
	\end{equation*}
	The result follows from the fact that $(\chi^{s_n}_{x,y}$) uniformly converges to $\chi^{h_\omega}_{x,y}$.
\end{proof}

As a corollary of the above lemmas, we get the following useful formula, for every $\gamma\in \Gamma$ and $f\in C(\bar X_h)$:
\begin{equation}
\label{eqn:useful-conformity}
	a^\rho_{\gamma o}(f)=\rho_\omega(\gamma)a^\rho_o(f\circ\gamma)=a^\rho_o(\chi_{\gamma o,o}^{h_\omega} f).
\end{equation}

%
\subsection{Twisted measure on the Gromov boundary}
%

We study now how the family $(a^{\rho}_x)$ -- thought as a family of measures  on $\partial_hX$ -- behaves compared to usual Patterson-Sullivan measures $\nu_x$.
From a dynamical point of view, it is more appropriate to work in the Gromov boundary $\partial X$ rather than in the horoboundary $\partial_h X$.
Therefore, we push forward the family $(a_x^\rho)$ by the natural $\Gamma$-equivariant continuous   map $\pi \colon \bar{X}_h \to \bar X$. For every $x \in X$, we set
\begin{equation*}
	\begin{array}{rccc}
		\pi_\ast a^{ \rho }_x \colon & C\left(\bar X\right) & \to & \mathcal B \left(\mathcal H_\omega\right) \\
		                                  & f                    & \mapsto & a^{ \rho}_x(f \circ \pi).
	\end{array}
\end{equation*}
	It follows from the previous study that $\pi_\ast a^\rho_o(\mathbf 1)$ has norm $1$ (\autoref{res: b - total mass}) and the support of $\pi_\ast a^\rho_x$ is contained in $\partial X$ for every $x \in X$ (\autoref{res: b - support}).
	Since $\pi \colon \bar X_h \to \bar X$ is $\Gamma$-equivariant, the family $(\pi_\ast a^\rho_x)$ is $\rho_\omega$-equivariant (\autoref{res: b - equivariance}).
	Let us now focus on the conformality of $\pi_\ast a^\rho_x$ which is slightly more technical.
 
\begin{lemm}[$h_\omega$-quasi-conformality]
\label{res: b - quasi-conformality}
There exists $C \in \R_+^*$ with the following property.
Let $x,y \in X$ and $\xi \in \partial X$.
There is a neighbourhood $V_\xi \subset \bar X$ of $\xi$ such that for every cocycle $b \in \pi^{-1}(\xi)$, for every $f \in C(\bar X)$ whose support is contained in $V$ we have
		\begin{equation*}
			\frac 1C \pi_\ast a^\rho_x(f) \prec e^{-h_\omega b(x,y)}a^\rho_y(f) \prec C a^\rho_x(f).
		\end{equation*}
	\end{lemm}

\begin{proof}
Let $x,y \in X$.
Let $\xi \in \partial X$.
Using the hyperbolicity of $X$, we observe that there exists a neighbourhood $V_\xi \subset \bar X$ such that for every $b \in \pi^{-1}(\xi)$ the following holds:
if $z$ in a point in $V_\xi \cap X$ then
\begin{equation*}
\abs{\left[\dist zx - \dist zy\right] - b(x,y)} \leq 100\delta;
\end{equation*}
moreover if $b'$ is a cocycle in $\pi^{-1}(V_\xi) \cap \partial_h X$, then $\abs{b'(x,y) - b(x,y)} \leq 100\delta$.
We now fix $b \in \pi^{-1}(\xi)$ and $f \in C(\bar X)$ whose support is contained in $V_\xi$.
Let $\varepsilon > 0$.
Since $\theta$ is a slowly increasing function, there exists $t_0 \in \R_+$ such that for every $t \geq t_0$ and $u \geq 0$, we have $\theta(t+u) \leq e^{\varepsilon u}\theta(t)$.
		We fix a map $g \colon \bar X \to [0,1]$ whose support is contained in $X$ and whose restriction to $B(x,t_0)$ and $B(y,t_0)$ is constant equal to $1$.
		It follows from \autoref{res: b - support} that both $\pi_\ast a^\rho_x( gf)$ and $\pi_\ast a^\rho_y(gf)$ vanish.
		Consequently it suffices to compare $\pi_\ast a^\rho_x(f')$ and $\pi_\ast a^\rho_y(f')$ where $f' = (1-g)f$.
		Using the conformality of $(a^\rho_x)$ we get
		\begin{equation}
			\label{eqn: b - quasi-conformality}
			\pi_\ast a^\rho_x\left(f'\right)
			= a^\rho_x\left(f' \circ \pi\right)
			= a^\rho_y\left(\chi_{x,y}^{h_\omega} f' \circ \pi\right).
		\end{equation}
		It follows from our choice of $t_0$ and $V_\xi$ that for every $z \in \bar X_h$ lying in the support of $f'$ we have
		\begin{equation*}
			\frac 1{C(\varepsilon)} e^{-h_\omega b(x,y)}
			\leq \chi_{x,y}^{h_\omega}(z)
			\leq C(\varepsilon) e^{-h_\omega b(x,y)},
		\end{equation*}
		where $C(\varepsilon) = e^{100h_\omega \delta} e^{\varepsilon \dist xy}$.
		Since $a^\rho_y$ is a positive linear functional, (\ref{eqn: b - quasi-conformality}) becomes
		\begin{equation*}
			\frac 1{C(\varepsilon)} e^{-h_\omega b(x,y)}\pi_\ast a^\rho_x\left(f'\right)
			\prec e^{-h_\omega b(x,y)}a^\rho_y\left(f' \circ \pi\right)
			\prec C(\varepsilon)\pi_\ast a^\rho_x\left(f'\right),
		\end{equation*}
		hence
		\begin{equation*}
			\frac 1{C(\varepsilon)} e^{-h_\omega b(x,y)}\pi_\ast a^\rho_x\left(f\right)
			\prec e^{-h_\omega b(x,y)}\pi_\ast a^\rho_y\left(f \circ \pi\right)
			\prec C(\varepsilon)\pi_\ast a^\rho_x\left(f\right).
		\end{equation*}
		This inequality holds for every $\varepsilon \in \R_+^*$, consequently
		\begin{equation*}
			\frac 1C \pi_\ast a^\rho_x\left(f\right)
			\prec e^{-h_\omega b(x,y)}\pi_\ast a^\rho_y\left(f \circ \pi\right)
			\prec C \pi_\ast a^\rho_x\left(f\right),
		\end{equation*}
		where $C = e^{100 \delta h_\omega }$ is a universal parameter.
	\end{proof}

\paragraph{Remark.}
Note that if $h_\omega < h_\Gamma$, then the operator valued measures $(\pi_\ast a_x^\rho)$ cannot have bounded variation -- see \cite[Chapter~1]{Diestel:1977ui} for a definition. Indeed otherwise their variations would be a $\Gamma$-invariant, $h_\omega$-quasi-conformal family of measures on $\partial X$.
Such measures do not exists unless $h_\omega \geq h_\Gamma$ \cite[Corollaire~6.6]{Coo93}.
Later we will use a Radon-Nikodym derivative theorem for $\pi_\ast a_o^\rho$.
This observation somehow tells us that all the theory exposed in \cite{Diestel:1977ui} does not apply here unless $h_\omega = h_\Gamma$.

\paragraph{Shadow lemma.}
\begin{lemm}[Half shadow lemma]
	\label{res: b - shadow lemma}
	For every $r \in \R_+$, there exists $C \in \R_+$, with the following property.
	Let $\gamma \in \Gamma$ and $f \in C^+(\bar X)$.
	If the support of $f$ is contained in $\mathcal O_o(\gamma o,r)$, then
	\begin{equation*}
		\norm{\pi_\ast a^\rho_o(f)} \leq C e^{-h_\omega \dist o{\gamma o}} \norm[\infty] f.
	\end{equation*}
\end{lemm}

\begin{proof}
	Combining Lemmas~\ref{res: b - equivariance} and \ref{res: b - conformality} we observe that
	\begin{equation*}
		\rho_\omega(\gamma^{-1})\pi_\ast a^\rho_o(f)
		= a^\rho_{\gamma^{-1}o}(f \circ \pi \circ \gamma)
		= a^\rho_o\left( \chi^{h_\omega}_{\gamma^{-1}o,o}  f\circ \pi \circ \gamma\right).
	\end{equation*}
	For simplicity we set
	\begin{equation*}
		f_\gamma = \chi^{h_\omega}_{\gamma^{-1}o,o} f\circ \pi  \circ \gamma.
	\end{equation*}
	Let $\varepsilon > 0$.
	By \autoref{res: patterson trick}, there exists $t_0 \in \R_+$ such that for  $t \geq t_0$ and $u \geq 0$,
	\begin{equation*}
		\theta(t+u) \leq e^{\varepsilon u} \theta(t).
	\end{equation*}
	We fix a continuous map $g \colon X \to [0,1]$, with compact support whose restriction to $B(o,t_0)$ is constant equal to $1$.
	It allows to decompose $f_\gamma$ as $f_\gamma = g f_\gamma + (1 - g)f_\gamma$.
	Since the support of $g f_\gamma$ is contained in $X$ we have $a^\rho_o(g f_\gamma) = 0$ (\autoref{res: b - support}).
	Consequently
	\begin{equation*}
		\rho_\omega(\gamma^{-1})a^\rho_o(f) = a^\rho_o( (1 - g) f_\gamma).
	\end{equation*}
	Let us now consider a point $z \in \bar X_h$ in the support of $(1 - g) f_\gamma$.
	By construction $z$ belongs to $\pi^{-1}(\mathcal O_{\gamma^{-1}o}(o,r)) \setminus B(o,t_0)$.
	If $z=b \in \partial_h X$ is a cocycle, then
	\begin{equation*}
		b(\gamma^{-1} o,o) \geq \dist o{\gamma o} - 2r.
	\end{equation*}
	On the other hand, if $z\in X$, then
	\begin{equation*}
		\dist {\gamma^{-1}o}z - \dist oz \geq \dist o{\gamma o} - 2r.
	\end{equation*}
	In addition $\dist o{\gamma o} \geq t_0$, thus according to our choice of $t_0$,
	\begin{equation*}
		\theta\left(\dist {\gamma^{-1}o}z\right) \leq \theta(\dist o{\gamma o} + \dist oz) \leq e^{\varepsilon \dist o{\gamma o}} \theta(\dist oz).
	\end{equation*}
	In both cases, we get

	\begin{equation*}
		\chi^{h_\omega}_{\gamma^{-1}o,o}(z) \leq  e^{2h_\omega r} e^{-(h_\omega - \varepsilon) \dist o{\gamma o}}.
	\end{equation*}
	Hence
	\begin{equation*}
		0 \leq (1-g)f_\gamma \leq e^{2h_\omega r} e^{-(h_\omega - \varepsilon)\dist o{\gamma o}} \norm[\infty] f \mathbf 1.
	\end{equation*}
	Since $a^\rho_o$ is a positive functional, we get
	\begin{equation*}
		\rho_\omega(\gamma^{-1})\pi_\ast a^\rho_o(f) \prec e^{2h_\omega r} e^{-(h_\omega - \varepsilon) \dist o{\gamma o}}\norm[\infty] f a^\rho_o(\mathbf 1).
	\end{equation*}
	Recall that $\rho_\omega$ is a unitary representation.
	Taking the norm, we get
	\begin{equation*}
		\norm{\pi_\ast a^\rho_o(f)} \leq C e^{-(h_\omega - \varepsilon) \dist o{\gamma o}},
	\end{equation*}
	where $C = e^{2h_\omega r} \norm{a^\rho_o(\mathbf 1)}$.
	As it holds for all $\varepsilon > 0$, the result follows.
\end{proof}

%
\subsection{Absolute continuity}
%
\label{sec: absolute continuity}

\paragraph{Radial limit set.} 
Let $K$ be a compact subset of $X$.
Recall that the $K$-radial limit set $\Lambda_{\rm{rad}}^K$ is the set of all points $\xi \in \partial X$ for which there exists a geodesic ray $c \colon \R_+ \to X$ ending at $\xi$ whose image $c(\R_+)$ intersects infinitely many copies $\gamma K$ of $K$.
As  explained before, we think of $\pi_\ast a^\rho_o$ as an operator valued measure on $\bar X$.
The next step consists in proving that this ``measure'' gives full mass to $\Lambda_{\rm{rad}}^K$ for some compact $K$ (\autoref{res: twisted PS charges radial limit set}).
This is probably the \emph{most crucial} point in the proof.
Indeed,  Shadow Lemmas \autoref{res: shadow lemma} and \autoref{res: b - shadow lemma} tell us that when $h_\omega = h_\Gamma$, the measures $\pi_\ast a^\rho_o$ and $\nu_o$ can be compared on shadows.
As both measures give full measure to $\Lambda_{\rm rad}^K$ for closed ball $K= \bar B(o,r)$ with fixed $r>0$, a Vitali type argument, approximating any Borel set by a union of shadows, allows to deduce that $\pi_\ast a^\rho_o$ is absolutely continuous with respect to $\nu_o$ (\autoref{res: twisted absolute continuity}).
\autoref{res: twisted PS charges radial limit set} is the only place where we use in an essential way the fact that the action of $\Gamma$ on $X$ is strongly positively recurrent. All other arguments in the article work  under a weaker assumption (e.g. if the geodesic flow is conservative).

\medskip
The proof of the next statements follows exactly the same steps as the one of \autoref{res: PS charges radial limit set}. It relies on the same auxiliary sets $\mathcal L_K$ and $U_K^T$ defined in \autoref{sec: spr radial limit set}.
However since it is the only place where we use (in a crucial way!) the existence of a growth gap at infinity to get our main theorem, we decided to detail it here.

\begin{prop}
	\label{res: twisted measure of UTK}
	Assume that $h_\Gamma^\infty < h_\omega$.
	There exists a compact subset $K$ of $X$ and numbers $\alpha, C, T_0 \in \R_+^*$ such that for every $T \geq T_0$, for every $f \in C^+(\bar X)$ whose support is contained in $U_K^T$, we have
	\begin{equation*}
		\norm{\pi_\ast a^\rho_o(f)} \leq Ce^{-\alpha T}\norm[\infty] f.
	\end{equation*}
\end{prop}

\begin{proof}
	By assumption, there exists a compact subset $k$ of $X$ containing $o$ such that $h_{\Gamma_k} < h_\omega$.
	Let $K$ be the $7\delta$-neighbourhood of $k$.
	By \autoref{res: open around AK covered by shadow},
	there exists a finite subset $S$ of $\Gamma$ and a number $r \in \R_+$ such that for every $T \in \R_+$,
	\begin{equation}
	\label{eqn: twisted measure of UTK}
		U_K^T \cap \Gamma o \subset \bigcup_{\substack{\beta \in S\Gamma_k \\ \dist o{\beta o} \geq T - r}} \mathcal O_o(\beta o ,r) .
	\end{equation}
	We fix $\varepsilon > 0$ such that $h_\omega -2\varepsilon > h_{\Gamma_k}$.
	Define $F$ as
	\begin{equation*}
		F = \set{\gamma \in \Gamma}{\dist o{\gamma o} \leq t_0}.
	\end{equation*}
	\medskip
	Let $T \geq t_0 + r$, and  $f \in C^+(\bar X_h)$ be a non-negative function whose support is contained  $U_K^T$.
	Up to rescaling $f$, we  assume   $\norm[\infty] f = 1$.
	Let $n \in \N$.
	By (\ref{eqn: twisted measure of UTK}),
	\begin{equation}
		\label{eqn: measure of UTK - a - preintegral f}
		a^\rho_{o,n}(f)
		\prec
		\frac 1{\norm{A'_n(s_n)}} \sum_{\substack{\beta \in S\Gamma_k, \\ \dist o{\beta o} \geq T-r}} \sum_{\substack{\gamma \in \Gamma \\ \gamma o \in \mathcal O_o(\beta o,r)}} \theta(\dist o{\gamma o})e^{-s_n\dist{\beta o}{\gamma o}}\rho_n(\gamma)
	\end{equation}
	Let $\beta \in S\Gamma_k$ such that $\dist o{\beta o} \geq T - r$. 
	As in the proof of \autoref{res: measure of UTK}, when  $y \in \mathcal O_o(\beta o,r)$, 
	\begin{itemize}
		\item if $\dist {\beta o}y \geq t_0 $, then $\displaystyle 	\theta_0(\dist oy) \leq e^{\varepsilon \dist o{\beta o}} \theta_0(\dist {\beta o}y)$, whereas
		\item if  $\dist o{\beta o} \geq t_0$, then $\displaystyle \theta_0(\dist oy)  \leq e^{\varepsilon \dist o{\beta o}}\theta(t_0)$.
	\end{itemize}
	Consequently,
	\begin{equation*}
		a^\rho_{o,n}(f)\prec
		\frac 1{\norm{A'_n(s_n)}} \sum_{\substack{\beta \in S\Gamma_k, \\ \dist o{\beta o} \geq T-r}}
		e^{2sr}e^{-(s_n-\varepsilon) \dist o{\beta o}}\left(\Sigma_1 + \Sigma_2\right),
	\end{equation*}
	where
	\begin{align*}
		\Sigma_1 & = \sum_{\substack{\gamma \in \Gamma \\ \gamma o \in \mathcal O_o(\beta o,r),\ \dist{\beta o}{\gamma o} < t_0}} \theta(t_0)e^{-s_n\dist{\beta o}{\gamma o}}\rho_n(\gamma), \\
		\Sigma_2 & = \sum_{\substack{\gamma \in \Gamma \\ \gamma o \in \mathcal O_o(\beta o,r),\ \dist{\beta o}{\gamma o} \geq t_0}} \theta(\dist {\beta o}{\gamma o})e^{-s_n\dist{\beta o}{\gamma o}}\rho_n(\gamma).
	\end{align*}
	The number of terms in $\Sigma_1$ is at most $\card F$, so that $\|\Sigma_1\|\le \card F \theta(t_0)$, whereas $\norm{\Sigma_2}$ is bounded above by $\norm{A'_n(s_n)}$.
	Combining all these inequalities we get
	\begin{equation*}
		\norm{a^\rho_{o,n}(f)} \leq e^{2s_nr}\left(1 + \frac{\card F \theta(t_0)}{\norm{A'_n(s_n)}}\right) \sum_{\substack{\beta \in S\Gamma_k, \\ \dist o{\beta o} \geq T-r}}e^{-(s_n-\varepsilon) \dist o{\beta o}}.
	\end{equation*}
	After passing to the limit, it becomes
	\begin{equation*}
		\norm{\pi_\ast a^\rho_o(f)} \leq e^{2h_\Gamma r}\sum_{\substack{\beta \in S\Gamma_k, \\ \dist o{\beta o} \geq T - r}}e^{-(h_\omega-\varepsilon) \dist o{\beta o}}.
	\end{equation*}
	Since $h_\omega -2\varepsilon > h_{\Gamma_k}$, we obtain as in (\ref{eqn: measure of UTK - remainder})
	\begin{equation*}
		\norm{\pi_\ast a^\rho_o(f)} \leq B e^{2h_\omega r}e^{-(h_\omega - h_{\Gamma_k} - 2\varepsilon)T}.
	\end{equation*}
	Recall that $B$, $k$, $r$ and $\varepsilon$ do not depend on $T$ or $f$, whence the result.
\end{proof}

\begin{coro}
	\label{res: twisted PS charges radial limit set}
	Assume that $h_\Gamma^\infty < h_\omega$.
	There exists a compact subset $K$ of $X$ such that for every $\varepsilon > 0$, there is a open subset $V \subset \bar X$ containing $\partial X \setminus\Lambda_{\rm{rad}}^K$ with the following property.
	For every $f \in C^+(\bar X)$ whose support is contained in $V$ we have
	\begin{equation*}
		\norm{\pi_\ast a^\rho_o(f)} \leq \varepsilon \norm[\infty] f.
	\end{equation*}
\end{coro}

\begin{proof}
	According to \autoref{res: twisted measure of UTK} there exists a compact subset $K$ of $X$ as well as numbers $C,\alpha,T_0 \in \R_+^*$ such that for every $T \geq T_0$, for every $f \in C^+(\bar X)$ whose support  of $\bar X$ contained in $U_K^T$,
	\begin{equation*}
		\norm{\pi_\ast a^\rho_o(f)} \leq Ce^{- \alpha T}\norm[\infty] f.
	\end{equation*}
	We fix a summable function $w \colon \Gamma \to \R_+^*$ whose sum is $1$.
	Let $\varepsilon > 0$.
	For every $\gamma \in \Gamma$, we fix $T_\gamma \geq T_0$ such that
	\begin{equation*}
		Ce^{-\alpha T_\gamma} \leq \varepsilon w(\gamma)
	\end{equation*}
	and an open subset $V_\gamma$ of $\bar X$ such that
	\begin{equation*}
		\mathcal L_K \subset V_\gamma \subset U_k^{T_\gamma}.
	\end{equation*}
	According to \autoref{res: K-radial limit set vs AK}, the set $\partial X \setminus \Lambda_{\rm{rad}}^K$ is contained in $\Gamma \mathcal L_K$.
	Hence the set
	\begin{equation*}
		V = \bigcup_{\gamma \in \Gamma} V_\gamma
	\end{equation*}
	is an open neighbourhood of $\partial X \setminus \Lambda_{\rm{rad}}^K$.
	Let $f \in C^+(\bar X)$ whose support is contained in $V$.
	Without loss of generality we can assume that $\norm[\infty] f= 1$.
	As this support is compact, it is actually contained in
	\begin{equation*}
		\bigcup_{\gamma \in S}V_\gamma,
	\end{equation*}
	where $S$ is a finite subset of $\Gamma$.
	We fix a a partition of unity, i.e. a family $(g_\gamma)_{\gamma \in S}$ of elements of $C^+(\bar X)$ such that the support of $g_\gamma$ is contained in $V_\gamma$, for every $\gamma \in S$ and
	\begin{equation*}
		\sum_{\gamma \in S} g_\gamma
	\end{equation*}
	is constant equal to $1$, when restricted to the support of $f$.
	Combining \autoref{res: twisted measure of UTK} with our choice of $T_\gamma$, we get
	\begin{equation*}
		\norm{\pi_\ast a^\rho_o(f)}
		\leq \sum_{ \gamma \in S} \norm{\pi_\ast a^\rho_o(f g_\gamma)}
		\leq \sum_{\gamma \in S} Ce^{-\alpha T_\gamma}
		\leq \varepsilon  \sum_{\gamma \in \Gamma} w(\gamma)
		\leq \varepsilon .\qedhere
	\end{equation*}
\end{proof}

\paragraph{A Vitali type argument.}
We now exploit the previous result to prove that whenever $h_\omega = h_\Gamma$ the ``measure'' $\pi_\ast a^\rho_o$ is absolutely continuous with respect to the usual Patterson-Sullivan measure $\nu_o$.
The first lemma is an easy exercise of hyperbolic geometry.
Its proof is left to the reader.

\begin{lemm}
	\label{res: intersecting vs contained shadows}
	There exists $r_1 \in \R_+^*$ with the following property.
	Let $r\geq r_1$.
	Let $x, y \in X$ such that ${\dist ox \leq \dist oy }$.
	If $\mathcal O_o(x,r)$ and $\mathcal O_o(y,r)$ have a non-empty intersection, then $\mathcal O_o(y,r)$ is contained in $\mathcal O_o(x,4r)$.
\end{lemm}

The second lemma is a Vitali like Lemma.
\begin{lemm}[Vitali's Lemma]\label{lem:Vitali}
	Let $K$ be a compact subset of $X$.
	There exists $r_1 \in \R_+^*$ such that for every $r \geq r_1$, for every $R \in \R_+$, there exists a subset $S$ of $\Gamma$ with the following properties.
	\begin{enumerate}
		\item For all $\alpha\in S$, $\dist o{\alpha o} \geq R$.
		\item The union $\displaystyle  \bigcup_{ \alpha \in S} \mathcal O_o(\alpha o, 4r)$ covers $\displaystyle \Lambda_{\rm rad}^K$.
		\item The shadows  $(\mathcal O_o(\alpha o, r))_{\alpha \in S}$ are pairwise disjoint.
	\end{enumerate}
\end{lemm}

\begin{proof}
	Let $r_1$ be the parameter given by \autoref{res: intersecting vs contained shadows}. 
	Without loss of generality, we can assume that $r_1 \geq \diam {K \cup \{ o\}}$.
	Let $r \geq r_1$ and $R \in \R_+$.
	For simplicity we set
	\begin{equation*}
		U_R = \set{\gamma \in \Gamma}{\dist o{\gamma o} \geq R}.
	\end{equation*}
	We build the set $S$ by induction, adding one element at each step.
	We start with $S_0 = \emptyset$.
	For every $n \in \N$, we define the set $S_{n+1}$ by adding to $S_n$ the element $\gamma \in U_R\setminus S_n$ such that $\mathcal O_o(\gamma o,r)$ is disjoint from all the previous shadows $(\mathcal O_o(\alpha o,r))_{\alpha \in S_n}$ and which minimizes $\dist o{\gamma o}$.
	Standard elementary arguments using \autoref{res: intersecting vs contained shadows}  show that  the increasing union of all $S_n$  satisfies the above statement.
\end{proof}

\begin{prop}
	\label{res: twisted absolute continuity}
	Assume that $h_\omega = h_\Gamma$.
	There exists $C \in \R_+^*$ such that for every $f \in C(\bar X)$,
	\begin{equation*}
		\norm{\pi_\ast a^\rho_o(f)} \leq C  \int_{\partial X} \abs f d \nu_o.
	\end{equation*}
\end{prop}

As already mentioned, this proposition is a direct consequence of Shadow lemmas.
Indeed, the key \autoref{res: twisted PS charges radial limit set} allows to approximate every Borel set by unions of shadows of fixed radius,  through a Vitali type argument.

\begin{proof}
	Let $K$ be the compact subset of $\bar X$ given by \autoref{res: twisted PS charges radial limit set}.
	Fix $r \geq \max\{ r_0,r_1\}$ where $r_0$ and $r_1$ are
	respectively given by Lemmas~\ref{res: shadow lemma} and \ref{res: intersecting vs contained shadows}.
	By Shadow Lemmas \ref{res: shadow lemma} and \ref{res: b - shadow lemma}, there exists $C_0 \in \R_+^*$ such that for every $\gamma \in \Gamma$,
	\begin{itemize}
		\item $\displaystyle \nu_o\left( \mathcal O_o(\gamma o,r)\right) \geq \frac 1{C_0} e^{-h_\Gamma \dist o{\gamma o}}$
		\item for every $f \in C^+(\bar X)$ whose support is contained in $\mathcal O_o(\gamma o, 4r)$ we have
		      \begin{equation*}
			      \norm{\pi_\ast a^\rho_o(f)} \leq C_0e^{-h_\omega \dist o{\gamma o}}\norm[\infty] f.
		      \end{equation*}
	\end{itemize}
	Let $f \in C(\bar X)$.
	We first assume that $f$ is non-negative.
	Let $\varepsilon > 0$.
	We fix some auxiliary subsets of $X$ to decompose the map $f$ into a sum of functions supported on appropriate small shadows.
	Since the action of $\Gamma$ is strongly positively recurrent, $h_\Gamma^\infty < h_\Gamma = h_\omega$.
	According to \autoref{res: twisted PS charges radial limit set} there exists an open set $V$ containing $\partial X \setminus \Lambda_{\rm{rad}}^K$ such that for every $g \in C(\bar X)$ whose support is contained in $V$, we have
	\begin{equation*}
		\norm{\pi_\ast a^\rho_o(g)} \leq \varepsilon \norm[\infty] g.
	\end{equation*}
	Since $f$ is continuous, for all $\varepsilon>0$, there exists $R>0$ such that on any shadow $\mathcal O_o(y,4r)$, with $d(o,y)\ge R$, the variations of $f$ are bounded by $\varepsilon$. 
	Let $S$ be the collection of elements of $\Gamma$ given by Vitali's \autoref{lem:Vitali}. 
	Since $f$ is continuous, there exists a finite subset $S_0$ of $S$ such that the support of $f$ is contained in
	\begin{equation*}
		\left(\bigcup_{\gamma \in S_0} \mathcal O_o(\gamma o, 2r) \right) \cup V.
	\end{equation*}
	We now fix a partition of unity, i.e. a collection $\{ g \} \cup \{g_\gamma\}_{\gamma \in S_0}$ of continuous functions from $\bar X$ to $[0,1]$ such that the support of $g_\gamma$ (\resp $g$) is contained in $\mathcal O_o(\gamma o, 4r)$ (\resp $V$) and
	\begin{equation*}
		g + \sum_{\gamma \in S_0} g_\gamma
	\end{equation*}
	is constant equal to $1$ when restricted to the support of $f$.
	We now first estimate $\norm{\pi_\ast a^\rho_o(f)}$ from above.
	The triangle inequality yields
	\begin{equation*}
		\norm{\pi_\ast a^\rho_o(f)}
		\leq \norm{\pi_\ast a^\rho_o(gf)} + \sum_{\gamma \in S_0} \norm{\pi_\ast a^\rho_o(g_\gamma f)}.
	\end{equation*}
	By \autoref{res: twisted PS charges radial limit set},  $\norm{\pi_\ast a^\rho_o(gf)} \leq \varepsilon \norm[\infty] f$.
	For every $\gamma \in S_0$ we let
	\begin{equation*}
		f_\gamma = \sup_{x \in \mathcal O_o(\gamma o, 2r)} f(x).
	\end{equation*}
	so that $\norm[\infty]{g_\gamma f} \leq f_\gamma$.
	It follows  from the Half-Shadow \autoref{res: b - shadow lemma} that
	\begin{equation*}
		\sum_{\gamma \in S_0} \norm{\pi_\ast a^\rho_o(g_\gamma f)}
		\leq C_0\sum_{\gamma \in S_0} e^{-h_\omega \dist o{\gamma o}}f_\gamma.
	\end{equation*}
	Consequently
	\begin{equation}
		\label{res: twisted absolute continuity - up}
		\norm{\pi_\ast a^\rho_o(f)} \leq \varepsilon \norm[\infty] f + C_0 \sum_{\gamma \in S_0} e^{-h_\omega \dist o{\gamma o}}f_\gamma.
	\end{equation}
	Let us now estimate $\nu_o(f)$ from below.
	Let $\gamma \in S_0$.
	Since $\dist {\gamma o}o \geq R$, the map $f$ restricted to $\mathcal O_o(\gamma o, 4r)$ varies by at most $\varepsilon$.
	On the other hand the shadows $(\mathcal O_o(\gamma o, r))_{\gamma \in S}$ are pairwise disjoint.
	We get from the standard Shadow Lemma
	\begin{equation}
		\label{res: twisted absolute continuity - low}
		\begin{split}
			\int f d\nu_0
			\geq \sum_{\gamma \in S_0} \int_{O_o(\gamma o,r)}   f d\nu_0
			& \geq \sum_{\gamma \in S_0} (f_\gamma - \varepsilon) \nu_o\left(\mathcal O_o(\gamma o, r)\right) \\
			& \geq \frac 1{C_0} \sum_{\gamma \in S_0} f_\gamma e^{-h_\Gamma \dist o{\gamma o}} - \varepsilon.
		\end{split}
	\end{equation}
	Recall that $h_\omega = h_\Gamma$.
	Hence combining (\ref{res: twisted absolute continuity - up}) and (\ref{res: twisted absolute continuity - low}) yields
	\begin{equation*}
		\norm{\pi_\ast a^\rho_o(f)}
		\leq \varepsilon \norm[\infty] f  + C_0^2 \left(\int f d\nu_o + \varepsilon\right).
	\end{equation*}
	This inequality holds for every $\varepsilon > 0$, hence
	\begin{equation*}
		\norm{\pi_\ast a^\rho_o(f)}
		\leq C_0^2 \int f d\nu_o .
	\end{equation*}
	If $f$ is not nonnegative anymore,  decomposing $f$ into its positive and negative part  leads immediately to the result.
\end{proof}

\begin{coro}
	\label{res: twisted - radon-nikodym}
	Assume that $h_\omega = h_\Gamma$.
	There exists a unique continuous linear map $D \colon \mathcal H_\omega \to L^\infty((\partial X,\nu_0), \mathcal H_\omega)$ such that for every $\phi \in \mathcal H_\omega$, for every $f \in C(\bar X)$,   we have
	\begin{equation*}
		\pi_\ast a^\rho_o(f)\phi = \int f D(\phi) \ d\nu_o.
	\end{equation*}
\end{coro}

\paragraph{Remark.}
The integral in the statement is an integral in the sense of Bochner (\autoref{sec: bochner space}).
The map $D$ can be thought as a kind of Radon-Nikodym derivative of $\pi_\ast a^\rho_o$ with respect to $\nu_o$.

\begin{proof}
	Let $C$ be the constant given by \autoref{res: twisted absolute continuity}.
	Let $\phi \in \Phi$.
	It follows from \autoref{res: twisted absolute continuity} that for every $f \in C(\bar X)$ we have
	\begin{equation*}
		\norm{\pi_\ast a^\rho_o(f)\phi} \leq C \int \abs f d\nu_0 \norm \phi.
	\end{equation*}
	Thus the map sending $f \in C(\bar X)$ to $\pi_\ast a^\rho_o(f)\phi$ extends to a continuous map $L^1(\partial X,\nu_o) \to \mathcal H_\omega$, whose norm is at most $C\norm \phi$.
	As a Hilbert space, $\mathcal H_\omega$ is reflexive, hence satisfies the Radon-Nikodym property (\autoref{res: app - reflexive implies RN}).
	Consequently there exists a vector $D(\phi) \in L^\infty((\partial X,\nu_0),\mathcal H_\omega)$, whose norm is at most $C\norm{\phi}$ such that for every $f \in C(\bar X)$ we have
	\begin{equation*}
		\pi_\ast a^\rho_o(f)\phi = \int f D(\phi) \ d\nu_o.
	\end{equation*}
	This construction defines a map $D \colon \mathcal H_\omega \to L^\infty((\partial X,\nu_o),\mathcal H_\omega)$.
	Uniqueness and linearity of $D$ follow from \autoref{res: app - unicity Bochner}.
	By construction, $\norm[\infty]{D(\phi)} \leq C \norm{\phi}$, for every $\phi \in \mathcal H_\omega$.
	Hence $D$ is continuous.
\end{proof}

%
\subsection{Invariant vectors}
%
\label{sec: invariant vectors}

From now on we assume that $h_\omega = h_\Gamma$.
The goal is now to study the map $D \colon \mathcal H_\omega \to L^\infty((\partial X,\nu_0), \mathcal H_\omega)$ given by \autoref{res: twisted - radon-nikodym}.

Heuristically the idea is the following.
Using the ergodicity of the action of $\Gamma$ on $(\partial^2X, \mu)$ we are going to prove that $D(\phi)$ is almost surely constant, so that viewed as a measure with values in $\mathcal B(\mathcal H_\omega)$, the twisted Patterson-Sullivan measure $\pi_\ast a_o^\rho$ satisfies 
\begin{equation*}
	\pi_\ast a_o^\rho(f) \phi = D(\phi) \int f  d\nu_o,\quad  \forall f \in C(\bar X).
\end{equation*}
Comparing the invariance of $\nu_o$ and $\pi_\ast a_o^\rho$, we will observe that $D(\phi)$ is a $\rho_\omega$-invariant vector, that is  a limit of $\rho_n$ almost-invariant vectors.
Below is a rigorous exposition of this strategy.

\medskip
Fix $\phi \in \mathcal H_\omega^+$.
For simplicity, set $\Psi = D(\phi)$.
Recall that $\Psi$ is a bounded map from $\bar X$ to $\mathcal H_\omega$.
Actually it directly follow from \autoref{res: b - support} that the support of $\Psi$ is contained in $\partial X$.
Since $\phi$ is positive, $\Psi$ takes its values in $\mathcal H_\omega^+$ (\autoref{res: app - bochner integral positive}).
 
\begin{lemm}
	\label{res: transfer conf + equiv}
	There exists $C \in \R_+^*$, which does not depend on $\phi$, such that for every $\gamma \in \Gamma$, we have
	\begin{equation*}
		\frac 1C \Psi \prec \rho_\omega(\gamma) \Psi \circ \gamma^{-1} \prec C \Psi.
	\end{equation*}
\end{lemm}

\paragraph{Remark.}
Comparing pointwise two functions defines an order which endows $L^\infty((\partial X,\nu_o),\mathcal H_\omega)$ with a lattice structure (\autoref{res: app - bochner banach lattice}).
The inequalities in the lemma are meant in $L^\infty((\partial X,\nu_o),\mathcal H_\omega)$.

\begin{proof}
We first fix a measurable section of $\pi$
		\begin{equation*}
			\begin{array}{rccc}
				\sigma\colon & \partial X & \to     & \partial_h X \\
				             & \xi        & \mapsto & b_\xi.
			\end{array}
		\end{equation*}
		Since $(\nu_x)$ is $h_\Gamma$-quasi-conformal,
		there exists $C_0 \in \R_+^*$ such that for every $\gamma \in \Gamma$, for $\nu_o$-almost every $\xi \in \partial X$, we have
		\begin{equation}
			\label{eqn: transfer conf + equiv - PS measure}
			\frac 1{C_0} e^{-h_\Gamma b_\xi(\gamma o,o)}
			\leq \frac{d \gamma_*\nu_o}{d\nu_o}(\xi)
			\leq C_0  e^{-h_\Gamma b_\xi(\gamma o,o)}
		\end{equation}
		We denote by $C_1 \in \R_+^*$ the universal constant given by the $h_\omega$-quasi-conformality of $(\pi_\ast a^\rho_x)$ (\autoref{res: b - quasi-conformality}).
		Let $\gamma \in \Gamma$.
		We are going to work with the points $x = \gamma o$ and $y = o$.
		For every $\xi \in \partial X$, we write $V_\xi$ for the neighbourhood of $\xi$ given by \autoref{res: b - quasi-conformality}.
		Up to decreasing $V_\xi$ we can always assume that for any $b,b' \in \pi^{-1}(V_\xi) \cap \partial_hX$,
		\begin{equation*}
			\abs{b(x,y) - b'(x,y)} \leq 100\delta.
		\end{equation*}
	
	Let $f \in C(\bar X)$.
	Since the support of $f$ is compact, there exists a finite subset $S$ of $\partial X$ such that this support is contained in
	\begin{equation*}
		\left( \bigcup_{ \xi \in S} V_\xi \right) \cup X.
	\end{equation*}
	We now fix a partition of unity, i.e. a collection of maps $g \colon \bar X \to [0,1]$ and $g_\xi \colon \bar X \to [0,1]$ (one for each $\xi \in S$) such that the support of $g$ (\resp $g_\xi$) is contained in $X$ (\resp $V_\xi$) and the sum
	\begin{equation*}
		g + \sum_{\eta \in S} g_\eta
	\end{equation*}
	equals $1$ when restricted to the support of $f$.
	Since the support of $gf$ is contained in $X$, we have $\pi_\ast a^\rho_{\gamma o}(gf) = 0$.
	Hence the $\rho_\omega$-equivariance (\autoref{res: b - equivariance}) of $(\pi_\ast a^\rho_x)$ yields
	\begin{equation*}
		\rho_\omega(\gamma)\int (f\circ \gamma) \Psi d\nu_o
		= \rho_\omega(\gamma) \pi_\ast a^\rho_o(f\circ \gamma ) \phi
		= \pi_\ast a^\rho_{\gamma o}(f) \phi
		= \sum_{\eta \in S} \pi_\ast a^\rho_{\gamma o}(g_\eta f) \phi
	\end{equation*}
	Combined with the $h_\omega$-quasi-conformality (\autoref{res: b - quasi-conformality}) of $(\pi_\ast a^\rho_x)$,we get
	\begin{equation*}
		\rho_\omega(\gamma)\int (f\circ \gamma) \Psi d\nu_o
		\prec C_1 \sum_{\eta \in S} e^{-h_\omega b_\eta(\gamma o,o)} a^\rho_o(g_\eta f)\phi.
	\end{equation*}
	This inequality can be written using the definition $\Psi$ as
	\begin{equation*}
		\rho_\omega(\gamma) \int (f \circ \gamma) \Psi d\nu_o
		\prec C_1 \left( \sum_{\eta \in S} e^{-h_\omega b_\eta(o, \gamma o)} \int g_\eta f \Psi d\nu_o \right)
	\end{equation*}
	Recall now   first that the support of $\nu_o$ is contained in $\partial X$, second that for every $\xi$ in the support of $g_\eta$ the quantities $b_\xi(\gamma o,o)$ and $b_\eta(\gamma o,o)$ differ by at most $100\delta$.
	Consequently \autoref{res: app - bochner integral positive} gives
	\begin{align*}
		\rho_\omega(\gamma) \int (f \circ \gamma) \Psi d\nu_o
		 & \prec C_1e^{100h_\omega \delta} \left( \sum_{\eta \in S} \int g_\eta(\xi) f(\xi)  \Psi(\xi) e^{-h_\omega b_\xi(\gamma o,o)} d\nu_o(\xi) \right) \\
		 & \prec C_1e^{100h_\omega \delta}\int f(\xi)  \Psi(\xi) e^{-h_\omega b_\xi(\gamma o,o)} d\nu_o(\xi)
	\end{align*}
	Recall that $h_\omega = h_\Gamma$.
	Hence the invariance and quasi-conformality of $(\nu_x)$ yields
	\begin{equation*}
		\rho_\omega(\gamma) \int (f \circ \gamma) \Psi d\nu_o
		\prec C_0C_1e^{100h_\omega \delta} \int (f \circ \gamma) (\Psi \circ \gamma)  d\nu_o
	\end{equation*}
	Note that this inequality holds for every $f \in C(\bar X)$, hence $\rho(\gamma) \Psi \prec C (\Psi \circ \gamma)$ where $C = C_0C_1e^{100h_\omega \delta} $ is a universal constant (\autoref{res: app - bochner integral positive reverse}).
	The other inequality follows by symmetry.
	 
\end{proof}

If $\partial X$ and $\partial_h X$ coincide, all the Patterson-Sullivan measures are $\Gamma$-equivariant and conformal (not just quasi-conformal).
Hence our argument proves that for every $\gamma \in \Gamma$, we have
\begin{equation*}
	\rho_\omega(\gamma) \Psi \circ \gamma^{-1} = \Psi.
\end{equation*}
When the two boundaries differ we do not have quite equality.
To deal with this problem, we proceed as follows.
By \autoref{res: transfer conf + equiv} there exists $\varepsilon \in (0,1)$, which does not depend on $\phi$, such that the set
\begin{equation*}
	\set{\rho_\omega (\gamma) \Psi \circ \gamma^{-1}}{\gamma \in \Gamma}.
\end{equation*}
is non-empty and bounded below by $\varepsilon \Psi$.
We define $\Psi' \in L^\infty((\partial X, \nu_o),\mathcal H_\omega)$ as its greatest lower bound, i.e.
\begin{equation*}
	\Psi' = \inf_{\gamma \in \Gamma} \rho_\omega (\gamma) \Psi \circ \gamma^{-1}.
\end{equation*}
Such an element is well-defined as $L^\infty((\partial X,\nu_o),\mathcal H_\omega)$ is countably order complete (\autoref{res: app - bochner banach lattice - completeness}).
By construction
\begin{equation}
	\label{eqn: inequality psi'}
	\varepsilon \Psi \prec \Psi' \prec \Psi.
\end{equation}
In  particular, $\Psi'$ takes its values in $\mathcal H^+_\omega$.
Moreover, for every $\gamma \in \Gamma$ we have
\begin{equation}
	\label{eqn: invariance psi'}
	\rho_\omega (\gamma) \Psi' \circ \gamma^{-1} = \Psi'.
\end{equation}

\begin{lemm}
	\label{res: psi' constant}
	The function $\Psi'\in L^\infty((\partial X,\nu_o),\mathcal H_\omega)$ is constant $\nu_o$-almost everywhere.
\end{lemm}

\begin{proof}
	According to (\ref{eqn: invariance psi'}) for every $\gamma \in \Gamma$, for $\nu_o$-almost every $\eta, \xi \in \partial X$, we have
	\begin{equation*}
		\sca{\Psi'(\gamma\eta)}{\Psi'(\gamma\xi)}
		= \sca{\rho_\omega(\gamma)\Psi'(\eta)}{\rho_\omega(\gamma)\Psi'(\xi)}
		= \sca{\Psi'(\eta)}{\Psi'(\xi)}.
	\end{equation*}
	It exactly means that the map
	\begin{equation*}
		\begin{array}{rccc}
			Q \colon & (\partial X\times \partial X, \nu_o\otimes \nu_o) & \to & \R_+                          \\
			         & (\eta, \xi)                                       & \to & \sca{\Psi'(\eta)}{\Psi'(\xi)}
		\end{array}
	\end{equation*}
	is $\Gamma$-invariant.
	Recall now that by \autoref{res: ergodicity BM}, the action of $\Gamma$ on the space ${(\partial X\times \partial X, \nu_o\otimes \nu_o)}$ is ergodic. 
	The map $Q$ is hence constant $\nu_o\otimes \nu_o$-almost everywhere.
	We write $m \in \R$ for this value.
	Observe now that for every $f_1,f_2 \in L^1(\nu_o)$ we have
	\begin{equation}
		\label{eqn: psi' constant - fubini}
			\sca{\int f_1\Psi' d\nu_o}{\int f_2\Psi' d\nu_o}= m \left(\int f_1d\nu_o \right) \left( \int  f_2d\nu_o\right).
	\end{equation}
	 A standard argument using the equality case of the Cauchy-Schwarz inequality shows that there exists $\psi'_0 \in \mathcal H_\omega$ such that for every $f \in L^1_+(\nu_o)$ we have
	\begin{equation*}
		\int_{\bar X} f\Psi' d\nu_o = \sqrt m \left[\int_{\bar X}  fd\nu_o\right] \psi'_0.
	\end{equation*}    
	Consequently $\Psi'$ is  $\nu_0$-almost surely constant,  equal to $\sqrt m\psi'_0$ (\autoref{res: app - unicity Bochner}).
\end{proof}

\begin{lemm}
	\label{res: psi' rho-inv}
	The unique essential value of $\Psi'$ is a  $\rho_\omega$-invariant vector of $\mathcal H_\omega$.
\end{lemm}

\begin{proof}
	As we proved in \autoref{res: psi' constant}, $\Psi'$ is constant $\nu_o$-almost surely.
	To avoid ambiguity we write $\psi' \in \mathcal H_\omega$ for its value.
	Recall that for every $\gamma \in \Gamma$ we have $\rho_\omega (\gamma) \Psi' \circ \gamma^{-1} = \Psi'$, see (\ref{eqn: invariance psi'}).
	Replacing $\Psi'$ by its value exactly says that $\psi'$ is $\rho_\omega$-invariant.
\end{proof}

\paragraph{Remark.}
If the horoboundary $\partial_hX$ coincides with the Gromov boundary $\partial X$, our arguments prove that there exists a $\rho_\omega$-invariant vector $\psi \in \mathcal H_\omega^+$ such that for every $f \in C(\bar X)$, we have
\begin{equation*}
	a^\rho_o(f)\phi = \left( \int f d\nu_o \right) \psi.
\end{equation*}

\medskip
Next proposition summarizes the results of this section.
\begin{prop}
	\label{res: non-zero inv vector limit rep}
	If $h_\omega = h_\Gamma$, then the representation $\rho_\omega$ has non-zero invariant vectors.
\end{prop}

\begin{proof}
	The operator $\pi_\ast a^\rho_o(\mathbf 1)$ has norm $1$ (\autoref{res: b - total mass}).
	Hence there exists a vector $\phi \in \mathcal H_\omega^+$ such that $\pi_\ast a^\rho_o(\mathbf 1)\phi$ is not zero.
	To such a vector we associate a bounded function $\Psi \colon \partial X \to \mathcal H_\omega^+$ such that for every $f \in C(\bar X)$
	\begin{equation*}
		\pi_\ast a^\rho_o(f) \phi = \int f \Psi d\nu_o.
	\end{equation*}
	In particular $\Psi$ is a non-zero function.
	We proved that the map $\Psi' \colon \partial X \to \mathcal H_\omega^+$ defined by
	\begin{equation*}
		\Psi' = \inf_{\gamma \in \Gamma} \rho(\gamma) \Psi\circ \gamma^{-1}.
	\end{equation*}
	is constant and its value $\psi'$ is $\rho_\omega$-invariant (\autoref{res: psi' constant}).
	Moreover there exists $\varepsilon \in (0,1)$, which does not depend on $\phi$, such that $\varepsilon \Psi \prec \Psi' \prec \Psi$ (\autoref{res: transfer conf + equiv}).
	It follows from this inequality that $\psi'$ is non-zero.
	Indeed otherwise $\Psi'$ and thus $\Psi$ would be zero as well.
\end{proof}

\medskip
We complete this section with the proof of \autoref{res: growth vs almost inv vec - quantified}.

\begin{proof}[Proof of \autoref{res: growth vs almost inv vec - quantified}]
	The proof proceeds by contradiction.
	Let $S$ be a finite subset of $\Gamma$ and $\varepsilon \in \R_+^*$.
	Assume that the theorem is false.
	For each $n \in \N$, we can find a Hilbert lattice $\mathcal H_n$ and a positive representation $\rho_n \colon \Gamma \to \mathcal U(\mathcal H_n)$ with the following properties.
	\begin{enumerate}
		\item $(h_{\rho_n})$ converges to $h_\Gamma$.
		\item For every $n \in \N$, the representation does not have any $(S, \varepsilon)$-invariant vector.
	\end{enumerate}
	Let $\omega$ be a non-principal ultra-filter.
	We let $\mathcal H_\omega = \limo \mathcal H_n$ and denote by $\rho_\omega \colon \Gamma \to \mathcal U(\mathcal H_\omega)$ the limit representation induced by $(\rho_n)$.
	Observe that we are exactly in the setting of \autoref{sec: conf family of hilbert functionals}.
	Moreover
	\begin{equation*}
		h_\omega = \limo h_{\rho_n} = h_\Gamma.
	\end{equation*}
	It follows from \autoref{res: non-zero inv vector limit rep} that $\rho_\omega$ admits an invariant unit vector $\psi$ that we can write $\psi = \limo \psi_n$, where $\psi_n$ is a unit vector in $\mathcal H_n$.
	By definition of the representation $\rho_\omega$,  for all $\gamma \in \Gamma$, we have
	\begin{equation*}
		\limo \norm{\rho_n(\gamma) \psi_n - \psi_n} = 0.
	\end{equation*}
	Since $S$ is finite, it forces
	\begin{equation*}
		\sup_{\gamma \in S} \norm{\rho_n(\gamma) \psi_n - \psi_n} < \varepsilon,\  \text{\oas}.
	\end{equation*}
	Hence $\psi_n$ is an $(S,\varepsilon)$-invariant vector of $\rho_n$ \oas, which contradicts the definition of $\rho_n$.
\end{proof}

%
\section{Applications to group theory} \label{sec:groups}
%

Let $X$ be a hyperbolic proper geodesic space.
Let $\Gamma$ be a group acting by isometries on $X$.
Let $\mathcal H$ be a Hilbert space and $\rho \colon \Gamma \to \mathcal U(\mathcal H)$ be a unitary representation.
Let $S$ be a finite subset of $\Gamma$ and $\varepsilon > 0$.
Recall that an $(S, \varepsilon)$-invariant vector is a vector $\phi \in \mathcal H$ such that
\begin{equation*}
	\sup_{\gamma \in S} \norm{\rho(\gamma) \phi - \phi} < \varepsilon \norm \phi.
\end{equation*}
Moreover, the representation $\rho$ almost admits invariant vectors  if for every finite subset $S$ of $\Gamma$ for every $\varepsilon > 0$, 
it has an $(S, \varepsilon)$-invariant vector.
We now investigate the consequences of \autoref{res: growth vs almost inv vec - quantified} by varying the representations of~$\Gamma$.

\medskip
Our main source of applications deals with the growth of subgroups of $\Gamma$.
Let $\Gamma'$ be a subgroup of $\Gamma$.
We denote by $Y$ the space of  left   cosets
$Y = \Gamma'\backslash \Gamma$ on which $\Gamma$ acts on the  right.
Let $\mathcal H = \ell^2(Y)$ be the space of square summable map
$Y \to \R$ endowed with its usual Hilbert structure and order (\autoref{exa : sq sum function})
We denote by $\rho \colon \Gamma \to \mathcal U(\mathcal H)$
the corresponding Koopman representation.
Recall that $h_\rho$ is the critical exponent of the operator series
\begin{equation*}
	A(s)  = \sum_{\gamma \in \Gamma} e^{-s \dist{\gamma o}o} \rho(\gamma),
\end{equation*}
whereas $h_\Gamma$ is the exponential growth rate of $\Gamma$ (for its action on $X$).

\begin{lemm}\label{res: lower bound delta rho}
	The critical exponents $h_\rho$ and $h_{\Gamma'}$
	satisfy $h_{\Gamma'} \leq h_\rho$.
\end{lemm}

\begin{proof}
	Let $s > h_\rho$.
	We write $y_0$ for the point of $Y$ corresponding to the coset $\Gamma'$ and $\psi \in \ell^2(Y)$ for the Dirac mass at $y_0$.
	Note that $\rho(\gamma)\psi = \psi$, for every $\gamma \in \Gamma'$.
	Hence
	\begin{equation*}
		\mathcal P_{\Gamma'}(s)\psi
		=
		\sum_{\gamma \in \Gamma'} e^{-s \dist{\gamma o}o} \rho(\gamma)\psi
		\prec \sum_{\gamma \in \Gamma} e^{-s \dist{\gamma o}o} \rho(\gamma)\psi
		= A(s)\psi.
	\end{equation*}
	Consequently $\mathcal P_{\Gamma'}(s)$ converges.
	This statement holds for every $s > h_\rho$, hence the result $h_\rho \geq h_{\Gamma'}$.
\end{proof}

\paragraph{Remark.}
In the next sections we explore various properties of groups defined in terms of unitary representations.
These properties make sense for locally compact groups.
However we restrict ourselves to discrete groups as they are the only ones that we consider  in this article.

%
\subsection{Amenability}
%

\paragraph{Amenability.}
There are numerous equivalent definition of amenability.
The most suitable for  our purpose can be formulated in terms of the regular representation.

\begin{defi}
	\label{def: amenability}
	The action of a discrete group $\Gamma$ on a set $Y$ is \emph{amenable} if and only if the induced representation $\rho \colon \Gamma \to \mathcal U(\ell^2(Y))$ almost admits invariant vectors.
	A subgroup $\Gamma'$ of $\Gamma$ is \emph{co-amenable in $\Gamma$} if the action of $\Gamma$ on $Y = \Gamma' \backslash \Gamma$ is amenable.
\end{defi}

The action of $\Gamma$ on $Y$ is amenable if and only if one of the following equivalent facts holds.
\begin{itemize}
	\item (Invariant mean) There exists a $\Gamma$-invariant positive mean on the set $\ell^\infty(Y)$.
	\item (F\o lner sets) For every finite subset $S$ of $\Gamma$, for every $\varepsilon > 0$, there exists a finite subset $Y_0$ of $Y$ such that
	      \begin{equation*}
		      \sup_{\gamma \in S} \frac{\card{\gamma Y_0 \Delta  Y_0}}{\card {Y_0}} \leq \varepsilon.
	      \end{equation*}
	\item (Reiter's criterion) For every finite subset $S$ of $\Gamma$, for every $\varepsilon > 0$ there exists a non-zero map $L_+^1(Y)$ such that
	      \begin{equation*}
		      \sup_{\gamma \in S} \norm{ f\circ \gamma - f} \leq \varepsilon \norm f
	      \end{equation*}
\end{itemize}
The proof for amenable actions works verbatim as for amenable groups, see for instance \cite[Appendix~G]{BekHarVal08} or \cite{Juschenko:2015}. Another reference for amenable action is \cite{Eymard:1972ug}.
We can now prove our main theorem, which we recall.

\begin{theo}
\label{theo:mainbis}
	Let $(X,d)$ be a hyperbolic proper geodesic space.
	Let $\Gamma$ be a group acting  properly  by isometries on $X$ and $\Gamma'$ a subgroup of $\Gamma$.
	Assume that the action of $\Gamma$ is  strongly positively recurrent. 
The following are equivalent.
	\begin{enumerate}
		\item  \label{res: mainbis - exp}
		$h_{\Gamma'} = h_\Gamma$
		\item \label{res: mainbis - moy}
		The subgroup $\Gamma'$ is co-amenable in $\Gamma$.
	\end{enumerate}
\end{theo}

\paragraph{From critical exponent to amenability.}
We start with the proof of the implication (\ref{res: mainbis - exp})~$\Rightarrow$~(\ref{res: mainbis - moy}).
Assume that $h_{\Gamma'} = h_\Gamma$.
Since $h_{\Gamma'} \leq h_\rho \leq h_\Gamma$ (Lemmas~\ref{res: upper bound delta rho} and \ref{res: lower bound delta rho}) we have $h_\rho = h_\Gamma$.
It follows from \autoref{res: growth vs almost inv vec} that $\rho$ almost has invariant vectors, which exactly means that $\Gamma'$ is co-amenable in $\Gamma$.

\paragraph{From amenability to critical exponents.}

We now focus on the so called ``easy direction'', i.e. (\ref{res: mainbis - moy})~$\Rightarrow$~(\ref{res: mainbis - exp}).
As explained in the introduction, if $\Gamma'$ is a \emph{normal} subgroup of $\Gamma$, then Roblin's proof for CAT($-1$) spaces \cite{Roblin:2005fn} directly extends to our setting.
However if $\Gamma'$ is no more a normal subgroup, we are not aware of any existing proof in the literature that would work in the general context of Gromov hyperbolic spaces.
We expose here a strategy  based on the approach of Coulon-Dal'bo-Sambusetti \cite{Coulon:2017vz} revisited through ideas of Roblin-Tapie \cite{Roblin:2013bh}.

Let $\Gamma'$ be a subgroup of $\Gamma$.
We denote by $Y = \Gamma' \backslash \Gamma$ the space of left cosets of $\Gamma'$.
The strategy is to estimate in terms of $h_{\Gamma'}$ the spectral radius of a certain random walk on the space $Y$.
When $\Gamma'$ is co-amenable in $\Gamma$, Kesten's amenability criterion tells us that any random walk on $Y$ as spectral radius $1$, which leads to the expected relation between $h_{\Gamma'}$ and $h_\Gamma$.

\medskip
We begin with general considerations on random walks.
Let $\mathcal F(Y, \C)$ be the set of \emph{all} maps from $Y$ to $\C$ and $\mathcal H = \ell^2(Y)$ the subset consisting of all square summable functions with its canonical Hilbert space structure.
The group $\Gamma$ acts on the right on $Y$ inducing a left action of $\Gamma$ on $\mathcal F(Y, \C)$ as follows.
For every $\phi \in \mathcal F(Y, \C)$, for every $\gamma \in \Gamma$,
\begin{equation*}
	[\gamma \cdot \phi](y) = \phi (y \gamma), \quad\forall y \in Y.
\end{equation*}
Restricted to $\mathcal H$, this action defines a unitary representation $\rho \colon \Gamma \to \mathcal U(\mathcal H)$.
Let $p \colon \Gamma \to [0,1]$ be a \emph{symmetric} probability measure on $\Gamma$ with \emph{finite} support.
The convolution by $p$ defines a operator $M$ on $\mathcal F(Y, \C)$ given by
\begin{equation}
	\label{eqn: easy dir - markov op}
	M\phi = \phi \ast p = \sum_{\gamma \in \Gamma} p(\gamma) \left[\gamma^{-1}\cdot \phi \right].
\end{equation}
Its restriction to $\mathcal H$, still denoted by $M$, is  the \emph{Markov operator} of the random walk on $Y$ associated to $p$.
Seen as operator of $\mathcal H$, the spectral radius $\tau(M)$ of $M$ is at most $1$.
The \og easy direction \fg\ of Kesten's amenability criterion tells us that if $\Gamma'$ is co-amenable in $\Gamma$ then $\tau(M) = 1$.
Our first task is to relate $\tau(M)$ to the critical exponents of $\Gamma'$.
To that end we use a discrete version of Barta's inequality \cite{Barta:1937tf} exposed in the next two statements.

\begin{lemm}
	\label{res: easy dir - barta trick}
	Let $u, \phi \colon \Gamma \to \R_+$ be two non negative maps.
	Then
	\begin{equation*}
		\sca{M(u\phi)}{u\phi} \leq \sca{u^2}{\phi M\phi}.
	\end{equation*}
\end{lemm}

\paragraph{Remark.}
We do not assume that $u$ or $\phi$ are square summable.
In particular we allow the above scalar products to be infinite.

\begin{proof}
	Assume first that both $u$ and $\phi$ have finite support, so that all objects in the following computations are well-defined.
	Observe that
	\begin{equation*}
		\sca{M(u\phi)}{u\phi} -\sca{u^2}{\phi M\phi}
		= \sum_{y \in Y}\sum_{\gamma \in \Gamma} \left[u\left(y\gamma^{-1}\right)-u(y)\right]u(y)\phi\left(y\gamma^{-1}\right)\phi(y) p(\gamma).
	\end{equation*}
	Recall that $p$ is symmetric.
	Reindexing the double sum provides another way to write this difference, namely
	\begin{align*}
		\MoveEqLeft{\sca{M(u\phi)}{u\phi} -\sca{u^2}{\phi M\phi}}                                                                                                         \\
		 & = \sum_{y \in Y}\sum_{\gamma \in \Gamma} \left[u(y)-u\left(y\gamma^{-1}\right)\right]u\left(y\gamma^{-1}\right)\phi\left(y\gamma^{-1}\right)\phi(y) p(\gamma).
	\end{align*}
	Averaging these two expressions yields
	\begin{equation*}
		\sca{M(u\phi)}{u\phi} -\sca{u^2}{\phi M\phi} = - \frac 12 \sum_{y \in Y}\sum_{\gamma \in \Gamma} \left[u\left(y\gamma^{-1}\right)-u(y)\right]^2\phi\left(y\gamma^{-1}\right)\phi(y) p(\gamma).
	\end{equation*}
	Hence
	\begin{equation*}
		\sca{M(u\phi)}{u\phi}  \leq \sca{u^2}{\phi M\phi}.
	\end{equation*}
If $u$ and $\phi$ are any non-negative maps, we approximate them by functions supported on larger and larger finite subsets of $Y$.
	The conclusion then follows from the monotone convergence theorem.
\end{proof}

\begin{prop}[Barta's inequality]
	\label{res: easy dir - applying barta}
	Let $\lambda \in [0,1]$.
	If there exists a positive function $\phi \colon Y \to \R^*_+$ such that $M\phi \leq \lambda \phi$, then $\tau(M) \leq \lambda$, where $\tau(M)$ is the spectral radius of $M$ seen as an operator on $\mathcal H$.
\end{prop}

\paragraph{Remark.}
We think of $\phi$ as a kind of $\lambda$ super-harmonic function for $M$.
The strength of this statement is that it provides an estimate of $\tau(M)$ without assuming that $\phi$ is square summable.

\begin{proof}
	Recall that $\mathcal H^+$ stands for the functions in $\mathcal H$ taking values in $\R_+$
	Since $p$ is symmetric, $M$ is a self-adjoint positive operator of $\mathcal H$.
	Hence its spectral radius can be computed as follows
	\begin{equation*}
		\tau(M) = \sup_{\psi \in \mathcal H^+ \setminus \{0\}} \frac {\sca {M\psi}{\psi}}{\norm \psi^2}.
	\end{equation*}
	Let $\psi \in \mathcal H^+$.
	Since $\phi$ is positive we can always write $\psi = u\phi$ where $u \colon Y \to \R_+$ is a non-negative function.
	It follows from \autoref{res: easy dir - applying barta} that
	\begin{equation*}
		\sca{M\psi}{\psi}
		\leq \sca{M(u\phi)}{u\phi}
		\leq \sca {u^2}{\phi M \phi}
		\leq \lambda \sca{u^2}{\phi^2}
		\leq \lambda \norm\psi^2
	\end{equation*}
	This inequality holds for every $\psi \in \mathcal H^+$, hence the result.
\end{proof}

We now exploit the previous proposition to estimate the spectral radius of $M$.
To that end we fix a base point $o \in X$ and a $\Gamma'$-invariant,  $h_{\Gamma'}$-quasi-conformal family of measures $(\nu'_x)$ on $\partial X$.
In addition we choose a  measurable  section $\partial X \to \partial_h X$, sending $\xi$ to $b_\xi$.
We define a function $\phi \colon \Gamma \to \R_+^*$ sending $\gamma$ to the total mass of $\nu'_{\gamma o}$, i.e.
\begin{equation*}
	\phi(\gamma) = \int \mathbf 1 d\nu'_{\gamma o}.
\end{equation*}
Since the family $(\nu'_x)$ is $\Gamma'$-invariant, $\phi$ induces a map $Y \to \R_+^*$ that we still denote $\phi$.
This function will play the role of the function $\phi$ in \autoref{res: easy dir - applying barta}.
To that end we need to compute $M\phi$.
Since $(\nu'_x)$ is $h_{\Gamma'}$-quasi-conformal, there exists a constant $C_1 \in \R_+$ such that for every point $y = \Gamma' \beta$ of $Y$, we have
\begin{equation}
	\label{eqn: easy dir - upper bound spec radius - def A}
	[M \phi](y) \leq C_1 \int B(\beta^{-1}\xi) d\nu'_{\beta o}(\xi)
\end{equation}
where $B \colon \partial X \to \R_+$ is defined by
\begin{equation}
	\label{eqn: easy dir - upper bound spec radius - aux A}
	B(\xi) = \sum_{\gamma \in \Gamma} e^{-h_{\Gamma'}b_\xi(\gamma o, o)}p(\gamma)
\end{equation}
Consequently, to estimate $M\phi$ and thus $\tau(M)$, it suffices to
  bound  $B(\xi)$ uniformly from above.

\medskip
Until now we  worked with an arbitrary symmetric probability measure~$p$.
In order to estimate the map $B \colon \partial X \to \R_+$ defined above we now specialize to a specific measure.
Basically we are going to consider measures supported by \og spheres \fg\ of large radius.
Before doing so we make a small digression in order to study the growth of spheres.
Let $r, a \in \R_+$ and $x \in X$.
We denote by
\begin{equation*}
	S_\Gamma(x,r, a) = \set{\gamma \in \Gamma}{ r -a < \dist{\gamma o}x \leq r}
\end{equation*}
the \og sphere \fg\ of radius $r$ (and thickness $a$) centred at $x$.
Similarly we define the \og ball \fg\ of radius $r$ centred at $x$ by
\begin{equation*}
	B_\Gamma(x,r) = \set{\gamma \in \Gamma}{\dist{\gamma o}x \leq r}.
\end{equation*}
Since the action of $\Gamma$ on $X$ is proper, these sets are finite.
Since the usual Patterson-Sullivan measure associated to the ambient group $\Gamma$ gives full measure to the radial limit set (\autoref{res: PS charges radial limit set}), there exists $C_2 \in \R_+$, such that for every $r \in \R_+$
\begin{equation}
\label{res: easy dir - growth ball}
	\card{B_\Gamma(o,r)} \leq C_2e^{rh_\Gamma},
\end{equation}
see for instance \cite[Corollaire~6.8]{Coo93}.
The next statement precise these estimates in the presence of a growth gap at infinity. 

\begin{lemm}[Yang {\cite[Theorem~5.3]{Yang:2016wa}}] 
\label{res: easy dir - growth sphere} 
	Let $\Gamma$ be a discrete group acting properly by isometries on a Gromov-hyperbolic space $X$. Assume that the action is strongly positively recurrent, i.e. there exists a growth gap at infinity.
	There exists $a, C_3 \in \R_+$, such that for every $r \in \R_+$, we have	
	\begin{equation*}
		\frac 1{C_3} e^{r h_\Gamma} \leq \card{S_\Gamma(o,r, a)} \leq C_3 e^{rh_\Gamma}.
	\end{equation*}
\end{lemm}

The previous lemma provides an estimate for the cardinality of any ball centred at a point in the $\Gamma$-orbit of $o$.
The goal of the next proposition is to provide a similar estimate for balls centred at any point $x \in X$.

\begin{prop}
\label{res: easy dir - growth ball gal} 
	Let $\Gamma$ be a discrete group acting properly  by isometries on a Gromov-hyperbolic space. Assume that the action is strongly positively recurrent.
	For all $\varepsilon \in \R_+^*$, there exists $C_4(\varepsilon) \in \R_+^*$, such that for all $x\in X$, we have
	\begin{equation*}
		\card{B_\Gamma(x,r)}
		\leq C_4(\varepsilon) e^{(2h_\Gamma^\infty+\varepsilon-h_\Gamma)\dist x{\Gamma o}}e^{rh_\Gamma}
	\end{equation*}
\end{prop}

\paragraph{Remark.}
This estimate is reminiscent from \cite[Theorem~3.2]{Schapira:2004bq}. 
Following the same proof, it is likely that in geometric situations where the growth of $\Gamma_K$ is purely exponential, this estimate should admit a similar lower bound.

\begin{proof}
	In the course of this proof, many parameters will appear.
	Those parameters only depend on $\varepsilon$ (and not on $x$).
	We denote them all by $C$, or $C(\varepsilon)$ if we want to emphasize the dependence in  $\varepsilon$.
	Without loss of generality,  we  choose $0<\varepsilon <h_\Gamma-h_\Gamma^\infty$.
	There exists a compact subset $k \subset X$ such that $h_{\Gamma_k} < h_\Gamma^\infty + \varepsilon/4$.
	Up to enlarging $k$, we  assume that $o$ belongs to $k$.
	Let $K$ be the $\delta$-neighbourhood of $k$ and   $D$ its diameter.
	Let $S \subset \Gamma$ and $r_0 \in \R_+$ be given by \autoref{res: thickening lemma} applied to $k$ and $K$.
	By definition of exponential growth rate, there exists $C(\varepsilon) \in \R_+$ such that for every $r \in \R_+$, we have
	\begin{equation*}
		\card {S\Gamma_k \cap B_\Gamma(o,r)} \leq C(\varepsilon) e^{r (h_{\Gamma_k} + \varepsilon/4)}.
	\end{equation*}
	Let $x \in X$.
	For simplicity, set $d  = \dist x{\Gamma K}$.
	We fix $\alpha \in \Gamma$ and $q \in \alpha K$ such that  $q$ is a projection of $x$ on $\Gamma K$.
	Given any geodesic $\geo qx$ from $q$ to $x$, the intersection $\alpha^{-1} \geo qx \cap \Gamma K$ is contained in $K$.
	By \autoref{res: thickening lemma}, for every $\gamma \in \Gamma$, there exists $\beta \in S \Gamma_k$ such that $\gro {\alpha^{-1}x}{\alpha^{-1}\gamma o}{\beta o} \leq r_0$.
	In particular,
	\begin{equation*}
		\dist{\beta o}{\alpha^{-1}\gamma o}
		 \leq \dist x{\gamma o} - \dist {\alpha^{-1}x}{\beta o} + 2r_0.
	\end{equation*}
	Consequently
	\begin{equation*}
		\alpha^{-1} B_\Gamma(x,r) \subset \bigcup_{\beta \in S \Gamma_k} B_\Gamma\left(\beta o, r - \dist{\alpha^{-1}x}{\beta o} + 2r_0\right)
	\end{equation*}
	Combined with (\ref{res: easy dir - growth ball}) it yields
	\begin{equation}
		\label{eqn: easy dir spr - growth sphere gal}
		\card{B_\Gamma(x,r)} \leq Ce^{rh_\Gamma}\sum_{\beta \in S\Gamma_k} e^{-h_\Gamma \dist{\alpha^{-1}x}{\beta o}}.
	\end{equation}
	Let us now estimate the latter sum.
	Recall that $D$ is the diameter of $K$, which contains both $o$ and $\alpha^{-1}q$.
	Hence for every $\beta \in S\Gamma_k$, we have $\dist{\alpha^{-1}x}{\beta o} \geq d - D$ 
	and
	\begin{equation*}
		\dist o{\beta o}
		\leq \dist o{\alpha^{-1} x} + \dist{\alpha^{-1}x}{\beta o}
		\leq d + \dist{\alpha^{-1}x}{\beta o} + D.
	\end{equation*}
	Consequently (\ref{eqn: easy dir spr - growth sphere gal}) becomes
	\begin{align*}
		\card{B_\Gamma(x,r)}
		 & \leq e^{rh_\Gamma}\sum_{\substack{ \ell \in \N               \\ \ell \geq d - D}} \sum_{\substack{\beta \in S\Gamma_k\\ \ell \leq \dist{\alpha^{-1}x}{\beta o} \leq \ell +1}} e^{-\ell h_\Gamma } \\
		 & \leq e^{rh_\Gamma}\sum_{\substack{ \ell \in \N               \\ \ell \geq d - D}} \card{S\Gamma_k \cap B_\Gamma(o,\ell + 1 + d)}e^{-\ell h_\Gamma } \\
		 & \leq C(\varepsilon)e^{rh_\Gamma}\sum_{\substack{ \ell \in \N \\ \ell \geq d - D}} e^{(h_{\Gamma_k} + \varepsilon/4)(\ell+d)}e^{-\ell h_\Gamma }.
	\end{align*}
	Recall that $h_{\Gamma_k} + \varepsilon/4 < h_\Gamma$. 
Up to increasing $C(\varepsilon)$, we get
	\begin{equation*}
		\card{B_\Gamma(x,r)}  \leq C(\varepsilon) e^{(2h_{\Gamma_k} + \varepsilon/2 - h_\Gamma)d}e^{rh_\Gamma}.
	\end{equation*}
	As $o$ belongs to $K$, we have $d  \leq \dist x{\Gamma o}$.
	Moreover $h_{\Gamma_k} \leq h_\Gamma^\infty + \varepsilon/4$, whence the result.
\end{proof}

\medskip
We now come back to the study of random walks in $Y$.
Let $a$ and $C_3$ be the parameters given by \autoref{res: easy dir - growth sphere}.
Without loss of generality we can assume that $a > 1$.
For every $n \in \N$, we denote by $p_n$ the uniform probability measure on $S_\Gamma(o,n, a)$,  $M_n$  the associated Markov operator (\ref{eqn: easy dir - markov op}) and $B_n \colon \partial X \to \R_+$  the auxiliary map associated to $p_n$ in (\ref{eqn: easy dir - upper bound spec radius - aux A}) .
By \autoref{res: easy dir - growth sphere},  we have $p_n(\gamma) \leq C_3e^{a h_\Gamma} e^{-nh_\Gamma}$ if $\gamma \in S_\Gamma(o,n, a)$, and $p_n(\gamma) = 0$ otherwise.

\begin{prop}
	\label{res: easy dir - upper bound A}
	For every $\varepsilon>0$, there exists $C_5(\varepsilon) \in \R_+$, such that for every $n \in \N$, for every $\xi \in \partial X$, we have
	\begin{equation*}
		B_n(\xi) \leq
		C_5(\varepsilon)\max \left\{ e^{-n h_{\Gamma'}}, e^{n(h_\Gamma^\infty + \varepsilon - h_\Gamma)}, e^{n(h_{\Gamma'} - h_\Gamma)}\right\}.
	\end{equation*}
\end{prop}

\begin{proof}
	As above, the proof involves many parameters which only depend  on~$\varepsilon$ (and not on $n$ or $\xi$).
	We still denote them all by $C$, or $C(\varepsilon)$.
	Choose $\varepsilon>0$ such that $h_\Gamma^\infty + \varepsilon < h_\Gamma$ and define
	\begin{equation*}
		h_{\text{aux}} = \max \{ \varepsilon, 2h_\Gamma^\infty + \varepsilon - h_\Gamma \}.
	\end{equation*}
	Up to decreasing $\varepsilon$, we can assume that $h_{\Gamma'} \neq ( h_\Gamma \pm h_{\text{aux}})/2$.
	Note that $0 < h_{\text{aux}} \leq h_\Gamma^\infty + \varepsilon$.
	Let $n \in \N$ and $\xi \in \partial X$.
	We fix a geodesic $[o, \xi)$ joining $o$ to $\xi$.
	For every $\ell \in \N$ we denote by $x_\ell$ the point on $[o, \xi)$ at distance $\ell$ from $o$.
	We now split the sum defining $B_n(\xi)$ according to the value of the Gromov product $\gro{\gamma o}\xi o$.
	\begin{equation*}
		B_n(\xi) = \sum_{\ell \in \N} \sum_{ \substack{\gamma \in \Gamma \\ \ell \leq \gro{\gamma o}\xi o \leq \ell + 1}} e^{-h_{\Gamma'} b_\xi(\gamma o,o)} p_n(\gamma).
	\end{equation*}
	Note first that the first sum is actually a finite sum.
	Indeed for every $\gamma \in S_\Gamma(o, n, a)$ the Gromov product $\gro{\gamma o}\xi o$ is at most $n$.
	Let $\ell \in \N$ and $\gamma \in S_\Gamma(o,n, a)$ such that
	\begin{equation*}
		\ell \leq \gro{\gamma o}\xi o \leq \ell +1.
	\end{equation*}
	A standard exercise of hyperbolic geometry shows that $\gamma$ belongs to $B_\Gamma(x_\ell, n - \ell + \delta)$ and  $b_\xi(\gamma o, o) \geq n - 2\ell  - (a+\delta)$.
	On the other hand, as we noticed before
	\begin{equation*}
		p_n(\gamma) \leq Ce^{-nh_\Gamma}
	\end{equation*}
	Consequently
	\begin{equation*}
		B_n(\xi) \leq C e^{-n \left(h_\Gamma + h_{\Gamma'}\right)}
		\sum_{\ell \leq n} \sum_{\gamma \in B_\Gamma(x_\ell,n - \ell + \delta)} e^{2\ell h_{\Gamma'}}.
	\end{equation*}
	Note that if $B_\Gamma(x_\ell,n - \ell + \delta)$ is non-empty, then
	\begin{equation*}
		\dist{x_\ell}{\Gamma o} \leq \min \{ \ell, n - \ell \} + \delta.
	\end{equation*}
	Using \autoref{res: easy dir - growth ball gal} we get
	\begin{equation*}
		B_n(\xi) \leq C(\varepsilon)e^{-n h_{\Gamma'}}
		\sum_{\ell \leq n} e^{\left(2 h_{\Gamma'}-h_\Gamma\right)\ell} e^{h_{\text{aux}} \min \left\{ \ell, n - \ell\right\}}.
	\end{equation*}
	We now split the  sum according to the value of
	$\min \left\{ \ell, n - \ell\right\}$. We get
	\begin{equation}
		\label{eqn: easy dir - upper bound A}
		B_n(\xi)
		\leq C(\varepsilon)e^{-n h_{\Gamma'}}
		\left[
			\sum_{\ell \leq n/2} e^{\left(2 h_{\Gamma'}- h_\Gamma + h_{\text{aux}}\right)\ell}
			+ e^{nh_{\text{aux}}}\sum_{n/2 <\ell \leq n} e^{ \left(2h_{\Gamma'}-h_\Gamma - h_{\text{aux}}\right)\ell}
			\right].
	\end{equation}
	We now distinguish several cases depending on the value of $h_{\Gamma'}$ compared to $(h_\Gamma \pm h_{\text{aux}})/2$.
	Recall that we chose $\varepsilon$ in such a way that  $h_{\Gamma'} \neq (h_\Gamma \pm h_{\text{aux}})/2$.
	\paragraph{Case 1.}
	\emph{Assume that $h_{\Gamma'} <(h_\Gamma-h_{\text{aux}})/2$}. 
	Then both  terms within the bracket  in (\ref{eqn: easy dir - upper bound A}) are bounded. 
	We get
	\begin{equation*}
		B_n(\xi)
		\leq C(\varepsilon)e^{-n h_{\Gamma'}}
	\end{equation*}
	\paragraph{Case 2.} \emph{Assume that $(h_\Gamma - h_{\text{aux}})/2 < h_{\Gamma'} < (h_\Gamma + h_{\text{aux}})/2$}.
	In this case the two terms within the brackets in  (\ref{eqn: easy dir - upper bound A}) have exactly the same asymptotic behaviour.
	More precisely, the computation yields
	\begin{equation*}
		B_n(\xi)
		\leq C(\varepsilon)e^{\left(h_{\text{aux}}-h_\Gamma\right)n/2}
		\leq C(\varepsilon)e^{\left(h_\Gamma^\infty + \varepsilon-h_\Gamma\right)n}
	\end{equation*}
	\paragraph{Case 3.}
	\emph{Assume that $h_{\Gamma'} >(h_\Gamma+h_{\text{aux}})/2$}. 
	Both sums in (\ref{eqn: easy dir - upper bound A}) diverge exponentially, however the second term dominates the first one.
	Hence
	\begin{equation*}
		B_n(\xi)
		\leq C(\varepsilon) e^{ \left(h_{\Gamma'}-h_\Gamma\right)n}.
	\end{equation*}
	The result is the combination of these three cases.
\end{proof}

\begin{coro}
	\label{res: easy dir - asymp spec rad}
	The asymptotic behaviour of the spectral radius $\tau(M_n)$ of $M_n$ is asymptotically controlled as follows
	\begin{equation*}
		\limsup_{n \to \infty} \frac 1n \ln \tau(M_n)
		\leq \max \left\{ - h_{\Gamma'},  h_\Gamma^\infty - h_\Gamma, h_{\Gamma'} - h_\Gamma \right\}.
	\end{equation*}
\end{coro}

\begin{proof}
	Let $\varepsilon > 0$. 
	Recall that $\phi \colon Y \to \R_+^*$ is the map sending $y = \Gamma' \beta$ to the total mass of the measure $\nu'_{\beta o}$.
	Let $n \in \N$.
	Injecting in (\ref{eqn: easy dir - upper bound spec radius - def A}) the estimate given by \autoref{res: easy dir - upper bound A}, we get
	\begin{equation*}
		M_n \phi \leq C(\varepsilon)  \lambda_n \phi,
		\quad \text{where} \quad
		\lambda_n  \max \left\{ e^{-n h_{\Gamma'}}, e^{n(h_\Gamma^\infty + \varepsilon - h_\Gamma) }, e^{n(h_{\Gamma'} - h_\Gamma)}\right\}.
	\end{equation*}
	By Barta's inequality (\autoref{res: easy dir - applying barta}), we deduce $\tau(M_n) \leq C(\varepsilon)\lambda_n$.
	Observe that $C(\varepsilon)$ does not depend on $n$.
	Passing to the limit we obtain
	\begin{equation*}
		\limsup_{n \to \infty} \frac 1n \ln \tau(M_n) \leq \max\left\{ - h_{\Gamma'}, h_\Gamma^\infty + \varepsilon - h_\Gamma, h_{\Gamma'} - h_\Gamma\right\}.
	\end{equation*}
	This inequality holds for every $\varepsilon \in \R_+^*$,  whence the result.
\end{proof}

The next corollary completes the proof of the ``easy direction'' in \autoref{theo:mainbis}.

\begin{coro}
	\label{res: easy dir - final statement}
	If $\Gamma'$ is co-amenable in $\Gamma$, then $h_{\Gamma'} = h_\Gamma$.
\end{coro}

\begin{proof}
	It follows from Kesten's amenability criterion that the spectral radius of any random walk on $Y = \Gamma' \backslash \Gamma$ is $1$ \cite{Kes59a,Day:1964tl}, see also \cite{Coulon:2017vz} for the case where $\Gamma'$ is not a normal subgroup of $\Gamma$.
	Consequently \autoref{res: easy dir - asymp spec rad}  yields
	\begin{equation*}
		\max \left\{ - h_{\Gamma'},  h_\Gamma^\infty - h_\Gamma , h_{\Gamma'} - h_\Gamma \right\} \geq 0.
	\end{equation*}
	Since $h_\Gamma^\infty < h_\Gamma$, the only options are $h_{\Gamma'} = 0$ or $h_{\Gamma'} = h_\Gamma$.
	It remains to rule out the first case.
	Assume that $h_{\Gamma'} = 0$.
	We claim that $\Gamma'$ is amenable.
	Since $\Gamma'$ is countable, it can be written as an increasing union of finitely generated subgroups.
	Hence, it suffices to prove that every finitely generated subgroup of $\Gamma'$ is amenable.\ 
	Let $S$ be a finite subset of $\Gamma'$ and $\Gamma'_S$ the subgroup of $\Gamma$ generated by $S$.
	Obviously $h_{\Gamma'_S} = 0$.
	However the word metric on $\Gamma'_S$ (with respect to $S$) dominates the metric induced by the action on $X$.
	It follows that $\Gamma'_S$ has sub-exponential growth with respect to the word metric, hence $\Gamma'_S$ is amenable, which completes the proof of our claim.
	By assumption the action of $\Gamma$ on $Y$ is amenable.
	Moreover the stabiliser of any point $y \in Y$ is conjugated to $\Gamma'$, hence amenable.
	It follows  that $\Gamma$ is amenable \cite[Lemma~3.2]{Juschenko:2013de} or \cite[Lemma~4.5]{Glasner:2007eo}, which contradicts the fact that $\Gamma$ is non-elementary.
\end{proof}

%
\subsection{Rigidity and growth gap}
%

We now exploit rigidity properties to exhibit the existence of growth gaps for subgroups of $\Gamma$.
We first recall the definition of the famous Kazdhan property (T).
For more details we refer to \cite{BekHarVal08}.

\begin{defi}[Kazhdan property]
	\label{def: property T}
	A discrete group $\Gamma$ has \emph{Kazhdan property} (T), if any unitary representation of $\Gamma$ with almost invariant vectors admits a non-zero invariant vector.
\end{defi}

For our purpose this property is too strong.
Indeed we only consider unitary representations induced by an action on a countable set.
In this context the appropriate rigidity property is Property (FM) studied by Monod and Glasner \cite{Glasner:2007eo} or de Cornulier \cite{deCornulier:2015ie}.
Similar properties have also been considered by Bekka and Olivier \cite{Bekka:2014io}.

\begin{defi}
	\label{def: property FM}
	A discrete $\Gamma$ has \emph{Property} (FM) if every amenable action of $\Gamma$ on a discrete countable set has a finite orbit.
\end{defi}

Let $Y$ be a countable discrete set endowed with an action of $\Gamma$.
The induced representation $\rho \colon \Gamma \to \mathcal U(\ell^2(Y))$ has a non-zero invariant vector if and only if $\Gamma$ has a finite orbit.
In view of this remark, Property (FM) can be reformulated as follows.

\begin{prop}
	\label{res: property FM rep}
	A discrete group $\Gamma$ has property (FM) if and only if for every action of $\Gamma$ on a discrete countable set $Y$, if the induced representation $\rho \colon \Gamma \to \mathcal U(\ell^2(Y))$ almost admits invariant vectors, then it has a non-zero invariant vector.
\end{prop}

Obviously, Property (T) implies Property (FM).
However the converse is not true.
For instance the free product of two infinite simple groups with Property (T) has property (FM) \cite[Lemma~3.2]{Glasner:2007eo} but cannot have property (T) as it acts on the corresponding Bass-Serre tree without global fixed point.
The next statement is an analogue of the existence of Kazhdan pairs, which quantifies Property (FM).
The proof works verbatim as in \cite[Proposition~1.2.1]{BekHarVal08} and is left to the reader.

\begin{lemm}
\label{res: quantified FM}
	A discrete group $\Gamma$ has Property (FM) if and only if there exists a finite subset $S$ of $\Gamma$ and $\varepsilon \in \R_+^*$ with the following property: 
	for every action of $\Gamma$ on a discrete countable set $Y$, if the induced representation $\rho \colon \Gamma \to \mathcal U(\ell^2(Y))$ has an $(S, \varepsilon)$-invariant vector, then it has a non-zero invariant vector.

\end{lemm}
\begin{theo}\label{theo:Main-precis}
	Let $X$ be a hyperbolic proper geodesic space.
	Let $\Gamma$ be a group with Property (FM) acting  properly  by isometries on $X$.
	We assume that the action of $\Gamma$ is strongly positively recurrent.
	There exists $\eta > 0$ such that for every subgroup $\Gamma'$ of $\Gamma$, if $h_{\Gamma'} \geq (1-\eta)h_\Gamma$, then $\Gamma'$ is a finite index subgroup of~$\Gamma$.
\end{theo}

\begin{proof}
	Since $\Gamma$ has Property (FM), there exists a finite subset $S$ of $\Gamma$ and $\varepsilon \in \R_+^*$ such that for every action of $\Gamma$ on a discrete countable set $Y$, if the induced representation $\Gamma \to \mathcal U(\ell^2(Y))$ has an $(S,\varepsilon)$-invariant vector, then it admits a non-zero invariant vector (\autoref{res: quantified FM}).
	According to \autoref{res: growth vs almost inv vec - quantified} there exists $\eta \in \R_+^*$ with the following property: assume that  $\rho \colon \Gamma \to \mathcal U(\mathcal H)$ is a unitary representation in a Hilbert lattice;
	if $h_\rho \geq (1-\eta) h_\Gamma$ then $\mathcal H$ admits $(S, \varepsilon)$-invariant vectors.
	Let $\Gamma'$ be a subgroup of $\Gamma$ such that $h_{\Gamma'} \geq (1- \eta) h_\Gamma$.
	We write $Y = \Gamma' \backslash \Gamma$ for the space of  left   cosets.
	Let $\mathcal H = \ell^2(Y)$ the Hilbert lattice of square summable functions and $\rho \colon \Gamma \to \mathcal U(\mathcal H)$ the corresponding Koopman representation.
	It follows from \autoref{res: lower bound delta rho} that $h_\rho \geq (1- \eta) h_\Gamma$.
	According to our choice of $\eta$, the representation $\rho$ admits an $(S, \varepsilon)$-invariant vector, hence a non-zero invariant vector.
	This exactly means that the action of $\Gamma$ on $Y$ has a finite orbit.
	However this action being transitive, $Y$ is finite.
	In other words $\Gamma'$ has finite index in $\Gamma$.
\end{proof}

%
\subsection{Counterexamples}
%
\label{sec:counterexample}

\paragraph{Counterexample without negative curvature.}
 
If the space $X$ is not hyperbolic, the ``easy direction'' of our main theorem fails.
Indeed there exists finitely generated amenable groups $\Gamma$ whose action on their Cayley graph $X$ has exponential growth, for instance Baumslag Solitar groups $\textrm{BS}(1,n)$, lamplighter groups, etc.
More generally, any solvable group which is not virtually nilpotent is so.
For such a group $\Gamma$ the trivial subgroup $\Gamma' = \{1\}$ obviously satisfies $h_{\Gamma'} < h_\Gamma$ although the quotient $\Gamma/\Gamma'$ is amenable. Note that the action of a group on its Cayley graph is cocompact, hence strongly positively recurrent.

This problem cannot be ``fixed'' by strengthening the assumption on the quotient $\Gamma / \Gamma'$, e.g. by asking that $\Gamma/\Gamma'$ has polynomial growth.
Consider indeed the lamplighter group $L$ defined by 
\begin{equation*}
	L = V \rtimes \Z,
	\quad \text{where} \quad
	V = \bigoplus_{n \in \Z} \Z_2.
\end{equation*}
An element $v = (v_n)$ of $V$ is a sequence of elements of the finite groups $\Z_2$ which are trivial for all but finitely many $n \in \Z$.
In particular we write $a = (a_n)$ for the sequence which is trivial everywhere except at $n = 0$.
The generator $t$ of $\Z$ acts on $V$ by the usual shift.
The set $\{a,t\}$ generates $L$.
Let $X$ be the Cayley graph of $\Gamma$ with respect to this set (on which $L$ acts properly cocompactly).
Parry \cite{Parry:1992ia} computed the associated growth series of $L$.
One can extract from his result that 
\begin{equation*}
	h_L(X) = \frac{1 + \sqrt 5}2 \approx 1.618,
\end{equation*}
see for instance \cite{Bucher:2017bs}.
Actually Parry provides an explicit formula for the length of an element in $L$ with respect to $\{a,t\}$ \cite[Theorem~1.2]{Parry:1992ia}.
In particular the length $\abs v$ of an element $v = (v_n)$ in $V$ is the sum of two contributions: 
\begin{enumerate}
	\item the \emph{length} of the shortest loop in $\Z$, based at the identity, that visits all indices $n$ for which $v_n \neq 1$.
	\item the \emph{number} of indices $n \in \Z$ such that $v_n \neq 1$.
\end{enumerate}
This can be used to compute the growth series $\zeta_V(z)$ of $V$ for its action on $X$.
All computations done we get
\begin{equation*}
	\zeta_V(z) = \sum_{v \in V} z^{\abs v}= 1 + z + \frac{z^2(1+z)(1-z)\left(2+3z+2z^2 \right)}{\left[1-z^2(z+1)\right]^2}.
\end{equation*}
Hence $h_V(X)$ is the root of $X^3-X-1=0$ which approximatively equals $1.3247$.
In particular $h_V(X) < h_L(X)$ while the quotient $L/V$ is isomorphic to $\Z$.

\paragraph{Counterexample without a growth gap at infinity.}
 We now provide a few counterexamples acting on Gromov hyperbolic spaces where the ``hard direction'' of our main theorem fails when we drop the strongly positively recurrent assumption.
 
  Parabolic discrete groups acting of $\mathbb H^n$ act by isometries on horospheres, which are Euclidean for their induced metric. Therefore, by Bieberbach theorem they are virtually abelian, hence amenable. Still have non-zero critical exponent: our main theorem cannot apply to such groups. One can elementarily show, using convexity of Busemann functions, that such parabolic groups do not have a growth gap at infinity. Let us now construct non-elementary examples.

For fundamental groups of negatively curved surfaces, having a strongly positively recurrent action is an optimal assumption to get \autoref{theo:mainbis}, as shown in the next proposition.

\begin{prop} 
	Let $S$ be a locally ${\rm CAT}(-1)$ surface, $\Gamma$ its fundamental group and $X$ its universal cover.
	It the action of $\Gamma$ on $X$ does not have a growth gap at infinity, then it admits normal subgroups $\Gamma'\triangleleft\Gamma$ with $h_\Gamma=h_{\Gamma'}$ and such that $\Gamma/\Gamma'$ contains a free group.
\end{prop}

\begin{proof} 
	Choose two disjoint closed non-separating geodesics $c_1$ and $c_2$ on $S$. 
	Such disjoint closed curve exist up to taking a finite covering of $S$. Cut $S$ along these curves; 
using the surface with boundary thus obtained, it is elementary to build a surface $S'$ which is a regular cover of $S$ with a covering group isomorphic to $\mathbb{F}_2$.
	If $K$ is a compact set containing $c_1$ and $c_2$ in $S$, this surface $S'$ contains many copies of $S\setminus K$ so that $\Gamma'=\pi_1(S')$ satisfies $h_{\Gamma'}\geq h_\Gamma^\infty=h_\Gamma$.  
	The proposition follows.
\end{proof}

This proposition is really due to the fact that $\Gamma$ is a surface group. It follows from \cite{Dalbo:2000eh} that there exists such surfaces with finitely generated fundamental group and pinched negative curvature. Negatively curved finite volume surfaces without growth gap at infinity were constructed in \cite{Dalbo:2017pp}.
Note that some of these examples even have a finite Bowen-Margulis measure. 
Constant curvature surfaces with finitely generated fundamental group always have a growth gap at infinity. 
A $\mathbb Z$-cover of a compact hyperbolic surface is typically a constant curvature surface which does not have a critical gap, and hence satisfies the above proposition.

Let us give a three dimensional constant curvature example.

	\begin{prop}
		Let $M = \mathbb H^3/\Gamma_1$ where $\Gamma_1$
 be a \emph{simply degenerated} representation of a surface group in $\mathbb H^3$. Then there exists a hyperbolic isometry  $h\in \text{Isom}^+(\mathbb H^3)$ satisfying the following. 
Let $\Gamma = \langle \Gamma_1, h\rangle$. Then $\Gamma_1$ is not co-amenable in $\Gamma$, and $h_\Gamma = h_{\Gamma_1} = 2$.
	\end{prop}

	\begin{proof}[Sketch of proof.]
		A simply degenerated representation of a surface group $\Gamma_1$ is the geometric limit of a sequence of \emph{quasi-fuchsian representations} $\rho_n(\Gamma_0)$ of a fixed surface group $\Gamma_0$ such that one end of $\mathbb H^3/\Gamma_1$ remains convex-cocompact, whereas the other end becomes geometrically infinite. We refer to \cite[Chapters~4 and 5]{Marden:2007kj}  for a precise definition of this terminology.

		It follows from \cite{Bishop:1997he} that $h_{\Gamma_1} = 2$. Now, since $\Gamma_1$ is simply degenerated, its \emph{discontinuity set} $\partial  \mathbb H^3 \backslash \Lambda(\Gamma_1)$ is non-empty. It is therefore possible to find a hyperbolic isometry $h\in \text{Isom}^+(\mathbb H^3)$ 
whose axis has end points in a ball contained in this discontinuity set. 
The groups $\Gamma_1$ and $\langle h\rangle$ are said to be in \emph{Schottky position}: an easy application of Klein's ping pong lemma shows then that
		\begin{equation*}
			\Gamma = \langle \Gamma_1, h\rangle = \Gamma_1 \ast \langle  h\rangle.
		\end{equation*}
		In particular $\Gamma_1$ is not co-amenable in $\Gamma$. Moreover, $2 =h_{\Gamma_1} \leq h_\Gamma \leq 2$ since  any  kleinian group in dimension $3$ has critical exponent at most 2.
	\end{proof}

 We complete this section with a last example coming from geometric group theory. 

\begin{prop}
	Let $\Gamma$ be a group and $\mathcal P$ a finite collection of residually finite subgroups of $\Gamma$ such that $\Gamma$ is hyperbolic relative to $\mathcal P$.
	Let $X$ be a metric space endowed with proper cusp-uniform action of $(\Gamma, \mathcal P)$.
	If $\mathcal P$ contains a subgroup $P$ such that $h_P = h_\Gamma$, then there exists a normal subgroup $\Gamma'$ of $\Gamma$ such that
	\begin{enumerate}
		\item $h_{\Gamma'} = h_\Gamma$;
		\item $\Gamma/ \Gamma'$ is non-elementary hyperbolic, hence non-amenable.
	\end{enumerate}
\end{prop}

\begin{proof}
	Using the group theoretic Dehn filling \cite{Groves:2008ip,Osin:2007ia}, there exists a finite index subgroup $P_0$ of $P$ such that the quotient of $\Gamma$ by $\Gamma' = \normal {P_0}$ is non-elementary hyperbolic.
	Since $P_0$ is a finite index subgroup of $P$, it has the same growth rate as $P$, i.e. $h_\Gamma$.
	As $\Gamma'$ contains $P_0$, its growth rate is $h_\Gamma$.
\end{proof}

%
\section{Comments and questions} \label{sec:comments}
%

Let us present some natural opening directions of this work.

%
\subsection{Generalizations of \autoref{res: main} and its variations}
%

\paragraph{Beyond hyperbolicity.}
The approach presented in this paper is most likely applicable to various context beyond groups acting on a $\delta$-hyperbolic space. 
Let $\Gamma$ be a discrete group acting by isometries on a  general proper geodesic metric space $(X,d)$.
As already noticed by Arzhantseva et al. \cite{Arzhantseva:2015cl} and Yang \cite{Yang:2016wa}, the existence of a growth gap at infinity provides many interesting results as soon as this action admits \emph{contracting elements} -- see for instance \cite{Yang:2016wa} for a definition. 
This settings includes for instance CAT$(0)$ groups with rank one elements or all convex-cocompact subgroups of the mapping class groups acting on Teichm\"uller space (including the mapping class group itself).
We currently work on the extension of our strategy to this more general context.

\paragraph{ Locally compact groups.}

Instead of considering a \emph{discrete} group $\Gamma$ acting on a metric space, we could also work with locally compact groups.
Let $X$ be a Gromov hyperbolic space such that $G= \isom X$ is locally compact group containing a lattice. 
Define its critical exponent $h_G$ to be the infimum of $s>0$ such that
\begin{equation*}
	\mathcal P_G(s) = \int_G e^{-s\dist o{go}} dg<\infty,
\end{equation*}
where $dg$ is the Haar measure on $G$. Still replacing Poincar\'e series by Haar integrals, we can then define analogously the entropy at infinity of $G$, Patterson-Sullivan theory on the horoboundary of $X$, etc. 
It seems likely  that all the theory would extend in this larger setting.
In particular it should lead to the following wide generalization of Corlette's rigidity result \cite{Corlette:1990br}.
Assume that $\isom X$ has Kazhdan's Property (T) and its action on $X$ is strongly positively recurrent.
Then there exists $\varepsilon \in \R_+^*$ such that for every discrete group $\Gamma$ of isometries of $X$ either $\Gamma$ is a lattice or $h_\Gamma \leq \textrm{dim}_{\rm vis}(\partial X) - \varepsilon$, where $\textrm{dim}_{\rm vis}(\partial X)$ stands for the visual dimension of $\partial X$.

%
\subsection{Twisted Patterson-Sullivan measures}
%

Let $\Gamma$ be a discrete group acting on a $\delta$-hyperbolic space, and let $\rho$ be a positive unitary representation of $\Gamma$ on some Hilbert lattice. 
The twisted Patterson-Sullivan density $a^\rho = (a^\rho_x)_{x\in X}$ which we introduced in \autoref{sec: twisted ps} is a powerful tool whose exploration should be fruitful. Let us mention some natural problems raised by our study. 
\begin{enumerate}
	\item If $h_\rho = h_\Gamma$,  understand the relation between the operator $a_o^\rho(\mathbf 1)$ and  the  orthogonal  projection on the subspace of invariant vectors of the limit representation $\rho_\omega$.

	\item If $h_\rho < h_\Gamma$, what can be said about 
 the operator $a_o^\rho(\mathbf 1)$?
	 
	\item 
	Let $\Gamma'$ be a subgroup of $\Gamma$ and $\mathcal H = \ell^2(\Gamma/\Gamma')$.
	The Patterson-Sullivan density twisted by the induced representation $\rho \colon \Gamma \to \mathcal U(\mathcal H)$ can be seen as a \emph{$\Gamma/\Gamma'$-extension} of the classical Patterson-Sullivan density. Many recent works deal with group extensions of Markov shifts over a finite alphabet, in particular when studying covers of negatively curved convex-cocompact manifolds or Schottky manifolds (see for instance \cite{Conze:2013,Jaerisch:2016,Stadlbauer:2013dg,Dougall:2016cc}).
	      It seems plausible that, using twisted Patterson-Sullivan measure, many ergodic results which have been obtained for group extensions of Markov shifts could be carried to the geodesic flow.

\end{enumerate}

\appendix

%
\section{Integration of vector-valued functions}
%
\label{sec: bochner integral}

%
\subsection{Bochner spaces}
%
\label{sec: bochner space}

We start by recalling the notion of Bochner integral and Bochner spaces.
The goal is to give a rigorous definition for the integral of a Hilbert valued map.
For our purpose, everything works verbatim as for the usual Lebesgue integral.
We refer the reader to the original article of Bochner \cite{Bochner1933} or \cite{Dinculeanu:1967uia,Diestel:1977ui}.

\medskip
Let $(X, \mathcal B, \nu)$ be a finite measure space and $(E, \normV)$ a Banach space.

\paragraph{Measurable functions.}
A map $\Phi \colon X \to E$ is \emph{simple} if it can be written $\Phi = \mathbf 1_{B_1} \phi_1 + \dots + \mathbf 1_{B_n} \phi_n$ where $B_i \in \mathcal B$ and $\phi_i \in E$.
A function $\Phi \colon X \to E$ is \emph{$\nu$-measurable} if there exists a sequence $(\Phi_n)$ of simple functions from $X$ to $E$ which converges $\nu$-almost everywhere to $\Phi$.

\paragraph{Bochner spaces.}
Let $p \in [1, \infty)$.
Observe that if $\Phi \colon X \to E$ is a $\nu$-measure map, then the function $X \to \R_+$ mapping $x$ to $\norm{\Phi(x)}$ is measurable (in the usual sense).
Hence we can define the \emph{$p$-norm} of $\Phi$ by
\begin{equation*}
	\norm[p]\Phi = \left( \int \norm{\Phi(x)}^p d \nu(x) \right)^{1/p}
\end{equation*}
The \emph{Bochner space} $L^p(\nu, E)$ is the set of $\nu$-measurable maps $\Phi \colon X \to E$ such that $\norm[p] \Phi < \infty$, up to the standard equivalence relation which identifies two maps which coincide $\nu$-almost everywhere.
The norm $\normV[p]$ gives to $L^p(\nu, E)$ a structure of Banach space.
Similarly we define a uniform norm by
\begin{equation*}
	\norm[\infty]\Phi  = \supess_{x \in X} \norm{f(x)}
\end{equation*}
The \emph{Bochner space} $L^\infty(\nu, E)$ consists of all $\nu$-measurable maps $\Phi \colon X \to E$ which are essentially bounded.
Again this definition is meant up to equality $\nu$-almost everywhere.
It is a Banach space.

\medskip
Since $\nu$ has finite measure a standard argument shows that $L^q(\nu, E)$ embeds in $L^p(\nu, E)$ provided $1 \leq p \leq q \leq \infty$.
For every $p \in [1,\infty)$ the set of simple functions is dense in $L^p(\nu,E)$.

\medskip
If $E = \R$, these spaces coincide with the usual function spaces $L^p(\nu)$.
If $\mathcal H$ is a Hilbert space, the Bochner space $L^2(\nu, \mathcal H)$ has a structure of Hilbert space, where the scalar product is given by
\begin{equation*}
	\sca {\Phi_1}{\Phi_2} = \int \sca{\Phi_1(x)}{\Phi_2(x)} d\nu(x), \quad \forall \Phi_1,\Phi_2 \in L^2(\nu, \mathcal H).
\end{equation*}

\paragraph{Bochner integral.}
The definition of the Bochner integral follows exactly the same steps than the one of the Lebesgue integral.
More precisely one starts by defining the integral of a simple function.
Given a simple function $\Phi = \mathbf 1_{B_1} \phi_1 + \dots + \mathbf 1_{B_n} \phi_n$, its integral is the vector of $E$ defined by
\begin{equation*}
	\int \Phi d\nu = \sum_{i \in I} \nu(B_i)\phi_i .
\end{equation*}
A $\nu$-measurable function $\Phi \colon X \to E$ is \emph{Bochner integrable} if there exists a sequence $(\Phi_n)$ of simple functions from $X$ to $E$ such that
\begin{equation*}
	\lim_{n \to \infty} \int \norm{\Phi - \Phi_n} d\nu  = 0,
\end{equation*}
in which case we define the integral of $\Phi$ as
\begin{equation*}
	\int \Phi d\nu = \lim_{n \to \infty} \int \Phi_n d\nu.
\end{equation*}
One checks easily that this integral is well defined and does not depend on the choice of $(\Phi_n)$.
A function $\Phi$ is Bochner integrable if and only if it belongs to $L^1(\nu, E)$ \cite[Chapter~II, Theorem~2]{Diestel:1977ui}.
The Bochner integral defines a $1$-Lipschitz linear map $L^1(\nu, E) \to E$ satisfying the following useful properties.

\begin{prop}[{\cite[Chapter~II, Corollary~5]{Diestel:1977ui}}]
	\label{res: app - unicity Bochner}
	Let $E$ be a Banach space.
	Let $\Phi, \Phi' \in L^1(\nu, E)$.
	If
	\begin{equation*}
		\int \mathbf 1_B \Phi d\nu = \int \mathbf 1_B \Phi' d\nu,\quad \forall B \in \mathcal B,
	\end{equation*}
	then $\Phi = \Phi'$ $\nu$-almost everywhere.
\end{prop}

\begin{prop}[{\cite[Chapter~II, Theorem~6]{Diestel:1977ui}}]
	\label{res: app - bochner integral commutes with linear map}
	Let $E$ and $F$ be two Banach space.
	Let $T \colon E \to F$ be a continuous linear operator.
	For every $\Phi \in L^1(\nu, E)$, the function $T(\Phi)$ belongs to $L^1(\nu,F)$.
	Moreover
	\begin{equation*}
		T\left(\int \Phi d\nu\right) = \int T(\Phi) d\nu.
	\end{equation*}
\end{prop}

%
\subsection{The Radon-Nikodym property}
%

Let $(X,\mathcal B, \nu)$ be a measure space.
The standard Radon-Nikodym theorem states that $L^\infty(\nu)$ is the dual of $L^1(\nu)$.
In general if $E$ is an arbitrary Banach space and $E'$ its dual, the space $L^\infty(\nu, E')$ is not necessarily the dual of $L^1(\nu,E)$.
The Radon-Nikodym property defined below is precisely designed to prevent this kind of pathology.
See \cite[Chapter~III, Definition~3 and Theorem~5]{Diestel:1977ui}.

\begin{defi}
	\label{def: app - RN property}
	A Banach space $E$ has the \emph{Radon-Nikodym property} if for every finite measure space $(X, \mathcal B, \nu)$ the following holds:
	for every continuous linear map $T \colon L^1(\nu) \to E$  there exists a function $\Phi \in L^\infty(\nu, E)$ such that
	\begin{equation*}
		T(f) = \int f\Phi\ d\nu,\quad \forall f \in L^1(\nu).
	\end{equation*}
\end{defi}

In this definition the integral is a Bochner integral as defined previously.
Note that the function $\Phi$ given by the definition is necessarily unique (\autoref{res: app - unicity Bochner}).
Moreover one checks that $\norm[\infty] \Phi = \norm T$ \cite[Chapter~III, Lemma~4]{Diestel:1977ui}.

\medskip
Recall that a Banach space $(E, \normV)$ is  \emph{reflexive} if the evaluation map $E \to E''$ from $E$ to its bidual space $E''$ is an isomorphism.
For instance every Hilbert space is reflexive.
The following important result is due to Phillips \cite{Phillips:1943dw}.

\begin{theo}[{\cite[Chapter~III, Corollary~13]{Diestel:1977ui}}]
	\label{res: app - reflexive implies RN}
	Reflexive Banach spaces have the Radon-Nikodym property.
\end{theo}

%
\section{Banach lattices}
%
\label{sec: banach lattices}

In this section we review the basic properties of Banach spaces endowed with a lattice structure.
For an in-depth study of Banach lattices we refer to \cite{Schaefer:1974vz} or~\cite{Aliprantis:2006tpa}.

%
\subsection{Definitions and main properties}
%
\label{sec: banach lattices - main properties}

\paragraph{Vocabulary and notations.}
A \emph{vector lattice} $(E,\prec)$ (also called Riesz space) is a vector space $E$ equiped with a partial order $\prec$, compatible with the vector space structure, which provides $E$ with a \emph{lattice} structure, i.e. such that for all $\phi,\psi\in E$, the set $\{\phi,\psi\}$ has a \emph{least upper bound} usually denoted by  $\phi\lor \psi\in E$ and a \emph{greater lower bound}, usually denoted by $\phi\land \psi\in E$. 
Given $\phi \in E$, its \emph{absolute value}, is the vector
\begin{equation*}
	\abs \phi = \phi \lor (-\phi) = \phi_+ + \phi_-
\end{equation*}
where $\phi_+ = \phi \lor 0$ and $\phi_- = (-\phi)\lor 0$ are respectively the positive and negative part of $\phi$.
The \emph{positive cone} of $E$, denoted by $E^+$, is the set of vector $\phi \in E$ such that $0 \prec \phi$.
An \emph{ideal} of $E$ is a vector subspace of $F$ of $E$ satisfying the following property: for every $\phi \in E$ and $\psi \in F$, if $\abs \phi \prec \abs \psi$, then $\phi$ belongs to $F$.
The vector lattice $(E, \prec)$ is \emph{(countably) order complete} if every non-empty (countable) subset of $E$ which is bounded from above admits a least upper bound.
A norm $\normV$ on $E$ is a \emph{monotone} if we have $\norm{\phi_1} \leq \norm{\phi_2}$ whenever  $\phi_1, \phi_2 \in E$  satisfy $\abs{\phi_1}\prec \abs{\phi_2}$.
If $E$ is (topologically) complete for such a norm, it is called a \emph{Banach lattice}.

\paragraph{Monotone convergence.}
Recall that a \emph{directed set} $(A, \prec)$ is a set $A$ endowed with a partial order $\prec$ such that for every $a,a' \in A$, there exists $b \in A$ with $a \prec b$ and $a' \prec b$.
If $I$ is a countable set, the collection of all finite subsets of $I$ endowed with the inclusion is an example of directed set.
A \emph{net} is a map $f \colon A \to E$ from a directed set $(A,\prec)$ to $(E, \prec)$.
Such a net
\begin{itemize}
	\item is \emph{non-decreasing} if $f(a) \prec f(a')$ whenever $a \prec a'$;
	\item is \emph{norm-bounded} if there exits $M \in \R_+$ such that for every $a \in A$, we have $\norm{f(a)} \leq M$;
	\item \emph{converges to} $b \in E$ if for every $\varepsilon \in \R_+^*$, there exists $a_0 \in A$, such that for every $a \in A$, with $a_0 \prec a$, we have $\norm{f(a) - b} \leq \varepsilon$.
	      In this case we write $b = \lim f$.
\end{itemize}

\begin{prop}[Schaefer {\cite[Chapter~II, Theorem~5.11]{Schaefer:1974vz}}]
	\label{res: app - conv bounded net}
	Assume that $E$ is a reflexive Banach lattice.
	Then $E$ is order complete.
	Moreover, every non-decreasing norm-bounded net $f \colon A \to E$ converges.
\end{prop}

\paragraph{Operator between lattices.}
Let $E$ and $F$ be two vector lattices.
A linear operator $U \in \mathcal L(E,F)$ is \emph{positive} if it maps $E^+$ into $F^+$.
This defines a partial order on $\mathcal L(E,F)$: given $U_1,U_2 \in \mathcal L(E,F)$ we say that $U_1 \prec U_2$ if $U_2 - U_1$ is positive.
However $\mathcal L(E,F)$ endowed with the order is in general not a vector lattice.
To bypass this difficult, we
 consider a smaller subspace of $\mathcal L(E,F)$.
A  linear operator $U \colon E \to F$ is \emph{regular} if is can be written as $U = U_+ - U_-$ where $U_+$ and $U_-$ are two positive linear operators from $E$ to $F$.
The set of all regular operators from $E$ to $F$, that we denote by $\mathcal L_r(E,F)$, is a vector subspace of $\mathcal L(E,F)$.

\begin{prop}[{Schaefer \cite[Chapter~IV, Propositions~1.3]{Schaefer:1974vz}}]
	\label{res: app - regular op lattice}
	If $E$ and $F$ are two vector lattices and $F$ is order complete, then $\mathcal L_r(E,F)$ is an order complete vector lattice.
\end{prop}

Suppose now that $E$ and $F$ are two Banach lattices and $F$ is order complete.
We write $\mathcal B_r(E,F)$ for the set of \emph{bounded regular operators}, i.e. the elements $U \in \mathcal L_r(E,F)$ such that $\abs U$ is a bounded operator.
This space is endowed with an \emph{regular norm} defined by $\norm[r] U = \norm{ \abs U }$ which turn $\mathcal B_r(E,F)$ into a Banach lattice \cite[Chapter~IV, Propositions~1.4]{Schaefer:1974vz}.
Note that both norms $\normV[r]$ and $\normV$ coincide on positive operators.

\medskip
Although $\mathcal B_r(E,F)$ is Banach lattice, we cannot expect as in \autoref{res: app - conv bounded net} that every non-decreasing norm-bounded net of regular operator converges for the norm $\normV[r]$.
However for our purpose,   pointwise convergence will be enough.

\begin{prop}
	\label{res: app - converging operator net}
	Assume that $E$ and $F$ are two Banach lattices and $F$ is reflexive.
	Let $f \colon A \to \mathcal B_r(E,F)$ be a non-decreasing norm-bounded net.
	For every $\phi \in E$, the net $f_\phi \colon A \to E$ mapping $a$ to $f(a)\phi$ converges.
	Moreover the map $V \colon E \to F$ defined by $V\phi = \lim f_\phi$	is a bounded regular operator.
\end{prop}

\paragraph{Remarks.}
It $f(a)$ is positive, for every $a \in A$, one easily  checks that
\begin{equation*}
	\norm V = \sup_{a \in A}\norm{f(a)}.
\end{equation*}

\begin{proof}
	Let $\phi \in E$.
	We write $\phi_+$ and $\phi_-$ for it positive and negative part respectively.
	Observe that the nets $f_{\phi_+}$ and $f_{\phi_-}$ are non-decreasing and norm-bounded, hence they converges (\autoref{res: app - conv bounded net}).
	Thus $f_\phi$ converges as well.
	One checks easily that the map $V \colon E \to F$ sending $\phi$ to $\lim f_\phi$ satisfies the announced properties.
\end{proof}

\begin{defi}
	\label{def: app - positive rep}
	Let $\Gamma$ be a group.
	We say that a unitary representation $\rho \colon \Gamma \to \mathcal B(E)$ is \emph{positive} if $\rho(\gamma)$ is positive for every $\gamma \in \Gamma$.
\end{defi}

\paragraph{Dual space.}
Suppose that $E$ is a Banach lattice.
Its (topological) dual space $E'$ endowed with the order inherited from $\mathcal L(E,\R)$ is an order complete Banach lattice \cite[Chapter~II, Proposition~5.5]{Schaefer:1974vz}.
Actually it is isomorphic to $\mathcal B_r(E,\R)$ \cite[Chapter~IV, Theorem~1.5]{Schaefer:1974vz}.
Recall that a subspace $F$ of $E'$ \emph{separates points} if for every distinct $\phi, \phi' \in E$, there exists $\lambda \in F$ such that $\lambda(\phi) \neq \lambda (\phi')$.

\begin{prop}[{\cite[Corollary~8.35]{Aliprantis:2006tpa}}]
	\label{res: app - positive vector by duality}
	Assume that $E$ is a Banach lattice.
	Let $F$ be an ideal of $E'$ which separates the points.
	A vector $\phi \in E$ belongs to $E$ if and only if for every $\lambda \in F$ such that $0 \prec \lambda$, we have $\lambda(\phi) \geq 0$.
\end{prop}

%
\subsection{Examples}
%

We review here the main examples of Banach lattices that are used in the article.

%
\subsubsection{Koopman representations}\label{exa : sq sum function}
%

In this article we are mostly interested with the following situation.
Let $Y$ be a set endowed with the counting measure.
The space $\mathcal H = \ell^2(Y)$ of square summable maps $\phi \colon Y \to \R$, endowed with the scalar product defined as
\begin{equation*}
	\sca {\phi_1}{\phi_2} = \sum_{y \in Y}\phi_1(y)\phi_2(y),
\end{equation*}
is a Hilbert space, hence a reflexive Banach space.
We endow this space with a partial order defined as follows.
Given $\phi,\phi' \in \mathcal H$ we say that $\phi \prec \phi'$ if $\phi(y) \leq \phi'(y)$ for every $y \in Y$.
It turns $\mathcal H$ into a Banach lattice.

\medskip
Let $\Gamma$ be a discrete group acting on $Y$.
This action induces a positive unitary representation $\rho \colon \Gamma \to \mathcal U(\mathcal H)$, called the \emph{Koopman representation}.

%
\subsubsection{Bochner spaces}
%

Let $(E, \prec, \normV)$ be a Banach lattice and $(X,\mathcal B, \nu)$ a be a finite measure space.
Let $p \in [1, \infty) \cup \{\infty\}$.
We define a binary relation on the Bochner space $L^p(\nu, E)$ as follows.
Given $\Phi, \Phi' \in L^p(\nu, E)$, we say that $\Phi \prec \Phi'$ if $\Phi(x) \prec \Phi'(x)$ $\nu$-almost everywhere.
It is obvious that this defines indeed a partial order on $L^p(\nu, E)$.

\begin{lemm}
	\label{res: app - bochner banach lattice}
	The Bochner space $L^p(\nu, E)$ endowed with $\prec$ is a Banach lattice.
\end{lemm}

\begin{proof}
	It is obvious that the order $\prec$ is compatible with the vector space structure on $L^p(\nu, E)$.
	Let $\Phi, \Phi' \in L^p(\nu, E)$.
	We define a map $\Psi \colon X \to E$ by $\Psi(x) = \Phi(x) \lor \Phi'(x)$, for all $x \in X$.
	We are going to prove that $\Psi$ is the least upper bound of $\Phi$ and $\Phi'$.
	Let us first prove that $\Psi$ is $\nu$-measurable and belongs to $L^p(\nu, E)$.
	By definition there exists two sequences $(\Phi_n)$ and $(\Phi'_n)$ of simple functions converging $\nu$-almost everywhere to $\Phi$ and $\Phi'$ respectively.
	One checks easily that the function $\Psi_n \colon X \to E$ sending $x$ to $\Phi_n(x) \lor \Phi'_n(x)$ is also a simple function.
	On the other hand the operation $\lor$ is uniformly continuous \cite[Chapter~II, Proposition~5.1]{Schaefer:1974vz}.
	It follows that $(\Psi_n)$ converges $\nu$-almost everywhere to $\Psi$, hence $\Psi$ is $\nu$-measurable.
	For every $x \in X$, we have
	\begin{equation*}
		\norm{\Phi(x) \lor \Phi'(x)} \leq \norm{\Phi(x)} +\norm{\Phi'(x)}.
	\end{equation*}
	See for instance \cite[Chapter~II, Proposition~1.4(6)]{Schaefer:1974vz}.
	Since $\Phi$ and $\Phi'$ belongs to $L^p(\nu, E)$ so does $\Psi$.
	It is now obvious to check that $\Psi$ is the least upper bound of $\Phi$ and $\Phi'$.
	We check in the same manner that the greatest lower bound of $\Phi$ and $\Phi'$ is the function $X \to E$ sending $x$ to $\Phi(x) \land \Phi'(x)$.
	Thus $L^p(\nu,E)$ is a vector lattice.
	Let us prove now that the norm is monotone.
	Let $\Phi, \Phi' \in L^p(\nu,E)$ such that $\abs{\Phi} \prec \abs{\Phi'}$.
	It follows from the previous discussion that $\abs \Phi \colon X \to E$ is exactly the function sending $x$ to $\abs{\Phi(x)}$.	The same holds for $\Phi'$, hence $\abs{\Phi(x)}\prec \abs{\Phi'(x)}$ $\nu$-almost surely.
	Since the norm of $E$ is monotone, $\norm{\Phi(x)} \leq \norm{\Phi'(x)}$ $\nu$-almost surely, hence $\norm[p] \Phi \leq \norm[p]{\Phi'}$.
	Consequently $L^p(\nu, E)$ is a Banach lattice.
\end{proof}

\begin{lemm}
	\label{res: app - bochner banach lattice - completeness}
	If $E$ is countably order complete, then  $L^p(\nu, E)$ is also .
\end{lemm}

\begin{proof}
	Let $A$ be a non-empty countable subset of $L^p(\nu, E)$ which is bounded from above.
	Up to translating $A$, we can always assume that $0$ belongs to $A$.
	Let $\Psi \in L^p(\nu, E)$ be an upper bound of $A$.
	Since $A$ is countable, there exists a subset $B$ of $X$ with $\nu(B) = 0$ such that $\Phi(x) \prec \Psi(x)$ for every $x \in X \setminus B$, for every $\Phi \in A$.
	In particular, for every $x \in X \setminus B$ the set $\set{\Phi(x)}{\Phi \in A}$ is non-empty subset of $E$ containing $0$ and which is bounded from above.
	As $E$ is countably order complete we can define a function $\Phi_M \colon X \to E^+$ by letting
	\begin{equation*}
		\Phi_M(x) = \sup_{\Phi \in A} \Phi(x),\quad \forall x \in X \setminus B.
	\end{equation*}
	By construction $0 \prec \Phi_M(x) \prec \Psi(x)$ $\nu$-almost everywhere.
	Hence $0 \prec \Phi_M \prec \Psi$ and $\norm[p]{\Phi_M} \leq \norm[p] \Psi$.
	In particular $\Phi_M$ belongs to $L^p(\nu, E)$.
	One checks easily that $\Phi_M$ is the least upper bound of $A$.
\end{proof}

\paragraph{Positivity of the Bochner integral.}
We now focus on the case where $p= 1$ and study the behaviour of the Bochner integral with respect the partial order on $L^1(\nu,E)$.

\begin{lemm}[Positivity]
	\label{res: app - bochner integral positive}
	Let $\Phi, \Phi' \in L^1(\nu, E)$.
	If $\Phi \prec \Phi'$, then
	\begin{equation*}
		\int \Phi d\nu \prec \int \Phi' d\nu.
	\end{equation*}
\end{lemm}

\begin{proof}
	Since the Bochner integral is linear it suffices to prove that the
	\begin{equation*}
		0 \prec \int \Phi d\nu.
	\end{equation*}
	whenever $0 \prec \Phi$.
	Note that the statement is obvious if $\Phi$ is a simple function.
	Hence we are left to prove that every positive function $\Phi$ is the limit of a sequence $(\Phi_n)$ of \emph{positive} simple functions.
	Let $\Phi\in L^1(\nu, E)$ be such a positive function.
	There exists a sequence $(\Phi_n)$ of simple function converging to $\Phi$ in $L^1(\nu, E)$.
	One checks that $(\Phi_n \lor 0)$ is a sequence of positive simple functions.
	As $L^1(\nu, E)$ is a Banach lattice, the operation $\lor$ on $L^1(\nu, E)$ is uniformly continuous, hence $(\Phi_n \lor 0)$ converges to $\Phi \lor 0$, i.e $\Phi$.
\end{proof}

\begin{prop}
	\label{res: app - bochner integral positive reverse}
	Let $E$ be a Banach lattice.
	Let $\Phi \in L^1(\nu, E)$.
	If for every $B \in \mathcal B$, we have
	\begin{equation*}
		0 \prec \int \mathbf 1_B\Phi d\nu,
	\end{equation*}
	then $0 \prec \Phi$.
\end{prop}

\begin{proof}
	Let $E'$ be the dual of $E$.
	We consider the bilinear map
	\begin{equation*}
		\begin{array}{ccc}
			L^\infty(\nu, E') \times L^1(\nu, E) & \to & \R                                                         \\
			(\Lambda, \Phi)                      & \to & \displaystyle\int \Lambda(x) \left[\Phi(x) \right] d\nu(x)
		\end{array}
	\end{equation*}
	that we denote by $\sca \Lambda\Phi$.
	This duality product induces an isometric embedding from $L^\infty(\nu, E')$ into the dual $D$ of $L^1(\nu, E)$ \cite[Chapter~IV, §1]{Diestel:1977ui}.
	Moreover, seen as an subspace of $D$, the space $L^\infty(\nu, E')$ is an ideal that separates the points.

	Let $\lambda \in E'$ such that $0 \prec \lambda$ and $B$ be a Borel subspace of $X$.
	It follows from our assumption and \autoref{res: app - bochner integral commutes with linear map} that the quantity
	\begin{equation*}
		\sca{\mathbf 1_B \lambda}{\Phi}= \int \mathbf 1_B \lambda \circ \Phi d\nu = \lambda \left(\int \mathbf 1_B \Phi d\nu \right)
	\end{equation*}
	is non negative.
	By linearity, for every positive simple function $\Lambda \in L^\infty(\nu, E')$, we have $\sca \Lambda\Phi \geq 0$.
	Let $\Lambda \in L^\infty(\nu, E')$ be an arbitrary positive function.
	By definition of $\nu$-measurability, there exists a sequence $(\Lambda_n)$ of simple functions of $L^\infty(\nu, E')$ which converge to $\Lambda$ $\nu$-almost everywhere.
	Up to replacing $\Lambda_n$ by $\Lambda_n \lor 0$ we can assume that each $\Lambda_n$ is positive.
	According to the dominated convergence theorem (for Lebesgue integrals) $\sca{\Lambda_n}{\Phi}$ converges to $\sca{\Lambda}{\Phi}$ which is thus non-negative.
	It follows then from \autoref{res: app - positive vector by duality} that $0 \prec \Phi$.
\end{proof}

\providecommand{\bysame}{\leavevmode\hbox to3em{\hrulefill}\thinspace}

\noindent
\emph{R\'emi Coulon} \\
Univ Rennes, CNRS \\
IRMAR - UMR 6625 \\
F-35000 Rennes, France\\
\texttt{remi.coulon@univ-rennes1.fr} \\
\texttt{http://rcoulon.perso.math.cnrs.fr}

\medskip

\noindent
\emph{Rhiannon Dougall} \\
University of Bristol \\
\texttt{ r.dougall@bristol.ac.uk}

\medskip

\noindent
\emph{Barbara Schapira} \\
Univ Rennes \\
IRMAR - UMR 6625 \\
F-35000 Rennes, France\\
\texttt{barbara.schapira@univ-rennes1.fr} \\
\texttt{https://perso.univ-rennes1.fr/barbara.schapira}

\medskip

\noindent
\emph{Samuel Tapie} \\
Universit\'e de Nantes \\
Laboratoire Jean Leray  - UMR 6629 \\
F-44322 Nantes, France\\
\texttt{samuel.tapie@univ-nantes.fr} \\
\texttt{http://www.math.sciences.univ-nantes.fr/~tapie}

\end{document}